\documentclass{amsart}

\usepackage{amssymb}
\usepackage{amsfonts}
\usepackage{amsmath}
\usepackage{amsthm}
\usepackage{graphicx}
\usepackage{mathrsfs}
\usepackage{dsfont}
\usepackage{amscd}
\usepackage{multirow}
\usepackage{palatino}
\usepackage{mathpazo}
\usepackage[all]{xy}
\usepackage[scaled]{helvet} 
\usepackage{courier} 
\normalfont
\usepackage[T1]{fontenc}
\usepackage{calligra}
\usepackage{verbatim}
\usepackage[usenames,dvipsnames]{color}
\usepackage[colorlinks=true,linkcolor=Blue,citecolor=Violet]{hyperref}

\makeatletter
\@namedef{subjclassname@2010}{%
  \textup{2010} Mathematics Subject Classification}
\makeatother

\theoremstyle{plain}
\newtheorem{theorem}{Theorem}[section]

\newtheorem{claim}[theorem]{Claim}

\newtheorem{case}{Case}

\newtheorem{corollary}[theorem]{Corollary}

\newtheorem{lemma}[theorem]{Lemma}
\newtheorem{proposition}[theorem]{Proposition}

\newtheorem{theorem-definition}[theorem]{Theorem-Definition}

\theoremstyle{definition}
\newtheorem{definition}[theorem]{Definition}

\newtheorem{notation}[theorem]{Notation}

\newtheorem{convention}[theorem]{Convention}

\theoremstyle{remark}

\newtheorem{example}[theorem]{Example}

\newtheorem{remark}[theorem]{Remark}

\numberwithin{equation}{section}

\newcommand{\N}{{\mathds{N}}}

\newcommand{\Q}{{\mathds{Q}}}
\newcommand{\R}{{\mathds{R}}}
\newcommand{\C}{{\mathds{C}}}

\newcommand{\A}{{\mathfrak{A}}}
\newcommand{\B}{{\mathfrak{B}}}
\newcommand{\M}{{\mathfrak{M}}}
\newcommand{\F}{{\mathfrak{F}}}

\newcommand{\Lip}{{\mathsf{L}}}

\newcommand{\qpropinquity}[1]{{\mathsf{\Lambda}_{#1}}}

\newcommand{\Kantorovich}[1]{{\mathsf{mk}_{#1}}}

\newcommand{\StateSpace}{{\mathscr{S}}}

\newcommand{\mongekant}{{Mon\-ge-Kan\-to\-ro\-vich metric}}

\newcommand{\Lqcms}{{\JLL} quantum compact metric space}
\newcommand{\Qqcms}[1]{${#1}$-quasi-Leibniz quantum compact metric space}
\newcommand{\gQqcms}{quasi-Leibniz quantum compact metric space}

\newcommand{\unit}{1}

\newcommand{\sa}[1]{{\mathfrak{sa}\left({#1}\right)}}

\newcommand{\JLL}{Lei\-bniz}

\newcommand{\dom}[1]{{\operatorname*{dom}({#1})}}
\newcommand{\diam}[2]{{\mathrm{diam}\left({#1},{#2}\right)}}

\newcommand{\bridgelength}[2]{{\lambda\left({#1}\middle|{#2}\right)}}

\newcommand{\alg}[1]{{\mathfrak{#1}}}

\newcommand{\BaireSpace}{{\mathscr{N}}}

\newcommand{\indmor}[2]{{\underrightarrow{#1^{#2}}}}
\newcommand{\af}[1]{{\alg{AF}_{#1}}}

\renewcommand{\geq}{\geqslant}
\renewcommand{\leq}{\leqslant}

\newcommand{\CondExp}[2]{{\mathds{E}\left({#1}\middle\vert{#2}\right)}}

\newcommand{\Latremoliere}{Latr\'{e}moli\`{e}re}

\allowdisplaybreaks[4]

\makeatletter
\newcommand{\vast}{\bBigg@{4}}
\newcommand{\Vast}{\bBigg@{5}}
\makeatother

\begin{document}

\baselineskip=14pt

\title[Convergence of Quotients of AF Algebras]{Convergence of Quotients of AF Algebras in Quantum Propinquity by Convergence of Ideals}
\author{Konrad Aguilar}
\address{School of Mathematical and Statistical Sciences \\ Arizona State University \\  901 S. Palm Walk, Tempe, AZ 85287-1804}
\email{konrad.aguilar@asu.edu}
\urladdr{}

\date{\today}
\thanks{The  author  was partially supported by the Simons - Foundation grant 346300 and the Polish Government MNiSW 2015-2019 matching fund, and this work is part of the project supported by the grant H2020-MSCA-RISE-2015-691246-QUANTUM DYNAMICS and the Polish Government grant 3542/H2020/2016/2}
\thanks{Some of this work was completed during the Simons semester hosted at IMPAN during September-December 2016 titled ``Noncommutative Geometry - The Next Generation''.}

\begin{abstract}
We introduce a topology on the ideal space of inductive limits of C*-algebras built by a topological  inverse limit of the Fell topologies on the C*-algebras of the given inductive sequence and we produce conditions for when this topology agrees with the Fell topology  of the inductive limit.  With this topology,  we impart criteria for when convergence of ideals of an AF algebra can provide convergence of quotients in the quantum Gromov-Hausdorff propinquity building off previous joint work with \Latremoliere. These findings bestow a continuous map from a  class of ideals of the Boca-Mundici AF algebra equipped with various topologies including Jacobson and Fell topologies to the space of quotients equipped with the propinquity topology. 
\end{abstract}
\maketitle

\setcounter{tocdepth}{2}
\tableofcontents

\allowdisplaybreaks[2]


\section{Introduction}

\Latremoliere's quantum Gromov-Hausdorff propinquity \cite{Latremoliere13,Latremoliere13c,Latremoliere15} provides a powerful tool for studying and constructing new continuous families of compact quantum metric spaces of  Rieffel \cite{Rieffel98a, Rieffel99} as seen in  \cite{Latremoliere05,Latremoliere13c,Latremoliere15c} and \cite{Rieffel01,Rieffel10c,Rieffel15}. Compact quantum metric spaces introduced Rieffel \cite{Rieffel98a} and motivated by A. Connes \cite{Connes89, Connes} are unital C*-algebras equipped with certain metrics on states built from noncommutative analogues of the Lipschitz seminorm associated to continuous functions on metric spaces.  A key contribution that quantum propinquity produces in the study of Noncommutative Metric Geometry is that it forms a distance on certain classes of compact quantum metric spaces that preserves the C*-algebraic structure as well as the metric structure, while staying within the category of C*-algebras \cite{Latremoliere13}. Now, as ideals (norm closed and two-sided) of a fixed C*-algebra are C*-algebras themselves and there exist topologies on ideals where ideals are viewed as points (for instance, the Jacobson and Fell topologies), it is then natural to desire to compare these topologies with the quantum propinquity topology.  However, in general, ideals need not be unital C*-algebras, thus a direct comparison of these topologies on ideals with quantum propinquity can not be accomplished as the quantum propinquity is only suitable for unital C*-algebras.  Yet, given a unital C*-algebra, quotients by non-trivial ideals are unital C*-algebras. Hence, a main consequence of this paper will be to list sufficient conditions for when convergence of ideals in certain topologies provide convergence of their quotients in the quantum propinquity topology.   Thus, it is with continuity that we will establish a nontrivial connection between topologies on ideals and the topology formed by quantum propinquity.  Therefore, this paper claims to advance both the study of topologies on ideals and noncommutative metric geometry by way of the quantum propinquity topology.  

Concerning quotients, the class of C*-algebras that we focus on is the class of unital AF algebras of Bratteli \cite{Bratteli72}, and in particular, unital AF algebras with faithful tracial states.  Our work with \Latremoliere \ in \cite{AL} already provided the quantum metrics for these particular AF algebras that we will use, which will allow us to focus on continuity aspects in this paper. After a background section, in Section (\ref{metric-ideal-IL}), we develop a topology on the ideal space of any C*-inductive limit.    The main application of this topology is to provide a notion of convergence for inductive sequences that determine the quotient spaces as fusing families (Definition (\ref{coll-def}))- a notion introduced in \cite{Aguilar16a} to provide sufficient conditions for convergence in quantum propinquity of AF algebras.  But, this topology on ideals has close connections to the Fell topology on the ideal space formed by the Jacobson topology on the primitive ideal space.  The Fell topology was introduced by Fell in \cite{Fell62} as a topology on closed sets of a given topology.  Fell then applied this topology to the closed sets of the Jacobson topology in \cite{Fell61} to provide a compact Hausdorff topology on the set of all ideals of a C*-algebra.  The topology on the ideal space of C*-inductive limits introduced in this paper is always stronger than the Fell topology, and we provide conditions for when this topology agrees with the Fell topology by way of conditions on the algbraic and analyticial properties on the types of ideals themselves.  In particular, our topology will agree with the Fell topology for any AF algebra, unital or not, which case we provide an explicit metric that metrizes thie topology.  We make other comparisons including taking into consideration the restriction to primitive ideals and comparison of the Jacobson topology as well as an analysis on unital commutative AF algebras and unital C*-algebras with Hausdorff Jacobson topology.

Next, Section (\ref{conv-quotients-quantum}) provides an answer to the question of when convergence of ideals can provide convergence of quotients.   In Section (\ref{b-m-af}),  we define the Boca-Mundici AF algebra  given in \cite{Boca08, Mundici88}, which arises from the Farey tessellation.  Next, we prove some basic results pertaining to its Bratteli diagram structure and ideal structure, and then apply our criteria for quotients converging to a subclass of ideals of the Boca-Mundici AF algebra, in which each quotient is *-isomorphic to an Effros-Shen AF algebra. In \cite{Boca08}, Boca proved that this subclass of ideals with its relative Jacobson topology is homeomorphic to the irrationals in $(0,1)$ with its usual topology, which provided our initial interest in our question about convergence of quotients.    The main result of this section, Theorem (\ref{main-result}), produces a continuous function from a subclass of ideals of the Boca-Mundici AF algebra to its quotients as quantum metric spaces in the quantum propinquity topology, where the topology of the subclass ideals is homeomorphic to both the Jacobson and Fell topologies and thus the topology introduced in this paper as well. Hence, we have an explicit example of when a metric geometry on quotients is related to a metric geometry on ideals by a continuous map.

\section{Quantum Metric Geometry and AF algebras}
The purpose of this section is to discuss our progress thus far in the realm of quantum metric spaces with regard to AF algebras, and thus places more focus on the AF algebra results, but we also provide a cursory overview of the material on quantum compact metric spaces.  We refer the reader to the survey by \Latremoliere \ \cite{Latremoliere15b} for a much more detailed and insightful introduction to the study of quantum metric spaces.

\begin{notation}
When $E$ is a normed vector space, then its norm will be denoted by $\|\cdot\|_E$ by default. Let $\A$ be a unital C*-algebra. The unit of $\A$ will be denoted by $\unit_\A$. The state space of $\A$ will be denoted by $\StateSpace(\A)$ and the self-adjoint part of $\A$ will be denoted $\sa{\A}$. 
\end{notation}

\begin{definition}[\cite{Rieffel98a, Latremoliere13, Latremoliere15}]\label{quasi-Monge-Kantorovich-def}
A {\Qqcms{(C,D)}} $(\A,\Lip)$, for some $C\geq 1$ and $D\geq 0$, is an ordered pair where $\A$ is unital C*-algebra and $\Lip$ is a seminorm defined on $\sa{\A}$ such that $\dom{\Lip}=\{a \in \sa{\A}: \Lip(a) < \infty\}$ is a dense Jordan-Lie subalgebra $\dom{\Lip}$ of $\sa{\A}$ such that:
\begin{enumerate}
\item $\{ a \in \sa{\A} : \Lip(a) = 0 \} = \R\unit_\A$,
\item the seminorm $\Lip$ is a \emph{$(C,D)$-quasi-Leibniz Lip-norm}, i.e. for all $a,b \in \dom{\Lip}$:
\begin{equation*}
\max\left\{ \Lip\left(\frac{ab+ba}{2}\right), \Lip\left(\frac{ab-ba}{2i}\right) \right\} \leq C\left(\|a\|_\A \Lip(b) + \|b\|_\A\Lip(a)\right) + D \Lip(a)\Lip(b)\text{,}
\end{equation*}
\item the \emph{\mongekant} defined, for all two states $\varphi, \psi \in \StateSpace(\A)$, by:
\begin{equation*}
\Kantorovich{\Lip} (\varphi, \psi) = \sup\left\{ |\varphi(a) - \psi(a)| : a\in\dom{\Lip}, \Lip(a) \leq 1 \right\}
\end{equation*}
metrizes the weak* topology of $\StateSpace(\A)$,
\item the seminorm $\Lip$ is lower semi-continuous with respect to $\|\cdot\|_\A$.
\end{enumerate}

\end{definition}

A primary interest in developing a theory of quantum metric spaces is the introduction of various hypertopologies on classes of such spaces, thus allowing us to study the geometry of classes of C*-algebras and perform analysis on these classes as we perform in this current article as well. A classical model for these hypertopologies is given by the Gromov-Hausdorff distance \cite{Gromov81, Gromov}. While several noncommutative analogues of the Gromov-Hausdorff distance have been proposed --- most importantly Rieffel's original construction of the quantum Gromov-Hausdorff distance \cite{Rieffel00} --- we shall work with a particular metric introduced by \Latremoliere, \cite{Latremoliere13}, as we did in \cite{AL}. This metric, known as the quantum propinquity, is designed to be best suited to {\gQqcms s}, and in particular, is zero between two such spaces if and only if they are quantum isometric, which is defined in the following theorem, and is, in part, a *-isomorphism between the C*-algebras.

\begin{theorem-definition}[\cite{Latremoliere13, Latremoliere15}]\label{def-thm}
Fix $C\geq 1$ and $D \geq 0$. Let $\mathcal{QQCMS}_{C,D}$ be the class of all {\Qqcms{(C,D)}s}. There exists a class function $\qpropinquity{C,D}$ from $\mathcal{QQCMS}_{C,D}\times\mathcal{QQCMS}_{C,D}$ to $[0,\infty) \subseteq \R$ such that:
\begin{enumerate}
\item for any $(\A,\Lip_\A), (\B,\Lip_\B) \in \mathcal{QQCMS}_{C,D}$ we have:
\begin{equation*}
 \qpropinquity{C,D}((\A,\Lip_\A),(\B,\Lip_\B)) \leq \max\left\{\diam{\StateSpace(\A)}{\Kantorovich{\Lip_\A}}, \diam{\StateSpace(\B)}{\Kantorovich{\Lip_\B}}\right\}\text{,}
\end{equation*}
\item for any $(\A,\Lip_\A), (\B,\Lip_\B) \in \mathcal{QQCMS}_{C,D}$ we have:
\begin{equation*}
0\leq \qpropinquity{C,D}((\A,\Lip_\A),(\B,\Lip_\B)) = \qpropinquity{C,D}((\B,\Lip_\B),(\A,\Lip_\A))
\end{equation*}
\item for any $(\A,\Lip_\A), (\B,\Lip_\B), (\alg{C},\Lip_{\alg{C}}) \in \mathcal{QQCMS}_{C,D}$ we have:
\begin{equation*}
\qpropinquity{C,D}((\A,\Lip_\A),(\alg{C},\Lip_{\alg{C}})) \leq \qpropinquity{C,D}((\A,\Lip_\A),(\B,\Lip_\B)) + \qpropinquity{C,D}((\B,\Lip_\B),(\alg{C},\Lip_{\alg{C}}))\text{,}
\end{equation*}
\item for all  for any $(\A,\Lip_\A), (\B,\Lip_\B) \in \mathcal{QQCMS}_{C,D}$ and for any bridge $\gamma$ from $\A$ to $\B$ defined in \cite[Definition 3.6]{Latremoliere13}, we have:
\begin{equation*}
\qpropinquity{C,D}((\A,\Lip_\A), (\B,\Lip_\B)) \leq \bridgelength{\gamma}{\Lip_\A,\Lip_\B}\text{,}
\end{equation*}
where $\bridgelength{\gamma}{\Lip_\A,\Lip_\B}$ is defined in \cite[Definition 3.17]{Latremoliere13},
\item for any $(\A,\Lip_\A), (\B,\Lip_\B) \in \mathcal{QQCMS}_{C,D}$, we have:
\begin{equation*}
\qpropinquity{C,D}((\A,\Lip_\A),(\B,\Lip_\B)) = 0
\end{equation*}
if and only if $(\A,\Lip_\A)$ and $(\B,\Lip_\B)$ are quantum isometric, i.e. if and only if there exists a *-isomorphism $\pi : \A \rightarrow\B$ with $\Lip_\B\circ\pi = \Lip_\A$,
\item if $\Xi$ is a class function from $\mathcal{QQCMS}_{C,D}\times \mathcal{QQCMS}_{C,D}$ to $[0,\infty)$ which satisfies Properties (2), (3) and (4) above, then:
\begin{equation*}
\Xi((\A,\Lip_\A), (\B,\Lip_\B)) \leq \qpropinquity{C,D}((\A,\Lip_\A),(\B,\Lip_\B))
\end{equation*}
 for all $(\A,\Lip_\A)$ and $(\B,\Lip_\B)$ in $\mathcal{QQCMS}_{C,D}$
\end{enumerate}
\end{theorem-definition}
 The quantum propinquity is, in fact, a special form of the dual Gromov-Hausdorff propinquity \cite{Latremoliere13b, Latremoliere14, Latremoliere15} also introduced by \Latremoliere, which is a complete metric, up to quantum isometry, on the class of {\Lqcms s}, and which extends the topology of the Gromov-Hausdorff distance as well. Thus, as the dual propinquity is dominated by the quantum propinquity \cite{Latremoliere13b}, we conclude that \emph{all the convergence results in this paper are valid for  dual Gromov-Hausdorff propinquity as well.}

In this paper, all our quantum metrics will be {\Qqcms{(2,0)}s}. Thus, we will simplify our notation as follows:
\begin{convention}
In this paper, $\qpropinquity{}$ will be meant for $\qpropinquity{2,0}$.
\end{convention}

Now, we provide some results from \cite{AL, Aguilar16a}.
For our work in AF algebras, it turns out that our Lip-norms are $(2,0)$-quasi-Leibniz Lip-norms.   The following Theorem (\ref{AF-lip-norms-thm}) is \cite[Theorem 3.5]{AL}. But, first some notation.

\begin{notation}\label{ind-lim}
Let $\mathcal{I} = (\A_n,\alpha_n)_{n\in\N}$ be an inductive sequence, in which $\A_n$ is a C*-algebra and $\alpha_n: \A_n \rightarrow \A_{n+1}$ is a *-homomorphism for all $n \in \N$, with limit $\A=\varinjlim \mathcal{I}$. We denote the canonical *-homomorphisms $\A_n \rightarrow\A$ by $\indmor{\alpha}{n}$ for all $n\in\N$, (see {\cite[Chapter 6.1]{Murphy}}).
\end{notation}

We display our Lip-norms built from faithful tracial states in both the inductive limit and closure of union context since both will be utilized throughout the paper. 
 We saw in \cite[Proposition 4.7]{Aguilar16a} that these two definitions are compatible.  These Lip-norms were motivated by  the work of C. Antonescu and E. Christensen in \cite{Antonescu04}.
\begin{theorem}[{\cite[Theorem 3.5]{AL}}]\label{AF-lip-norms-thm}
Let $\A$ be a unital AF algebra endowed with a faithful tracial state $\mu$. Let $\mathcal{I} = (\A_n,\alpha_n)_{n\in\N}$ be an inductive sequence of finite dimensional C*-algebras with C*-inductive limit $\A$, with $\A_0 \cong \C$ and where $\alpha_n$ is unital and injective for all $n\in\N$.

Let $\pi$ be the GNS representation of $\A$ constructed from $\mu$ on the space $L^2(\A,\mu)$.

For all $n\in\N$, let:
\begin{equation*}
\CondExp{\cdot}{\indmor{\alpha}{n}(\A_n)} : \A\rightarrow\A
\end{equation*}
be the unique conditional expectation of $\A$ onto the canonical image $\indmor{\alpha}{n}\left(\A_n\right)$ of $\A_n$ in $\A$, and such that $\mu\circ\CondExp{\cdot}{\indmor{\alpha}{n}(\A_n)} = \mu$.

Let $\beta: \N\rightarrow (0,\infty)$ have limit $0$ at infinity. If, for all $a\in\sa{\cup_{n \in \N} \indmor{\alpha}{n}\left(\A_n\right)}$, we set:
\begin{equation*}
\Lip_{\mathcal{I},\mu}^\beta(a) = \sup\left\{\frac{\left\|a - \CondExp{a}{\indmor{\alpha}{n}(\A_n)}\right\|_\A}{\beta(n)} : n \in \N \right\},
\end{equation*}
then $\left(\A,\Lip_{\mathcal{I},\mu}^\beta\right)$ is a {\Qqcms{2}}. Moreover, for all $n\in\N$:
\begin{equation*}
\qpropinquity{}\left(\left(\A_n,\Lip_{\mathcal{I},\mu}^\beta\circ\indmor{\alpha}{n} \right), \left(\A,\Lip_{\mathcal{I},\mu}^\beta \right)\right) \leq \beta(n)
\end{equation*}
and thus:
\begin{equation*}
\lim_{n\rightarrow\infty} \qpropinquity{}\left(\left(\A_n,\Lip_{\mathcal{I},\mu}^\beta\circ\indmor{\alpha}{n}\right), \left(\A,\Lip_{\mathcal{I},\mu}^\beta\right)\right) = 0\text{.}
\end{equation*}
\end{theorem}

\begin{theorem}[{\cite[Theorem 5.1]{Aguilar16a}}]\label{AF-lip-norms-thm-union} 
Let $\A$ be a unital AF algebra  with unit $1_\A$ endowed with a faithful tracial state $\mu$. Let $\mathcal{U} = (\A_n)_{n\in\N}$ be an increasing sequence of unital finite dimensional C*-subalgebras such that $\A=\overline{\cup_{n \in \N} \A_n}^{\Vert \cdot \Vert_\A}$ with $\A_0=\C 1_\A $.

Let $\pi$ be the GNS representation of $\A$ constructed from $\mu$ on the space $L^2(\A,\mu)$.

For all $n\in\N$, let:
\begin{equation*}
\CondExp{\cdot}{\A_n} : \A\rightarrow\A
\end{equation*}
be the unique conditional expectation of $\A$ onto $\A_n$ , and such that $\mu\circ\CondExp{\cdot}{\A_n} = \mu$.

Let $\beta: \N\rightarrow (0,\infty)$ have limit $0$ at infinity. If, for all $a\in\sa{\cup_{n \in \N} \A_n }$, we set:
\begin{equation*}
\Lip_{\mathcal{U},\mu}^\beta(a) = \sup\left\{\frac{\left\|a - \CondExp{a}{\A_n}\right\|_\A}{\beta(n)} : n \in \N \right\},
\end{equation*}
then $\left(\A,\Lip_{\mathcal{U},\mu}^\beta\right)$ is a {\Qqcms{2}}. Moreover, for all $n\in\N$:
\begin{equation*}
\qpropinquity{}\left(\left(\A_n,\Lip_{\mathcal{U},\mu}^\beta \right), \left(\A,\Lip_{\mathcal{U},\mu}^\beta \right)\right) \leq \beta(n)
\end{equation*}
and thus:
\begin{equation*}
\lim_{n\rightarrow\infty} \qpropinquity{}\left(\left(\A_n,\Lip_{\mathcal{U},\mu}^\beta \right), \left(\A,\Lip_{\mathcal{U},\mu}^\beta\right)\right) = 0\text{.}
\end{equation*}
\end{theorem}

In \cite{AL}, the fact that the defining finite-dimensional subalgebras provide explicit approximations of the inductive limit with respect to the quantum Gromov-Hausdorff propinquity allowed us to prove that both the UHF algebras and the Effros-Shen AF algebras are continuous images of the Baire space with respect to the quantum propinquity. Our pursuit was motivated by the fact that the Effros-Shen algebras were used by Pimsner and Voiculescu to classify the irrational rotation algebras \cite{PimVoi80a} and \Latremoliere \ showed continuity of the irrational rotation algebras in propinquity with respect to their irrational parameters in \cite{Latremoliere13c}. We list the Effros-Shen algebra result here since we will utilize both the definition of the Effros-Shen algebras extensively as well as the continuity result in Section (\ref{b-m-af}).

We begin by recalling the construction of the AF C*-algebras $\alg{AF}_\theta$ constructed in \cite{Effros80b} for any irrational $\theta$ in $(0,1)$. For any $\theta \in (0,1)\setminus\Q$, let $(a_j)_{j\in\N}$ be the unique sequence in $\N$ such that:
\begin{equation}\label{continued-fraction-eq}
\theta =\lim_{n \to \infty} [a_0, a_1, \ldots,a_n]\text{,}
\end{equation}
where $[a_0, a_1, \ldots,a_n]$ denotes the standard continued fraction.
The sequence $(a_j)_{j\in\N}$ is called the continued fraction expansion of $\theta$, and we will simply denote it by writing $\theta = [a_0 , a_1 , a_2, \ldots ] = [a_j]_{j\in\N}$. We note that $a_0 = 0$ (since $\theta\in(0,1)$) and $a_n \in \N\setminus\{0\}$ for $n \geq 1 $.

We fix $\theta \in (0,1)\setminus\Q$, and let $\theta = [a_j]_{j\in\N}$ be its continued fraction decomposition. We then obtain a sequence $\left(\frac{p_n^\theta}{q_n^\theta}\right)_{n\in\N}$ with $p_n^\theta \in \N$ and $q_n^\theta \in \N\setminus\{0\}$ by setting:
\begin{equation}\label{pq-rel-eq}
\begin{cases}
\begin{pmatrix}p_1^\theta & q_1^\theta \\ p_0^\theta & q_0^\theta \end{pmatrix} = \begin{pmatrix}a_0a_1+1 & a_1 \\ a_0 & 1 \end{pmatrix}\\
\begin{pmatrix}p_{n+1}^\theta & q_{n+1}^\theta \\ p_n^\theta & q_n^\theta \end{pmatrix} = \begin{pmatrix}a_{n+1} & 1 \\ 1 & 0 \end{pmatrix}\begin{pmatrix}p_n^\theta & q_n^\theta \\ p_{n-1}^\theta & q_{n-1}^\theta \end{pmatrix}\text{ for all $n\in \N\setminus\{0\}$.}
\end{cases}
\end{equation}
We then note that $\frac{p_n^\theta}{q_n^\theta}=[a_0, a_1, \ldots,a_n]$ for all $n \in \N$, and therefore $\left(\frac{p_n^\theta}{q_n^\theta}\right)_{n\in\N}$ converges to $\theta$ (see \cite{Hardy38}).

Expression (\ref{pq-rel-eq}) contains the crux for the construction of the Effros-Shen AF algebras. 

\begin{notation}
Throughout this paper, we shall employ the notation $x\oplus y \in X\oplus Y$ to mean that $x\in X$ and $y\in Y$ for any two vector spaces $X$ and $Y$ whenever no confusion may arise, as a slight yet convenient abuse of notation.
\end{notation}

\begin{notation}\label{af-theta-notation}
Let $\theta \in (0,1)\setminus\Q$ and $\theta = [a_j]_{j\in\N}$ be the continued fraction expansion of $\theta$. Let $(p_n^\theta)_{n\in\N}$ and $(q_n^\theta)_{n\in\N}$ be defined by Expression (\ref{pq-rel-eq}). We set $\af{\theta,0} = \C$ and, for all $n\in\N\setminus\{0\}$, we set:
\begin{equation*}
\af{\theta, n} = \alg{M}(q_{n}^\theta) \oplus \alg{M}(q_{n-1}^\theta) \text{,}
\end{equation*}
and:
\begin{equation*}
\alpha_{\theta,n} : a\oplus b \in \af{\theta,n} \longmapsto \begin{pmatrix}
a & & & \\
  & \ddots & & \\
  &        & a & \\
 &        &   & b 
\end{pmatrix} \oplus a \in \af{\theta, n+1} \text{,}
\end{equation*}
where $a$ appears $a_{n+1}$ times on the diagonal of the right hand side matrix above. We also set $\alpha_0$ to be the unique unital *-morphism from $\C$ to $\af{\theta,1}$.

We thus define the Effros-Shen C*-algebra $\af{\theta}$, after \cite{Effros80b}:
\begin{equation*}
\af{\theta} = \varinjlim \left(\af{\theta,n}, \alpha_{\theta,n}\right)_{n\in\N}=\varinjlim \mathcal{I}_\theta\text{.}
\end{equation*} 
\end{notation}

We now present our continuity result for Effros-Shen AF Algebras from \cite{AL}.  We note that the Baire space is homeomorphic to the irrationals in $(0,1)$.  A proof of this can be found in \cite[Proposition 5.10]{AL}.  

\begin{theorem}[{\cite[Theorem 5.14]{AL}}]\label{af-theta-thm}
Using Notation (\ref{af-theta-notation}) , the function:
\begin{equation*}
\theta \in ((0,1)\setminus \Q, \vert \cdot \vert) \longmapsto \left(\af{\theta}, \Lip_{\mathcal{I}_\theta ,\sigma_\theta}^{\beta_\theta} \right) \in (\mathcal{QQCMS}_{2,0}, \qpropinquity{} )
\end{equation*}
is continuous from $(0,1)\setminus\Q$, with its topology as a subset of $\R$, to the class of $(2,0)$-quasi-Leibniz quantum compact metric spaces metrized by the quantum propinquity $\qpropinquity{}$, where $\sigma_\theta$ is the unique faithful tracial state, and $\beta_\theta$ is the sequence of the reciprocal of dimensions of the inductive sequence, $\mathcal{I}_\theta$.
\end{theorem}

  In \cite{Aguilar16a}, we generalized the convergence results in \cite{AL} utilizing the notion of a fusing family of inductive sequences.  We will utilize this notion and this general convergence theorem in this paper for our quotient convergence results.  We list the appropriate definitions and results here.  
  
  We now define a notion of fusing  inductive sequences together in Definition (\ref{coll-def}), which is equivalent to  convergence of ideals of an AF algebra in the Fell topology, which is seen by Lemma (\ref{metric-fusing-equiv-lemma}).

\begin{notation}
Let $\overline{\N}=\N \cup \{\infty\}$ denote the Alexandroff compactification of $\N$ with respect to the discrete topology of $\N$. For $N \in \N$, let $\N_{\geq N} = \{ k \in \N : k \geq N \}$, and similarly, for $\overline{\N}_{\geq N}$.
\end{notation} 
\begin{definition}[{\cite[Definition 3.5]{Aguilar16a}}] \label{coll-def} We consider 2 cases of inductive sequences in this definition.
\begin{case}Closure of union
\end{case}
For each $k \in \overline{\N}$, let $\A^k$ be a C*-algebras with  $\A^k = \overline{\cup_{n \in \N} \A_{k,n} }^{\Vert \cdot \Vert_{\A^k}}$ such that $\mathcal{U}^k = \left(\A_{k,n}\right)_{n \in \N}$ is a non-decreasing sequence of  C*-subalgebras of $\A^k$, then we say $\{\A^k : k \in \overline{\N}\}$ is a {\em fusing family} if: 
\begin{enumerate} 
\item There exists  $(c_n)_{n \in \N} \subseteq \N$ non-decreasing such that $\lim_{n \to \infty} c_n=\infty$, and
\item for all $N \in \N $, if $k \in \N_{\geq c_N}$, then $\A_{k,n} = \A_{\infty,n}$ for all $n \in \{0,1,\ldots, N\}.$
\end{enumerate}

\begin{case}Inductive limit
\end{case}
For each $k \in \overline{\N}$, let  $\mathcal{I}(k)=(\A_{k,n}, \alpha_{k,n})_{n \in \N}$ be an inductive sequence with inductive limit, $\A^k$.  We say that the family of $C^*$-algebras $\{ \A^k : k \in \overline{\N} \}$ is an {\em IL-fusing family} of $C^*$-algebras if:
\begin{enumerate} 
\item There exists  $(c_n)_{n \in \N} \subseteq \N$ non-decreasing such that $\lim_{n \to \infty} c_n=\infty$, and
\item for all $N \in \N $, if $k \in \N_{\geq c_N}$, then $(\A_{k,n}, \alpha_{k,n}) = (\A_{\infty,n}, \alpha_{\infty,n})$ for all $n \in \{0,1,\ldots, N\}.$
\end{enumerate}

In either case, we call the sequence  $(c_n)_{n \in \N} $ the {\em fusing sequence}.
\end{definition}

\begin{remark}
Propinquity convergence results for sequences of AF algebras are all in terms of inductive limits.  We will see the closure of union case appear when working with ideals in Sections (\ref{metric-ideal-IL} - \ref{conv-quotients-quantum}). Also, note that any IL-fusing family may be viewed as a fusing family via the canonical *-homomorphisms of Notation (\ref{ind-lim}), which is why we don't decorate the term {\em fusing family} in the closure of union case. 
\end{remark}

Next, we provide our general criteria for convergence of AF algebras in propinquity using the notion of fusing family along with suitable notions of convergence of the remaining tools used to build our faithful tracial state Lip-norms.

\begin{theorem}[{\cite[Theorem 3.10]{Aguilar16a}}]\label{af-cont}For each $k \in \overline{\N}$, let  $\mathcal{I}(k)=(\A_{k,n}, \alpha_{k,n})_{n \in \N}$ be an inductive sequence of finite dimensional $C^*$-algebras  with $C^*$-inductive limit $\A^k ,$ such that $\A_{k,0} =\A_{k',0} \cong \C$ and $\alpha_{k,n}$ is unital and injective for all $k ,k' \in \overline{\N}, n\in\N$.  If:
\begin{enumerate}
\item $\{ \A^k : k \in \overline{\N} \}$ is an IL-fusing family with fusing sequence $(c_n)_{n \in \N}$,
\item $\{\tau^k : \A^k \rightarrow \C \}_{k \in \overline{\N}}$ is a family of faithful tracial states such that for each $N \in \N$, we have that $\left(\tau^k \circ \indmor{\alpha_k}{N}\right)_{k \in \N_{\geq c_N}}$ converges to $\tau^\infty \circ \indmor{\alpha_\infty}{N}$ in the weak-* topology on $\StateSpace(\A_{\infty,N})$, and
\item $\{\beta^k : \overline{\N} \rightarrow (0,\infty) \}_{k \in \overline{\N}}$ is a family of convergent sequences such that for all $N\in \N $ if $k \in \N_{\geq c_N}$, then  $\beta^k(n)=\beta^\infty(n)$ for all $n \in \{0,1,\ldots,N\}$ and there exists $B: \overline{\N} \rightarrow (0,\infty)$ with $B(\infty)=0$ and $\beta^m(l) \leq B(l)$ for all $m,l \in \overline{\N}$
\end{enumerate}
then:
\begin{equation*}
\lim_{k \to \infty} \qpropinquity{} \left(\left(\A^k , \Lip^{\beta^k}_{\mathcal{I}(k),\tau^k}\right),\left(\A^\infty , \Lip^{\beta^\infty}_{\mathcal{I}(\infty),\tau^\infty}\right)\right)=0,
\end{equation*}
where $\Lip^{\beta^k}_{\mathcal{I}(k), \tau^k}$ is given by Theorem (\ref{AF-lip-norms-thm}).
\end{theorem}
This theorem generalized the UHF and Effros-Shen algebra convergence results of \cite{AL}, in which we showed this in the Effros-Shen algebra case in the proof of \cite[Theorem 3.14]{Aguilar16a}.

\section{A topology on the Ideal space of C*-Inductive Limits}\label{metric-ideal-IL}

For a fixed C*-algebra, the ideal space may be endowed with  various natural topologies.  We may identify each ideal with a quotient, which is a C*-algebra itself.  Now, this defines a function from the ideal space, which has natural topologies, to the class of C*-algebras.  But, if each quotient has a quasi-Leibniz Lip-norm, then this function becomes much more intriguing as we may now discuss its continuity or lack thereof since we now have topology on the codomain provided by quantum propinquity. Towards this, we develop a  topology on ideals of any C*-inductive limit that is compatible with this goal. The purpose of this is to allow fusing families of ideals to provide fusing families of quotients in Proposition (\ref{ind-lim-quotients-prop}) --- a first step in providing convergence of quotients in quantum propinquity.  But, our topology is greatly motivated by the Fell topology on the ideal space and is stronger than the Fell topology in general and equal to the Fell topology in the AF case, and we provide an explicit metric to metrize the Fell topology in the AF case. In order to construct our topology, we will use the given inductive sequence of a C*-inductive limit to construct inverse topology from the Fell topologies of the given C*-algebras of the inductive sequence.  This is not only for aesthetic purposes, but also simplifies some proofs and provides a better understanding of the overall structure of the topology we introduce.  But, we first define the Fell topology on ideals and prove some basic results, which requires the Jacobson topology on the class of primitive ideals.  As the definition of the Jacobson topology  is quite involved, we do not provide a complete definition of the Jacobson topology, but we provide a reference and a characterization of the closed sets in Definition (\ref{Jacobson-topology}).  

\begin{definition}\label{Jacobson-topology}Let $\A$ be a C*-algebra.  Denote the set of norm closed two-sided ideals of $\A$ by $\mathrm{Ideal}(\A),$ in which we include the trivial ideals $\emptyset$ and $\A$. Define: 
\begin{equation*}
\begin{split}
& \mathrm{Prim}(\A) \\
&  = \left\{ J \in \mathrm{Ideal}(\A) : J=\ker \pi, \pi \ \text{is a non-zero irreducible *-representation of} \ \A \right\}.
\end{split}
\end{equation*}
Note that $\A \not\in \mathrm{Prim}(\A)$.

The {\em Jacobson topology} on $\mathrm{Prim}(\A)$, denoted  $Jacobson$ is defined in {\cite[Theorem 5.4.2 and Theorem 5.4.6]{Murphy}}.  Let $F$ be a closed set in the Jacobson topology, then there exists $I_F \in \mathrm{Ideal}(\A)$ such that $F=\{ J \in \mathrm{Prim}(\A) : J \supseteq I_F\}$  {\cite[Theorem 5.4.7]{Murphy}}.
\end{definition}
\begin{convention}
Given a C*-algebra, $\A$, and $I \in \mathrm{Ideal}(\A)$, an element of the quotient $\A/I$ will be denoted by $a+I$ for some $a \in \A$. Furthermore, the quotient norm will be denoted $\Vert a+I\Vert_{\A/I}= \inf \left\{ \Vert a+b\Vert_\A : b\in I \right\}$.  
\end{convention}
Now, we may define the Fell topology, which is a topology on all ideals of a C*-algebra.  We begin by presenting the definition of the Fell topology on closed sets of any topological space along with some facts, which will help with some later proofs. 
\begin{definition}[\cite{Fell62}]\label{topological-Fell-def}
Let $(X, \tau)$ be a topological space with topology $\tau$ (no further assumptions made).  Let $\mathcal{C}l(X)$ denote the set of closed subsets of $X$.  Let $K$ be a compact set of $X$, and let $F$ be a finite family of non-empty open subsets of $X$. Define:
\begin{equation*}
U(K,F)=\{  Y \in \mathcal{C}l(X) : Y \cap K = \emptyset \ \land \ Y\cap A\neq \emptyset \text{ for all } A \in F\}.
\end{equation*}  A basis for the Fell topology on $\mathcal{C}l(X)$ denoted by $\tau_{ \mathcal{C}l(X)}$ is given by:
\begin{equation*}
\left\{ U(K,F) \subseteq \mathcal{C}l(X): K\subseteq X \text{ is compact} \ \land \ F\subseteq \tau\setminus \{\emptyset\}, F \ \text{ is finite}\right\}. 
\end{equation*}
\end{definition}
Now, we list some facts about this topology and the striking conclusion that the Fell topology on $\mathcal{C}l(X)$ is always compact and is Hausdorff when $X$ is locally compact.
\begin{lemma}[{\cite[Lemma 1 and Theorem 1]{Fell62}}]\label{topological-Fell-lemma}

If $(X, \tau)$ is a topological space, then the topoogical space $\left(\mathcal{C}l(X),\tau_{ \mathcal{C}l(X)}\right)$ is compact.

If $(X, \tau)$ is a locally compact space, then  the topoogical space $\left(\mathcal{C}l(X),\tau_{ \mathcal{C}l(X)}\right)$ is compact Hausdorff.
\end{lemma}
Next, we are in a position to apply this to build a topology  on the ideal space of a C*-algebra.
\begin{definition}[{\cite{Fell61}}]\label{Fell-topology}
Let $\A$ be a C*-algebra.  Let $\mathcal{C}l(\mathrm{Prim}(\A))$ be the set of closed subsets of $(\mathrm{Prim}(\A),  Jacobson)$ with compact Hausdorff topology, $\tau_{\mathcal{C}l(\mathrm{Prim}(\A))}$, given by {\cite[Theorem 2.2]{Fell61}} and more generally \cite{Fell62} and Lemma (\ref{topological-Fell-lemma}), where we note that $\check{\A}=  \mathrm{Prim}(\A)$ in \cite{Fell61}. Let $fell : \mathrm{Ideal}(\A) \rightarrow \mathcal{C}l(\mathrm{Prim}(\A))$ denote the map: \begin{equation*}
fell (I) = \left\{ J \in \mathrm{Prim}(\A) : J \supseteq I\right\},
\end{equation*}
which is a one-to-one correspondence  {\cite[Theorem 5.4.7]{Murphy}}.  The {\em Fell topology} on $\mathrm{Ideal}(\A)$, denoted $Fell$, is the initial topology on $\mathrm{Ideal}(\A)$ induced by $fell$, which is the weakest topology for which $fell$ is continuous. Equivalently, 
\begin{equation*}
 Fell=\left\{ U \subseteq \mathrm{Ideal}(\A) : U=fell^{-1} (V), V \in \tau_{\mathcal{C}l(\mathrm{Prim}(\A))}\right\}, 
\end{equation*}
and $(\mathrm{Ideal}(\A), Fell)$ is therefore compact Hausdorff since $fell$ is a bijection.
\end{definition}

The following Lemma (\ref{Fell-conv}) is stated in \cite[Section 2]{Archbold87}, where the Fell topology,  $Fell$,  is denoted $\tau_s$. We provide a proof.

\begin{lemma}\label{Fell-conv}
Let $\A$ be a C*-algebra.  Let $\left(I_\mu \right)_{\mu \in \Delta} \subseteq \mathrm{Ideal}(\A)$ be a net and $I \in \mathrm{Ideal}(\A) $.  

The net $\left(I_\mu \right)_{\mu \in \Delta}$ converges to $I$ with respect to the Fell topology if and only if the net
$\left(\left\Vert a +I_\mu \right\Vert_{\A/I_\mu}\right)_{\mu \in \Delta} \subseteq \R$ converges to $\Vert a +I \Vert_{\A/I} \in \R$ with respect to the usual topology on $\R$ for all $a \in \A$.
\end{lemma}
\begin{proof}  
By {\cite[Theorem 2.2]{Fell61}}, let $Y \in \mathcal{C}l(\mathrm{Prim}(\A))$, define: 
\begin{equation*}
M_Y : a \in \A  \longmapsto  \sup \left\{ \left\Vert a+I\right\Vert_{\A/I} : I \in Y \right\} \in \R,  
\end{equation*}
since in Fell's notation, given an ideal $S$, we have  $S_a=a+S$  according to his definition of {\em transform}  in {\cite[Section 2.1]{Fell61}} in the context of the primitive ideal space $\check{\A}=\mathrm{Prim}(\A)$. But, by the first line of the proof of {\cite[Theorem 2.2]{Fell61}}, we note that $\cap_{I \in Y} I \in \mathrm{Ideal}(\A)$  and:
\begin{equation}\label{M-fell} M_Y(a)=\left\Vert a+\cap_{I \in Y} I \right\Vert_{\A/\left(\cap_{I \in Y} I \right)},
\end{equation}
for all $a \in \A.$

Let $P \in \mathrm{Ideal}(\A)$, then $fell(P)=\{ J \in \mathrm{Prim}(\A) : J \supseteq P\} \in \mathcal{C}l(\mathrm{Prim}(\A))$ by Definition (\ref{Fell-topology}). Note that $\cap_{H \in fell(P)} H = P$ by {\cite[Theorem 5.4.3]{Murphy}}. Thus, by Expression (\ref{M-fell}):
\begin{equation}\label{norm-func}
M_{fell(P)}(a) = \Vert a+P\Vert_{\A/P}.
\end{equation}

Now, assume that $\left(I_\mu \right)_{\mu \in \Delta} \subset \mathrm{Ideal}(\A)$  converges to $I \in \mathrm{Ideal}(\A)$ with respect to the Fell topology.  Since $fell$ is continuous, the net $\left(fell\left(I_\mu\right)\right)_{\mu \in \Delta} \subseteq \mathcal{C}l(\mathrm{Prim}(\A))$ converges to $fell(I) \in \mathcal{C}l(\mathrm{Prim}(\A))$ with respect to the topology on $\mathcal{C}l(\mathrm{Prim}(\A))$.  By {\cite[Theorem 2.2]{Fell61}}, the net of functions $\left(M_{fell\left(I_\mu \right)}\right)_{\mu \in \Delta}$ converges pointwise to $M_{fell(I)}$, which completes the forward implication by Equation (\ref{norm-func}). 

For the reverse implication, assume that $\left(\left\Vert a +I_\mu \right\Vert_{\A/I_\mu} \right)_{\mu \in \Delta} \subseteq \R$ converges to $\Vert a +I \Vert_{\A/I} \in \R$ with respect to the usual topology on $\R$ for all $a \in \A$ and for some net  $\left(I_\mu \right)_{\mu \in \Delta} \subseteq \mathrm{Ideal}(\A)$ and $I \in \mathrm{Ideal}(\A)$.   But, then by Equation (\ref{norm-func}) and assumption, the net $\left(M_{fell\left(I_\mu \right)}\right)_{\mu \in \Delta}$ converges pointwise to $M_{fell(I)}$.  By {\cite[Theorem 2.2]{Fell61}}, the net $\left(fell\left(I_\mu \right)\right)_{\mu \in \Delta} \subseteq \mathcal{C}l(\mathrm{Prim}(\A))$ converges to $fell(I) \in \mathcal{C}l(\mathrm{Prim}(\A))$ with respect to the topology on $\mathcal{C}l(\mathrm{Prim}(\A))$.  However, as $fell$ is a continuous bijection between the compact Hausdorff spaces $\left(\mathrm{Ideal}(\A), Fell\right)$ and \\ $\left(\mathcal{C}l(\mathrm{Prim}(\A)), \tau_{\mathcal{C}l(\mathrm{Prim}(\A))}\right)$, the map $fell$ is a homeomorphism.  Thus, we conclude that $\left(I_\mu \right)_{\mu \in \Delta}$ converges to $I$ with respect to the Fell topology.
\end{proof}
Now, the Fell topology induces a topology on $\mathrm{Prim}(\A)$ via its relative topology. But, the set $\mathrm{Prim}(\A)$ can also be equipped with the Jacobson topology (see Definition (\ref{Jacobson-topology})).  Thus, a comparison of both topologies is in order in Proposition (\ref{Fell-Jacobson}), which can be proven using Lemma (\ref{Fell-conv}).

\begin{proposition}\label{Fell-Jacobson}
The relative topology induced by the Fell topology of Definition (\ref{Fell-topology}) on  $\mathrm{Prim}(\A)$  contains the Jacobson topology of Definition (\ref{Jacobson-topology}) on $\mathrm{Prim}(\A)$.
\end{proposition} 
\begin{proof}
Let $F \subseteq \mathrm{Prim}(\A)$ be closed in the Jacobson topology. Then, there exists a unique $I_F \in \mathrm{Ideal}(\A)$ such that  $F=\{ J \in \mathrm{Prim}(\A) : J \supseteq I_F \}$ by Definition (\ref{Fell-topology}).   

Let $J \in  \mathrm{Prim}(\A)$ such that there exists  a convergent net $(J^\mu)_{\mu \in \Delta} \subseteq F$  that converges to $J \in \mathrm{Prim}(\A)$ in the Fell topology.  Let $x \in I_F $, then $x \in J^\mu $ for all $\mu \in \Delta$.  Thus, the net $\left(\left\Vert x+J^\mu  \right\Vert_{\A/J^\mu }\right)_{\mu \in \Delta}=(0)_{\mu \in \Delta},$ which is a net that converges to $\left\Vert x+J \right\Vert_{\A/J}$ by Lemma (\ref{Fell-conv}).  Thus, the limit $  \left\Vert x+J \right\Vert_{\A/J}=0$, which implies that $x \in J$.  Hence, $J \supseteq I_F$ and since $J \in \mathrm{Prim}(\A)$, we have $J \in F $. 

Thus, $F$ is closed in the relative topology on $\mathrm{Prim}(\A)$ induced by the Fell topology, which verifies the containment of the topologies.
\end{proof}

As stated earlier, it is with the Fell topology for which we will provide a notion of convergence of quotients from ideals of AF algebras.  But, it seems that a metric notion is in order to move from fusing family of ideals to a fusing family of quotients as we will see in Proposition (\ref{ind-lim-quotients-prop}). However, in the AF case, this metric is induced by an inverse limit topology induced by the Fell topology on the given inductive sequence.  This inverse limit topology can be stated in more generality than the AF case, and thus provides a new topology on the ideal space of C*-inductive limits. This requires a lemma, which also motivates the inverse limit topology.
\begin{lemma}\label{mono-cont-fell-lemma}
If $\A$ and $\B$ are C*-algebras such that there exists a *-monomorphism $\pi: \A \longrightarrow \B$, then the map:
\begin{equation*}
\pi_i : J \in \mathrm{Ideal}(\B) \mapsto \pi^{-1}\left(J \cap \pi(A)\right) \in \mathrm{Ideal}(\A) 
\end{equation*}
 is continuous with respect to the associated Fell topologies.

In particular, if $\A, \B$ are C*-algebras such that $\A \subseteq \B$ and denote the inclusion map by $\iota: \A \rightarrow \B$, then the map:
\begin{equation*}
\iota_i : J \in \mathrm{Ideal}(\B) \mapsto J \cap \A \in \mathrm{Ideal}(\A)
\end{equation*}
 is continuous with respect to the associated Fell topologies.
\end{lemma}
\begin{proof}  Since $\pi(\A)$ is a C*-subalgebra of $\B$, we have that $J \cap \pi(\A)$ is an ideal of $\pi(\A)$. Thus, the map $\pi_i$ is well-defined since $\pi$ is a *-homomorphism.

For continuity, we first prove the following claim which will serve many purposes in this article.
\begin{claim}\label{quo-ind-map} Let $\A$ be a C*-algebra and let $\A_k$ be a C*-subalgebra.

If $J\in \mathrm{Ideal}(\A)$, then the map: 
\begin{equation}\phi^k_J : a + J \cap \A_k \in \A_k/(J \cap \A_k) \longmapsto a+J \in \A/J .
\end{equation}
is a  *-monomorphism and thus an isometry.
\end{claim}
\begin{proof}[Proof of claim]Assume that $a,b \in \A_k $ such that $a+J \cap \A_k= b+J \cap \A_k$, which implies that $a-b \in J \cap \A_k \subseteq J \implies a+J = b+J$, and thus, $\phi^k_J$ is well-defined.  Next, assume that $a,b \in \A_k$ such that $a+J=b+J$, which implies that $a-b \in J$.  But, we have $a-b \in \A_k \implies a-b \in J \cap \A_k $ and $a+J \cap \A_k= b+J \cap \A_k $, which provides injectivity.  Thus, for each $k \in \N$, we have $\phi^k_J$ is a well-defined injective *-homomorphism since $J$ is an ideal.  Hence, the map $\phi^k_J $ is an isometry for  any $J \in \mathrm{Ideal}(\A)$, which proves the claim.
\end{proof}

To continue with continuity, let $(J_\mu)_{\mu \in \Pi} \subseteq \mathrm{Ideal}(\B)$ be a net that converges to $J \in \mathrm{Ideal}(\B)$ with respect to the Fell topology.  Fix $a \in \A$ and $\mu \in \Pi$, then since $\pi$ is injective, isometric, and surjects onto $J_\mu \cap \pi(\A)$:
\begin{align*}
\left\| a+ \pi_i(J_\mu)\right\|_{\A/\pi_i(J_\mu)} & = \left\| a+ \pi^{-1}(J_\mu\cap \pi(\A))\right\|_{\A/\pi_i(J_\mu)} = \inf_{a' \in \pi^{-1}(J_\mu\cap \pi(\A))} \left\| a-a' \right\|_\A\\
& =  \inf_{\pi(a') \in J_\mu\cap \pi(\A)} \left\| \pi(a)-\pi(a') \right\|_\B =\inf_{b \in J_\mu\cap \pi(\A)} \left\| \pi(a)-b \right\|_\B \\
&  =\left\| \pi(a) +(J_\mu \cap \pi(\A))\right\|_{\pi(\A)/(J_\mu \cap \pi(\A))} = \left\| \pi(a) +J_\mu \right\|_{\B/J_\mu},
\end{align*}
where we used Claim (\ref{quo-ind-map}) in the last equality.  Thus, by Lemma (\ref{Fell-conv}), we have that $(\pi_i(J_\mu))_{\mu \in \Pi}\subseteq \mathrm{Ideal}(\A)$ converges to $\pi_i(J) \in \mathrm{Ideal}(\A)$ in the Fell topology, which concludes the proof.
\end{proof}

With this we are ready to define a topology which will induce a new topology on the ideal space of an inductive limit.  But, first, a remark on our change in the language of inductive limits for some of the following results. 
\begin{remark}
By {\cite[Chapter 6.1]{Murphy}}, if $\mathcal{I}=(\A_n , \alpha_n)_{n \in \N}$ is an inductive sequence with inductive limit $\A=\underrightarrow{\lim} \ \mathcal{I}$ as in Notation (\ref{ind-lim}), then  $( \indmor{\alpha}{n}(\A_n))_{n \in \N}$ is a non-decreasing sequence of C*-subalgebras of $\A$, in which $\A=\overline{\cup_{n \in \N} \indmor{\alpha}{n}(\A_n)}^{\Vert \cdot \Vert_\A}$.  Thus, in some of the following definitions and results, when we say, "Let $\A$ be a C*-algebra with a non-decreasing sequence of C*-subalgebras $\mathcal{U}= (\A_n)_{n \in \N}$ such that $\A=\overline{\cup_{n \in \N} \A_n}^{\Vert \cdot \Vert_\A}$," we are also including the case of inductive limits.  The purpose of this will be to avoid  notational confusion later on  when we  work with multiple inductive limits (see for example Proposition (\ref{ind-lim-quotients-prop})), and the purpose of this remark is to note that this does not weaken our results.
\end{remark}
Following \cite{Willard}, we define the inverse limit sequence of topoogical spaces and its limit:
\begin{definition}\label{inv-seq-lim-def}
A family $(X_n, \tau_n, f_{n+1})_{n \in \N}$ is an {\em inverse limit sequence} of topological spaces if $(X_n, \tau_n)_{n \in \N}$ is a family of topological spaces  and $(f_{n+1})_{n \in \N}$  is a family of continuous functions such that  $f_{n+1} : X_{n+1} \rightarrow X_n$   for all $n \in \N$. The {\em inverse limit space} of $(X_n, \tau_n, f_{n+1})_{n \in \N}$ denoted by  $(X_\infty, \tau_\infty)$  is the subset $X_\infty$ of $\prod_{n \in \N} X_n$ defined by:
\begin{equation*}
X_\infty = \left\{ (x_n)_{n \in \N} \in \prod_{n \in \N} X_n : f_{n+1}(x_{n+1}) = x_n \text{ for all } n \in \N \right\},
\end{equation*}
where $\tau_\infty$ is the topology on $X_\infty$ given by the relative topology induced by the product topology on $\prod_{n \in \N} X_n$ with respect to the given topologies $\tau_n$  on $X_n$ for all $n \in \N$.
\end{definition}
Our topology on the ideal space will be induced by an initial topology by the following map once our inverse limit is established.
\begin{proposition}\label{ind-lim-ideals}
Let $\A$ be a C*-algebra with a non-decreasing sequence of C*-subalgebras $\mathcal{U}= (\A_n)_{n \in \N}$ such that $\A=\overline{\cup_{n \in \N} \A_n}^{\Vert \cdot \Vert_\A}$. The map: 
\begin{equation*}
i(\cdot, \mathcal{U} ): I \in \mathrm{Ideal}(\A) \longmapsto \left(I \cap \A_n \right)_{n \in \N} \in \prod_{n \in \N} \mathrm{Ideal}(\A_n)
\end{equation*}
is a well-defined injection.
\end{proposition}  
\begin{proof}
 Since $I \in \mathrm{Ideal}(\A) $ and $\A_n $  is a C*-subalgebra for all $n \in \N$, we have that $I \cap \A_n  \in \mathrm{Ideal}(\A_n) $ for all $n \in \N$.    Thus, the map $i(\cdot, \mathcal{U} )$ is well-defined.  

Next, for injectivity, assume that $I,J \in \mathrm{Ideal}(\A)$ such that $i(I, \mathcal{U})=i(J,\mathcal{U} ).$ Hence, the sets  $I \cap \A_n =J \cap \A_n$ for all $n \in \N$, which implies that $\cup_{n \in \N} (I\cap \A_n)= \cup_{n \in \N} (J\cap \A_n)$. Therefore, the closures $\overline{\cup_{n \in \N} (I\cap \A_n)}^{\Vert \cdot \Vert_\A} =\overline{\cup_{n \in \N} (J \cap \A_n)}^{\Vert \cdot \Vert_\A}$.  But, by {\cite[Lemma III.4.1]{Davidson}}, we conclude $I=J$.
\end{proof}
The following produces the remaining ingredients for our topology.
\begin{lemma}\label{fell-inverse-lemma}
Let $\A$ be a C*-algebra with a non-decreasing sequence of C*-subalgebras $\mathcal{U}= (\A_n)_{n \in \N}$ such that $\A=\overline{\cup_{n \in \N} \A_n}^{\Vert \cdot \Vert_\A}$.  

If for each $n \in \N$, we denote ${\iota_{n+1}}_i : J \in \mathrm{Ideal}(\A_{n+1}) \rightarrow J \cap \A_{n} \in \mathrm{Ideal}(\A_n)$, then the family $\left(\mathrm{Ideal}(\A_n), Fell, {\iota_{n+1}}_i\right)_{n \in \N}$ is an inverse limit sequence with non-empty compact Hausdorff inverse limit space $\left(\mathrm{Ideal}(\A)_\infty, Fell_\infty \right)$ such that:
\begin{equation*}
\mathrm{Ideal}(\A)_\infty=\left\{ (J_n)_{n \in \N} \in \prod_{n \in \N} \mathrm{Ideal}(\A_n) : J_{n+1}\cap \A_n = J_n \text{ for all } n \in \N\right\}
\end{equation*}
and thus, using notation from Proposition (\ref{ind-lim-ideals}):
\begin{equation*}
i\left(\mathrm{Ideal}(\A), \mathcal{U}\right) \subseteq \mathrm{Ideal}(\A)_\infty.
\end{equation*}
\end{lemma}
\begin{proof}
The conclusions follows immediately from Lemma (\ref{mono-cont-fell-lemma}), Definition (\ref{inv-seq-lim-def}), and Proposition (\ref{ind-lim-ideals}). The non-empty compact Hausdorff conclusion follows from \cite[Theorem 29.11]{Willard} and the fact that $\mathrm{Ideal}(\A_n)$ equipped with the Fell topology is a non-empty compact Hausdorff space for each $n \in \N$. 
\end{proof}
\begin{definition}\label{Fell-ind-top-def}
Let $\A$ be a C*-algebra with a non-decreasing sequence of C*-subalgebras $\mathcal{U}= (\A_n)_{n \in \N}$ such that $\A=\overline{\cup_{n \in \N} \A_n}^{\Vert \cdot \Vert_\A}$.  

By Lemma (\ref{fell-inverse-lemma}),  the intitial topology induced by $i(\cdot, \mathcal{U})$ and the topological space $\left(\mathrm{Ideal}(\A)_\infty, Fell_\infty \right)$  on $\mathrm{Ideal}(\A)$ exists and is a Hausdorff (by injectivity of $i(\cdot, \mathcal{U})$), which we denote by $Fell_{i(\mathcal{U})}$. 
\end{definition}
We will now provide some sufficient conditions for when $Fell_{i(\mathcal{U})}$ agrees with $Fell$, but first, we show that is always the case that $Fell \subseteq  Fell_{i(\mathcal{U})}$.
\begin{proposition}\label{Fell-topology-contain}
If $\A$ is a C*-algebra with a non-decreasing sequence of C*-subalgebras $\mathcal{U}= (\A_n)_{n \in \N}$ such that $\A=\overline{\cup_{n \in \N} \A_n}^{\Vert \cdot \Vert_\A}$, then  the Fell topology $Fell$ on $\mathrm{Ideal}(\A)$ is contained in the topology $Fell_{i(\mathcal{U})}$ of Definition (\ref{Fell-ind-top-def}). 
\end{proposition}
\begin{proof}
Let $(J_\mu)_{\mu \in \Pi} \subseteq \mathrm{Ideal}(\A)$ be a net that converges to $J \in \mathrm{Ideal}(\A)$ with respect to $Fell_{i(\mathcal{U})}$.  Hence, the net $(i(J_\mu, \mathcal{U}))_{\mu \in \Pi} \subseteq \mathrm{Ideal}(\A)_\infty$ is a net that converges to $i(J, \mathcal{U}) \in \mathrm{Ideal}(\A)_\infty$ with respect to $Fell_\infty$ by definition.  Again, by definition, the net $(i(J_\mu, \mathcal{U}))_{\mu \in \Pi} \subseteq \mathrm{Ideal}(\A)_\infty$ is a net that converges to $i(J, \mathcal{U}) \in \mathrm{Ideal}(\A)_\infty$ with respect to the product topology on $\prod_{n \in \N} \mathrm{Ideal}(\A_n),$ in which each $\mathrm{Ideal}(\A_n)$ is equipped with topology $Fell$.

First, fix $a \in \cup_{n \in \N} \A_n$.  So, there exists $N \in \N$ such that $a \in \A_N$.  Let $\mu \in \Pi$.  Thus, by Claim (\ref{quo-ind-map}):
\begin{equation*}
\left\| a+J_\mu\right\|_{\A/J_\mu}= \left\| a+J_\mu \cap \A_N \right\|_{\A_N/(J_\mu \cap \A_N)}.
\end{equation*}
Therefore, as the projection maps are continuous for the product topology, we conclude that the net  $\left( \left\| a+J_\mu \cap \A_N \right\|_{\A_N/(J_\mu \cap \A_N)}\right)_{\mu \in \Pi}$ converges to $ \left\| a+J \cap \A_N \right\|_{\A_N/(J \cap \A_N)}$ in the usual topology on $\R$.  

Hence, the net $\left(\left\| a+J_\mu\right\|_{\A/J_\mu}\right)_{\mu \in \Pi}$ converges to $ \left\| a+J \right\|_{\A/J }$ in the usual topology on $\R$ for all $ a \in \cup_{n \in \N} \A_n.$  Now, let $a\in \A$.  Let $\varepsilon >0$. There exists $N \in \N$ and $a_N \in \A_N \subseteq \cup_{n \in \N}\A_n$ such that $\|a-a_N\|_\A < \varepsilon/3$.  Thus, there exists $\mu_0 \in \Pi$ such that for all $\mu \geq \mu_0$, we have:
\begin{equation*}
\left| \left\|a_N + J_\mu \right\|_{\A/J_\mu} - \left\| a_N + J \right\|_{\A/J} \right| < \varepsilon/3.
\end{equation*}
Hence, if $\mu \geq \mu_0$, then:
\begin{align*}
\left| \left\|a + J_\mu \right\|_{\A/J_\mu} - \left\| a + J \right\|_{\A/J} \right|& \leq \left| \left\|a+J_\mu\right|_{\A/J_\mu} - \left\| a_N + J_\mu \right\|_{\A/J_\mu} \right| \\
& \quad + \left| \left\|a_N + J_\mu \right\|_{\A/J_\mu} - \left\| a_N + J \right\|_{\A/J} \right|  \\
& \quad  + \left| \left\|a_N+J \right\|_{\A/J} - \left\|a + J\right\|_{\A/J} \right| \\
& < \left\| a-a_N + J_\mu \right\|_{\A/J_\mu} + \varepsilon/3 + \left\| a_N-a +J \right\|_{\A/J}\\
& \leq 2\left\|a-a_N \right\|_\A +\varepsilon/3<\varepsilon,
\end{align*}
which completes the proof by Lemma (\ref{Fell-conv}).
\end{proof}
Thus, by this proposition and Lemma (\ref{fell-inverse-lemma}), if it is also the case that the topology $Fell_{i(\mathcal{U})}$ is compact, then it must agree with the topology $Fell$ by maximal compactness of Hausdorff spaces.  An obvious way that this would be true is if the map $i(\cdot, \mathcal{U})$ surjected onto $\mathrm{Ideal}(\A)_\infty$.  It turns out that this is the case for all AF algebras, and we provide a characterization of this scenario by a condition on the algebraic  ideals of a C*-algebra motivated by Bratteli's work in \cite{Bratteli72}.  This is the next lemma that follows after the following notation.
\begin{notation}\label{algideal-not}
Let $\A$ be a C*-algebra with a non-decreasing sequence of C*-subalgebras $\mathcal{U}= (\A_n)_{n \in \N}$ such that $\A=\overline{\cup_{n \in \N} \A_n}^{\Vert \cdot \Vert_\A}$.  

Let $\mathrm{algIdeal}(\cup_{n \in \N}\A_n)$ denote the set of two-sided ideals of $\cup_{n \in \N}\A_n$  that are not necessarily closed in $\A$.  

Let $\mathrm{algIdeal}(\cup_{n \in \N}\A_n)_{prod}= \{ J \in \mathrm{algIdeal}(\cup_{n \in \N}\A_n) : J \cap \A_n \in \mathrm{Ideal}(\A_n) \text{ for all } n \in \N \}.$
\end{notation}

\begin{lemma}\label{i-surj-lemma}
If $\A$ is a C*-algebra with a non-decreasing sequence of C*-subalgebras $\mathcal{U}= (\A_n)_{n \in \N}$ such that $\A=\overline{\cup_{n \in \N} \A_n}^{\Vert \cdot \Vert_\A}$, then using notation from Proposition (\ref{ind-lim-ideals}) and Lemma (\ref{fell-inverse-lemma}) and Notation (\ref{algideal-not}),  the map $J \in \mathrm{algIdeal}(\cup_{n \in \N}\A_n)_{prod} \mapsto \overline{J}^{\|\cdot \|_\A} \in \mathrm{Ideal}(\A)$ is a well-defined surjection onto $\mathrm{Ideal}(\A)$.

Furthermore, the following two statements are equivalent:
\begin{enumerate}
\item the function $J \in \mathrm{algIdeal}(\cup_{n \in \N}\A_n)_{prod} \mapsto \overline{J}^{\|\cdot \|_\A} \in \mathrm{Ideal}(\A)$ is injective, and thus a bijection onto $\mathrm{Ideal}(\A)$;
\item the function $i(\cdot, \mathcal{U})$ is surjective onto $\mathrm{Ideal}(\A)_\infty$, and thus a bijection onto $\mathrm{Ideal}(\A)_\infty$.
\end{enumerate}

In particular, if $\A$ is AF and $\mathcal{U}$ is chosen to be a family of finite-dimensional C*-algebras, then the map $i(\cdot, \mathcal{U})$ surjects onto $\mathrm{Ideal}(\A)_\infty$.
\end{lemma}
\begin{proof}
We first show that the map $J \in \mathrm{algIdeal}(\cup_{n \in \N}\A_n)_{prod} \mapsto \overline{J}^{\|\cdot \|_\A} \in \mathrm{Ideal}(\A)$ is a well-defined surjection onto $\mathrm{Ideal}(\A)$.  This  map  is clearly well-defined.  For surjectivity, let $J \in \mathrm{Ideal}(\A)$.  Then, we have that $J \cap \A_n \in \mathrm{Ideal}(\A_n)$ and thus, if we define $I=\cup_{n \in \N} (J \cap \A_n)$, then $I \in \mathrm{algIdeal}(\cup_{n \in \N}\A_n)_{prod}$ and $\overline{I}^{\|\cdot \|_\A}=\overline{\cup_{n \in \N} (J \cap \A_n)}^{\|\cdot \|_\A}=J$ by \cite[Lemma III.4.1]{Davidson}.

Assume (1). The map $i(\cdot, \mathcal{U})$ is already a well-defined injection by Proposition (\ref{ind-lim-ideals}) and  Proposition (\ref{fell-inverse-lemma}). For surjectivity,  let $(J_n)_{n \in \N} \in \mathrm{Ideal}(\A)_\infty$.  Thus $J_n \in \mathrm{Ideal}(\A_n)$ and $J_n \subseteq J_{n+1}$ for all $n \in \N$, and so if we let $J=\cup_{n \in \N} J_n$, then $J \in \mathrm{algIdeal}(\cup_{n \in \N}\A_n)_{prod}$.  Thus $\overline{J}^{\|\cdot \|_\A} \in \mathrm{Ideal}(\A)$.  We claim that $i\left( \overline{J}^{\|\cdot \|_\A}, \mathcal{U}\right)=(J_n)_{n \in \N}$.  Indeed, define $I_n= \overline{J}^{\|\cdot \|_\A} \cap \A_n$ for each $n \in \N$.  We have $I_n \in \mathrm{Ideal}(\A_n)$ and $I_n \subseteq I_{n+1}$ for all $n \in \N$.  Again, if we let $I=\cup_{n \in \N} I_n$, then $I \in \mathrm{algIdeal}(\cup_{n \in \N}\A_n)_{prod}$ and $\overline{I}^{\|\cdot \|_\A} \in \mathrm{Ideal}(\A)$.  However:
\begin{equation*}
\overline{I}^{\|\cdot \|_\A} =\overline{\cup_{n \in \N} I_n}^{\| \cdot \|_\A} = \overline{\cup_{n \in \N} \left(\overline{J}^{\|\cdot \|_\A} \cap \A_n\right)}^{\| \cdot \|_\A} = \overline{J}^{\|\cdot \|_\A}
\end{equation*}
by \cite[Lemma III.4.1]{Davidson}.  Hence $\cup_{n \in \N} I_n=I=J=\cup_{n \in \N} J_n$ by the assumption that the map of (1) is injective, which implies that  $\overline{J}^{\|\cdot \|_\A}\cap \A_n=I_n=J_n$ for all $n \in \N$.  Thus $i\left( \overline{J}^{\|\cdot \|_\A}, \mathcal{U}\right)=(I_n)_{n \in \N}=(J_n)_{n \in \N}$, which completes this direction.

Next, assume (2).  Let $I,J \in \mathrm{algIdeal}(\cup_{n \in \N}\A_n)_{prod}$ such that $I \neq J$.  Thus there exists $N \in \N$ such that $I \cap \A_N \neq J \cap \A_N.$  Hence $\left(I \cap \A_n\right)_{n \in \N} \neq \left(J \cap \A_n\right)_{n \in \N}$ where $\left(I \cap \A_n\right)_{n \in \N},\left(J \cap \A_n\right)_{n \in \N} \in \mathrm{Ideal}(\A)_\infty.$  By assumption that the map of (2) is a surjection, there exist $K_I,  K_J \in \mathrm{Ideal}(\A)$ such that $K_I\neq  K_J$ and  $i(K_I, \mathcal{U})=\left(I \cap \A_n\right)_{n \in \N}$ and $i(K_J, \mathcal{U})=\left(J \cap \A_n\right)_{n \in \N}$ since $i(\cdot, \mathcal{U})$ is well-defined.  However, this implies that $K_I \cap \A_n = I\cap \A_n$ and $K_J \cap \A_n = J \cap \A_n$ for all $n \in \N$.  But, then:
\begin{equation*} 
K_I= \overline{\cup_{n \in \N} (K_I \cap \A_n)}^{\|\cdot \|_\A} = \overline{\cup_{n \in \N} (I \cap \A_n)}^{\|\cdot \|_\A}= \overline{I}^{\|\cdot \|_\A}
\end{equation*}
by \cite[Lemma III.4.1]{Davidson}.
Similarly, we have $K_J=\overline{J}^{\|\cdot \|_\A}$, and therefore $\overline{I}^{\|\cdot \|_\A} \neq \overline{J}^{\|\cdot \|_\A},$ which completes the proof of the equivalence between (1) and (2).

Finally, assume $\A$ is AF and $\mathcal{U}$ is a family of finite-dimensional C*-algebras.  If $J \in \mathrm{algIdeal}(\cup_{n \in \N} \A_n)$, then $J\cap \A_n$ is finite-dimensional and thus closed for all $n \in \N$.  Hence $J \cap \A_n \in \mathrm{Ideal}(\A_n)$ for all $n \in \N$.  Therefore $\mathrm{algIdeal}(\cup_{n \in \N} \A_n) = \mathrm{algIdeal}(\cup_{n \in \N} \A_n)_{prod}$. However by \cite[Theorem 3.3]{Bratteli72}, the map $J \in \mathrm{algIdeal}(\cup_{n \in \N} \A_n )\mapsto  \overline{J}^{\|\cdot \|_\A} \in \mathrm{Ideal}(\A)$ is a bijection onto $\mathrm{Ideal}(\A)$, which completes the proof by the established equivalence of (1) and (2).
\end{proof}
We will now see that Lemma (\ref{i-surj-lemma}) produces a natural sufficient conditions for our topology to agree with the Fell topology.
\begin{theorem}\label{Fell-ind-top-thm}
Let $\A$ be a C*-algebra with a non-decreasing sequence of C*-subalgebras $\mathcal{U}= (\A_n)_{n \in \N}$ such that $\A=\overline{\cup_{n \in \N} \A_n}^{\Vert \cdot \Vert_\A}$.  

Using Notation (\ref{algideal-not}), if the function $J \in \mathrm{algIdeal}(\cup_{n \in \N}\A_n)_{prod} \mapsto \overline{J}^{\|\cdot \|_\A} \in \mathrm{Ideal}(\A)$ is injective, and thus a bijection onto $\mathrm{Ideal}(\A)$, then the topology $Fell_{i(\mathcal{U})}$ of Definition (\ref{Fell-ind-top-def}) agrees with the topology $Fell$ on $\mathrm{Ideal}(\A)$.

In particular, if $\A$ is AF and $\mathcal{U}$ is chosen to be a family of finite-dimensional C*-algebras, then the topology $Fell_{i(\mathcal{U})}$ agrees with the topology $Fell$ on $\mathrm{Ideal}(\A)$.
\end{theorem}
\begin{proof}
By Lemma (\ref{i-surj-lemma}), the map $i(\cdot, \mathcal{U})$ is a bijection onto $\mathrm{Ideal}(\A)_\infty$.  Hence, since the topological space  $(\mathrm{Ideal}(\A)_\infty, Fell_\infty)$ is compact Hausdorff by Lemma (\ref{fell-inverse-lemma}), then $Fell_{i(\mathcal{U})}$ is compact Hausdorff on $\mathrm{Ideal}(\A)$ since it is the intial topology induced by a bijection onto a compact Hausdorff space.  However, by Proposition (\ref{Fell-topology-contain}) and maximal compactness of Hausdorff spaces, the proof is complete.
\end{proof}

The injection of the above proposition will allows us  to form  metrics on $\mathrm{Ideal}(\A)$  using metrics on $\prod_{n \in \N} \mathrm{Ideal}(\A_n)$.  In most cases of inductive limits, we know much more about the structure of the $\A_n$ than the inductive limit.  The main consequence we have of this will be that the metric formed on $\mathrm{Ideal}(\A)$  using a metric on $\prod_{n \in \N} \mathrm{Ideal}(\A_n)$ will metrize the Fell topology of $\mathrm{Ideal}(\A)$ in the AF case. However, the Fell topology is metrizable when $\A$ is separable (we state this fact in Lemma (\ref{fell-metrize-lemma})), yet this metric follows from a metrization theorem and will not be of use to us --- especially when considering fusing families of ideals. Thus, we introduce a possible candidate for a metric on the ideal space of a separable inductive limit built by the inductive sequence, which will metrize the Fell topology in the AF case. Thus, the following results are in more general terms than AF and initially motivated our pursuit of a metric compatible with fusing families on the Fell topology of AF algebras.

\begin{lemma}\label{fell-metrize-lemma}
If $\A$ is a separable C*-algebra, then the Fell topology on $\mathrm{Ideal}(\A)$ is compact metrizable.
\end{lemma}
\begin{proof}
 The Fell topology on $\mathrm{Ideal}(\A)$ is already compact Hausdorff (Definition (\ref{Fell-topology})).  Since $\A$ is separable, the Jacobson topology on $\mathrm{Prim}(\A)$ is second countable by \cite[Corollary 4.3.4]{Pedersen}.  However, by \cite[(III) pg. 474]{Fell62}, the Fell topology has a countable basis. Thus, the Fell topology is second countable compact Hausdorff, which completes the proof by \cite[Urysohn's metrization theorem 23.1]{Willard}.
\end{proof}

\begin{proposition}\label{fell-prod-metric-prop}
If $\A$ is a separable C*-algebra with a non-decreasing sequence of C*-subalgebras $\mathcal{U}=(\A_n)_{n \in \N}$ such that $\A=\overline{\cup_{n \in \N} \A_n}^{\| \cdot \|_\A}$, then for each $n \in \N$, the Fell topology on $\mathrm{Ideal}(\A_n)$ is metrized by a  metric $d_n$ with diameter at  most $1$ and the $[0,\infty)$-valued map on $\prod_{n \in \N}\mathrm{Ideal}(\A_n) \times  \prod_{n \in \N} \mathrm{Ideal}(\A_n)$ defined by:
\begin{equation*}
d((I_n)_{n \in \N},(J_n)_{n \in \N})=\sum_{n=0}^\infty \frac{d_n(I_n, J_n)}{2^n}
\end{equation*}
is a compact metric on the product topology of $\prod_{n \in \N}\mathrm{Ideal}(\A_n)$ with respect to the Fell topology on each $\mathrm{Ideal}(\A_n)$ and induces a totally bounded metric on $\mathrm{Ideal}(\A)$ defined by:
\begin{equation*}  \mathsf{m}_{\prod(Fell), \mathcal{U}}(I,J)=\sum_{n=0}^\infty \frac{d_n(I \cap \A_n, J\cap \A_n)}{2^n},
\end{equation*}
which metrizes the topology $Fell_{i(\mathcal{U})}$ on $\mathrm{Ideal}(\A)$ of Definition (\ref{Fell-ind-top-def}).
\end{proposition}
\begin{proof}
Since $\A$ is separable, the subspace $\A_n$ is separable for all $n \in \N$.  Thus, by Lemma (\ref{fell-metrize-lemma}), we have that the Fell topology of each $\mathrm{Ideal}(\A_n)$ is metrized by some metric $d_n$.  If $d_n$ has diameter more than $1$, then simply use the metric $\frac{d_n}{1+d_n}$ instead, which metrizes the same topology and has diameter at most $1$,  and thus the metric $d$ defined in the statement of the proposition metrizes the product topology. The fact that $\mathsf{m}_{\prod(Fell), \mathcal{U}}$ is a totally bounded metric follows from the fact that $\mathsf{m}_{\prod(Fell), \mathcal{U}}=d \circ (i(\cdot, \mathcal{U}) \times i(\cdot, \mathcal{U}))$ and $i(\cdot, \mathcal{U}) $ is an injection by Proposition (\ref{ind-lim-ideals}). The fact that $\mathsf{m}_{\prod(Fell), \mathcal{U}}$ metrizes $Fell_{i(\mathcal{U})}$ follows by construction.
\end{proof}
In Theorem (\ref{AF-ideal-metric-thm}), we will show that the metric $\mathsf{m}_{\prod(Fell), \mathcal{U}}$ above metrizes the Fell topology when $\A$ is AF and $ \mathcal{U}$ contains only finite-dimensional C*-algebras.  The next Corollary shows that we can simplify the metric $d$ and thus $\mathsf{m}_{\prod(Fell), \mathcal{U}}$, when $\A$ is AF.
\begin{corollary}\label{ind-lim-metric-cor}
If $\A$ is a separable C*-algebra with a non-decreasing sequence of C*-subalgebras $\mathcal{U}=(\A_n)_{n \in \N}$ such that $\A=\overline{\cup_{n \in \N} \A_n}^{\| \cdot \|_\A}$ and $\mathrm{Ideal}(\A_n)$ is finite for each $n \in \N$, then the compact metric 
product topology of $\prod_{n \in \N}\mathrm{Ideal}(\A_n)$ with respect to the Fell topology on each $\mathrm{Ideal}(\A_n)$ is metrized by the metric:
\begin{equation*}
d_{i(\mathcal{U})}((I_n)_{n \in \N}, (J_n)_{n \in \N})=\begin{cases}
0 & \text{: if }  I_n=J_n  \text{ for all } n \in \N,\\
2^{-\min \{m \in \N : I_m \neq J_m \}} & \text{: otherwise }
\end{cases}
\end{equation*}
and induces a totally bounded metric on $\mathrm{Ideal}(\A)$ defined by:
\begin{equation*}  \mathsf{m}_{i (\mathcal{U})}(I,J)=\begin{cases}
0 & \text{: if } I \cap \A_n =J \cap \A_n \text{ for all } n \in \N\\
2^{-\min \{m \in \N : I\cap \A_m \neq J\cap \A_m \}} & \text{: otherwise }
\end{cases}
\end{equation*}
that metrizes the same topology of $ \mathsf{m}_{\prod(Fell), \mathcal{U}}$ of Propsition (\ref{fell-prod-metric-prop}) on $\mathrm{Ideal}(\A)$ and the topology $Fell_{i(\mathcal{U})}$ of Definition (\ref{Fell-ind-top-def}).
\end{corollary}
\begin{proof}
Since the Fell topology is always compact Hausdorff, the topology $\mathrm{Ideal}(\A_n)$ is discrete as the set is finite, and thus we may take our metrics $d_n$ from the previous proposition to be the discrete metric (that assigns $1$ to distinct points) for all $n \in \N$.   Finally, the topology given by $d_{i(\mathcal{U})}$ and $d$ of Theorem (\ref{fell-prod-metric-prop})  on $\prod_{n \in \N} \mathrm{Ideal}(\A_n)$ agree in this setting as these metrics are equivalent, which completes the proof by construction of $Fell_{i(\mathcal{U})}.$
\end{proof}
Now, we may prove a main result of this section in which the metric of the previous corrollary does in fact metrize the Fell topology for AF algebras.
\begin{theorem}\label{AF-ideal-metric-thm}
If $\A$ is an AF algebra, then for any non-decreasing sequence of finite-dimensional C*-subalgebras $\mathcal{U}=(\A_n)_{n \in \N}$ such that $\A=\overline{\cup_{n \in \N} \A_n}^{\| \cdot \|_\A}$, the metric $\mathsf{m}_{i(\mathcal{U})}$ of Corollary (\ref{ind-lim-metric-cor})  metrizes the Fell topology on $\mathrm{Ideal}(\A)$.
\end{theorem}
\begin{proof}
Observe that finite-dimensional C*-algebras have finitely many ideals and apply Theorem (\ref{Fell-ind-top-thm}) to Corollary (\ref{ind-lim-metric-cor}). 
\end{proof}
An immediate consequence of Theorem (\ref{AF-ideal-metric-thm}) is that, although the metric is built using a fixed inductive sequence, the metric topology with respect to an inductive sequence is homeomorphic to the metric topology on the same AF algebra with respect to any other inductive sequence. In particular, concerning continuity or convergence results, Corollary (\ref{ideal-metric-af-homeo}) provides that one need not worry about the possibility of choosing the wrong inductive sequence, and therefore, one may choose any inductive sequence without worry to suit the needs of the problem at hand.

\begin{corollary}\label{ideal-metric-af-homeo}
Let $\A, \B$ be AF algebras and fix any non-decreasing sequences of finite dimensional C*-subalgebras  $\mathcal{U}_\A =(\A_n)_{n \in \N}, \mathcal{U}_\B = (\B_n)_{n \in \N},$ respectively, such that $\A=\overline{\cup_{n \in \N} \A_n }^{\Vert \cdot \Vert_\A}$ and $\B= \overline{\cup_{n \in \N} \B_n }^{\Vert \cdot \Vert_\B}$.  

If $\A$ and $\B$ are *-isomorphic, then the metric spaces $\left(\mathrm{Ideal}(\A), \mathsf{m}_{i\left(\mathcal{U}_\A \right)}\right)$ and \\ $\left(\mathrm{Ideal}(\B), \mathsf{m}_{i\left(\mathcal{U}_\B\right)}\right)$ are homeomorphic.

In particular, if $\A$ is AF and $\A=\overline{\cup_{n \in \N} \A_{1,n} }^{\Vert \cdot \Vert_\A}= \overline{\cup_{n \in \N} \A_{2,n} }^{\Vert \cdot \Vert_\A}$, where $\mathcal{U}_1 = (\A_{1,n})_{n \in \N},$ $ \mathcal{U}_2 = (\A_{2,n})_{n \in \N}$ are any non-decreasing sequences of finite dimensional C*-subalgebras of $\A$, then the metric spaces $\left(\mathrm{Ideal}(\A), \mathsf{m}_{i\left(\mathcal{U}_1 \right)}\right)$ and $\left(\mathrm{Ideal}(\A), \mathsf{m}_{i\left(\mathcal{U}_2\right)}\right)$ are homeomorphic.
\end{corollary}
\begin{proof}
By construction of the Fell topology and the Jacobson topology Definition (\ref{Fell-topology}), if $\A$ and $\B$ are *-isomorphic then the Fell topologies are homeomorphic.  Thus, the conclusion follows by Theorem (\ref{AF-ideal-metric-thm}).  
\end{proof}

In the context of this paper, a main motivation for the metric of Corollary (\ref{ind-lim-metric-cor}) is to provide a fusing family of quotients via convergence of ideals. First, for a fixed ideal of an inductive limit of the form $\A=\overline{\cup_{n \in \N} \A_n}^{\Vert \cdot \Vert_\A}$ , we provide an inductive limit in the sense of Notation (\ref{ind-lim}) that is *-isomorphic to the quotient.  The reason for this is that given $I \in \mathrm{Ideal}(\A)$, then $\A/I$ has a canonical closure of union form as $\A/I=\overline{\cup_{n \in \N}((\A_n+I)/I)}^{\Vert \cdot \Vert_{\A/I}}$ (see Proposition (\ref{ind-lim-quotients-prop})), but if two ideals satisfy $I \cap \A_n =J \cap \A_n $ for some $n \in \N$, then even though this provides that $(\A_n +I)/I $ is *-isomorphic to $(\A_n +J)/J$ as they are both *-isomorphic to $(\A_n/(I \cap \A_n))$ (see Proposition (\ref{ind-lim-quotients-prop})) , the two algebras $(\A_n +J)/J$ and $(\A_n +I)/I$ are not equal in any  way if $I \neq J$, yet,  equality is a requirement for fusing families (see Definition (\ref{coll-def})). Thus, Notation (\ref{ind-lim-quotient}) will allow us to present, up to *-isomorphism, quotients as IL-fusing families from convergence of ideals in the metric of Corollary (\ref{ind-lim-metric-cor})  as we will see in Proposition (\ref{ind-lim-quotients-prop}).

Before we move to fusing families of quotients, we show that a fusing family of ideals is equivalent to convergence in the metric on ideals of Corollary (\ref{ind-lim-metric-cor}).  
\begin{lemma}\label{metric-fusing-equiv-lemma}
Let $\A$ be AF algebra and fix any non-decreasing sequence of finite dimensional C*-subalgebras $\mathcal{U}=(\A_n)_{n \in \N}$ such that $\A=\overline{\cup_{n \in \N} \A_n}^{\Vert \cdot \Vert_\A}$.

If $\left(I^k\right)_{k \in \overline{\N}} \subseteq \mathrm{Ideal}(\A)$, then the following are equivalent:
\begin{enumerate}
\item $\left\{ I^k=\overline{\cup_{n \in \N} I^k \cap \A_n}^{\Vert \cdot \Vert_\A} : k \in \overline{\N} \right\}$ is a fusing family of Definition (\ref{coll-def}),
\item $\left(I^k\right)_{k \in \N}$ converges to $I^\infty $ with respect to the metric $\mathsf{m}_{i(\mathcal{U})}$ of Corollary (\ref{ind-lim-metric-cor}), 
\item $\left(I^k\right)_{k \in \N}$ converges to $I^\infty $ in the Fell topology.
\end{enumerate}
\end{lemma}
\begin{proof}
We begin with the forward direction.  Assume that  $\left(I^k\right)_{k \in \N} \subseteq \mathrm{Ideal}(\A)$ converges to $I^\infty \in \mathrm{Ideal}(\A)$  with respect to $\mathsf{m}_{i(\mathcal{U})}$, which is equivlent to convergence in $Fell$ by Theorem (\ref{AF-ideal-metric-thm}).   Thus, we have  $\lim_{k \to \infty} \mathsf{m}_{i(\mathcal{U})}\left(I^k,I^\infty \right)=0$. From this, we construct an increasing sequence $(c_n)_{n \in \N} \subseteq \N \setminus \{0\}$ such that: 
\begin{equation*}
\mathsf{m}_{i(\mathcal{U})}\left(I^k,I^\infty \right) \leq 2^{-(n+1)}
\end{equation*}
for all $k \geq c_n $.  In particular, fix $N \in \N$, if $k\in \N_{\geq c_N}$, then  $I^k \cap \A_n =I^\infty \cap \A_n$  for all  $n \in \{0, \ldots, N\}$, which implies that $\left\{ I^k=\overline{\cup_{n \in \N} I^k \cap \A_n}^{\Vert \cdot \Vert_\A} : k \in \overline{\N} \right\}$ is a fusing family with fusing sequence $(c_n)_{n \in \N}$ by Definition (\ref{coll-def}).

For the other direction, assume that $\left\{ I^k=\overline{\cup_{n \in \N} I^k \cap \A_n}^{\Vert \cdot \Vert_\A} : k \in \overline{\N} \right\}$ is a fusing family with fusing sequence $(c_n)_{n \in \N}$.  Therefore, for all $N \in \N$, if $k \in \N_{\geq c_N}$, then  $I^k \cap \A_n =I^\infty \cap \A_n$ for all $n \in \{0, \ldots, N\}$.  Hence, 
let $\varepsilon >0.$ There exists $N \in \N$ such that $2^{-N} < \varepsilon $.  If $k \geq c_N \in \N$, then 
\begin{equation*}
\mathsf{m}_{i(\mathcal{U})}\left(I^k,I^\infty \right) \leq 2^{-(N+1)}< 2^{-N} < \varepsilon,
\end{equation*}
which completes the proof.
\end{proof}
\begin{remark}
Clearly, the metric $\mathsf{m}_{i(\mathcal{U})}$ of Corollary (\ref{ind-lim-metric-cor}) can be defined on any C*-inductive limit even without the assumption of AF or separability.  And, in general, this metric would produce an even finer topology than $Fell_{i(\mathcal{U})}$ as $\mathsf{m}_{i(\mathcal{U})}$ is given by a metric on the product topology induced by the discrete topology on the ideal space of each $\A_n$.  Furthermore, we  note the the equivalence between (1) and (2) in Lemma (\ref{metric-fusing-equiv-lemma}) would still hold for this metric in this more general setting.  This connection with fusing families was another strong motivation for the pursuit of this metric.
\end{remark}

\begin{notation}\label{ind-lim-quotient}
Let $\A$ be a C*-algebra with a non-decreasing sequence of C*-sub\-algebras $\mathcal{U}= (\A_n)_{n \in \N}$ such that $\A=\overline{\cup_{n \in \N} \A_n}^{\Vert \cdot \Vert_\A}$.  Let $I \in \mathrm{Ideal}(\A)$.  For $n \in \N$:
\begin{equation*}\gamma_{I,n} : a+I\cap \A_n  \in \A_n/( I\cap \A_n) \longmapsto a+(I\cap \A_{n+1}) \in \A_{n+1}/(I \cap \A_{n+1}) , 
\end{equation*}
 is a*-monomorphism by the same argument of Claim (\ref{quo-ind-map}) and $\mathcal{U}$ is non-decreasing.  Let $\mathcal{I} (\A/I) = \left(\A_n/( I\cap \A_n), \gamma_{I,n} \right)_{n \in \N}$, and denote the C*-inductive limit by $\underrightarrow{\lim}\ \mathcal{I}(\A/ I). $
 
 If  $\B \subseteq \A$ be a C*-subalgebra and $I\in \mathrm{Ideal}(\A)$, then denote: 
$$ \B+I=\overline{\{ b+c :b \in \B, c \in I\}}^{\Vert \cdot \Vert_\A}.$$
\end{notation}

\begin{proposition}\label{ind-lim-quotients-prop}
Let $\A$ be AF  and fix any non-decreasing sequence of finite-dimensional C*-subalgebras $\mathcal{U}= (\A_n)_{n \in \N}$  of $\A$ such that $\A=\overline{\cup_{n \in \N} \A_n}^{\Vert \cdot \Vert_\A}$. 

Using Notation (\ref{ind-lim-quotient}), if $I \in \mathrm{Ideal}(\A)$, then there exists a *-isomorphism $$\phi_I : \underrightarrow{\lim}\ \mathcal{I}(\A/ I) \rightarrow \A/I$$ such that for all $n \in \N$ the following diagram commutes:
\begin{displaymath}
\xymatrix{
 \A_n/( I\cap \A_n)   \ar[r]^{\indmor{\gamma_I}{n}} \ar[dr]_{\phi_I^n} &   \underrightarrow{\lim}\ \mathcal{I}(\A/ I) \ar[d]^{\phi_I}\\
  & \A/I  
},
\end{displaymath} where for all $n \in \N$, the maps  $\phi_I^n : a+( I\cap \A_n) \in \A_n/( I\cap \A_n) \longmapsto a+I \in (\A_n +I)/I \subseteq \A/I$ are *-monomorphisms onto  $(\A_n +I)/I$, in which  $\A_n +I =\{ a+b \in \A: a \in \A_n, b \in I\}$ is a C*-subalgebra of $\A$ containing $I$ as an ideal and $\cup_{n \in \N} ((\A_n +I)/I) $ is a dense *-subalgebra of $\A/I$ with $((\A_n +I)/I)_{n \in \N}$  non-decreasing.

Furthermore, if $(I^k)_{k \in \N} \subseteq \mathrm{Ideal}(\A)$ converges to $I^\infty \in \mathrm{Ideal}(\A)$  with respect to $\mathsf{m}_{i(\mathcal{U})}$ of Corollary (\ref{ind-lim-metric-cor}) or the Fell topololgy, then using Definition (\ref{coll-def}),  we have \\
$\left\{ I^k=\overline{\cup_{n \in \N} I^k \cap \A_n}^{\Vert \cdot \Vert_\A} : k \in \overline{\N} \right\}$ is a fusing family with respect to some fusing sequence $(c_n)_{n \in \N}$ such that $\left\{\underrightarrow{\lim}\ \mathcal{I}\left(\A/ I^k \right) : k \in \overline{\N} \right\}$ is an IL-fusing family  with fusing sequence $(c_n)_{n \in \N}$.
\end{proposition}

\begin{proof}
Let $I \in \mathrm{Ideal}(\A)$.  Fix $n \in \N$.  Note that $\A_n +I$ is a C*-subalgebra of $\A$ since $I\in \mathrm{Ideal}(\A)$, and furthermore $I \in \mathrm{Ideal}(\A_n +I)$. Now, we have $\A_n +I=\{a+b \in \A : a \in \A_n, b \in I\}$ since $\A_n$ and $I$ are both closed in $\A$ and $\A_n$ is finite dimensional.  Next, we have $\phi_I^n$ is an injective *-homomorphism by Claim (\ref{quo-ind-map}).  If   $a \in \A_n $, then $\phi_I^n (a+\A_n/( I\cap \A_n))= a+I$ and the composition $\phi_I^{n+1}(\gamma_{I,n}(a+( I\cap \A_n)))=\phi_I^{n+1}(a+( I\cap \A_{n+1}))=a+I$.  Hence, for all $n \in \N$, the following diagram commutes:

\begin{displaymath}
\xymatrix{
 \A_n/( I\cap \A_n)   \ar[r]^(.4){\gamma_{I,n}} \ar[dr]_{\phi_I^n} &   \A_{n+1}/( I\cap \A_{n+1}) \ar[d]^{\phi_I^{n+1}}\\
  & \A/I  
}.
\end{displaymath}
Thus, by \cite[Theorem 6.1.2]{Murphy}, the definition of inductive limit \cite[Chapter 6.1]{Murphy}, and the fact that each map in the above diagram is an isometry, there exists a unique isometric *-homomorphism (and thus a *-monomorphism) $\phi_I  : \underrightarrow{\lim}\ \mathcal{I}(\A/ I) \rightarrow \A/I$ such that for all $n \in \N$ the diagram in the statement of this theorem commutes.  

Next, fix $n \in \N$.  Let $x \in (\A_n +I)/I$, and so  $x=a+b+I$, where $a \in \A_n, b \in I$.  Thus, we have $a+b-a=b \in I \implies x-(a+I)=0+I\implies x=a+I$.  But, then, the image $\phi^n_I (a+(I\cap \A_n))=x$.  Hence, the map $\phi^n_I $ is onto $(\A_n +I)/I$. We thus have: \begin{equation*}\phi_I \left(\cup_{n \in \N} \indmor{\gamma_I}{n}( \A_n/( I\cap \A_n) )\right) = \cup_{n \in \N} \left((\A_n+I)/I\right),
 \end{equation*} in which the right-hand side is a dense *-subalgebra of $\A/I$ by continuity of the quotient map and the assumption that $\cup_{n \in \N} \A_n$ is dense in $\A$.  Hence, since the normed space $\underrightarrow{\lim}\ \mathcal{I}(\A/ I) $ is complete and $\phi_I$ is a linear isometry on $\underrightarrow{\lim}\ \mathcal{I}(\A/ I) $, we have $\phi_I$ surjects  onto $\A/I$.  Thus, the function $\phi_I : \underrightarrow{\lim}\ \mathcal{I}(\A/ I) \rightarrow \A/I$ is a *-isomorphism. 
 
Next, assume that  $\left(I^k \right)_{n \in \N} \subseteq \mathrm{Ideal}(\A)$ converges to $I^\infty \in \mathrm{Ideal}(\A)$  with respect to $\mathsf{m}_{i(\mathcal{U})}$, which is equivalent to convergence in $Fell$ by Theorem (\ref{AF-ideal-metric-thm}).  By Lemma (\ref{metric-fusing-equiv-lemma}), the family $\left\{ I^k=\overline{\cup_{n \in \N} I^k \cap \A_n}^{\Vert \cdot \Vert_\A} : k \in \overline{\N} \right\}$ is a fusing family with fusing sequence $(b_n)_{n \in \N}$ by Definition (\ref{coll-def}).

Let $c_n =  b_{n+ 1}$ for all $n \in \N$. Then, the sequence  $(c_n)_{n \in \N}$  is a fusing sequence for  $\left\{ I^k=\overline{\cup_{n \in \N} I^k \cap \A_n}^{\Vert \cdot \Vert_\A} : k \in \overline{\N} \right\}$.  Fix $N \in \N, n \in \{0, \ldots, N\}$, and $k \in \N_{\geq c_N}$. Then, the equality $I^k \cap \A_n =I^\infty \cap \A_n$ implies that $\A_n /(I^k \cap \A_n)=\A_n / (I^\infty \cap \A_n)$.  But, also, we gather $\gamma_{I^k,n} = \gamma_{I^\infty,n}$ since  $\A_{n+1} /(I^k \cap \A_{n+1})=\A_{n+1} / (I^\infty \cap \A_{n+1})$ as $c_n = b_{n+1} $.  Hence,  the familiy of inductive limits $\left\{\underrightarrow{\lim}\ \mathcal{I}\left(\A/ I^k \right) : k \in \overline{\N} \right\}$ is an IL-fusing family with fusing sequence $(c_n)_{n \in \N} $. 
\end{proof}

Now, that we have this identification with our metric and the Fell topology, we finish our discussion of the metric topology by considering it in the unital commutative case of AF algebras in Corollary (\ref{comm-metric}).  It will be the case that on the primitive ideals, the relative metric topology of Corollary (\ref{ind-lim-metric-cor}), the relative Fell topology, and the Jacobson topology all agree on the primitive ideals. However, we begin with a more general scenario, which we only assume that the Jacobson topology is Hausdorff on a unital C*-algebra since in this case the relative Fell topology  and the Jacobson topology all agree on the primitive ideals. First, a remark on restricitng to the unital case.

\begin{remark} 
In the following results of this section, we restrict our attention to unital C*-algebras since in this case  $\mathrm{Prim}(\A)$ is a compact subset of the Fell topology as seen in Lemma (\ref{hausdorff-fell-jacobson}). However, although the Jacobson topology is still locally compact  in the non-unital case \cite[Corollary 3.3.8]{Dixmier} and one can form the Alexandroff compactification in the Hausdorff case of the Jacobson topology, the fact that $\A \in \mathrm{Ideal}(\A)$ (note that $\A$ plays the role of the point at infinity of the Alexandroff compactification by Definition (\ref{Fell-topology}) and \cite[Corollary (1) pg. 475]{Fell62})  may not be isolated in the Fell topology in general diminshes any reasonable expectation that the relative Fell topology on $\mathrm{Prim}(\A)$ would agree with the Jacobson topology in this generality. An example of when $\A$ is not isolated in the Fell topology is when $\A=C_0(Y)$, where $Y=\{1/n \in \R: n \in \N \setminus \{0\}\}\subset (0,1]$. Indeed,  if we define for all $m \in \N$  the ideal $I_m= \{g \in \A : g(\{1/(n+2) \in \R: n \geq m \})=0 \} \subsetneq \A$, then the sequence $(I_m)_{m \in \N} \subset \mathrm{Ideal}(\A) \setminus \{\A\}$ converges to $\A$ in the Fell topology by Lemma (\ref{Fell-conv}) and definition of  $C_0(Y)$.

On the other hand, the element $\A \in \mathrm{Ideal}(\A)$ is always isolated in the Fell topology when $\A$ is unital regardless of any separation condition on the Jacobson topology. Indeed, if $J \in \mathrm{Ideal}(\A) \setminus \{\A\}$, then $\| 1_\A + J \|_{\A/J} \geq 1$ since the set $\{ a \in \A : \|a+1_\A \|_\A <1 \}$ contains only invertible elements by \cite[Corollary VII.2.3]{Conway}. Hence, no net of ideals in $\mathrm{Ideal}(\A) \setminus \{\A\}$ may converge to $\A$ by Lemma (\ref{Fell-conv}) since $\|1_\A + \A\|_{\A/\A}=0$.
\end{remark}

Before we move to the C*-algebra setting, we present a fact about the Fell topology in the context of topological spaces.  The following is mentioned in \cite{Fell62}, but we provide a detailed proof now.
\begin{lemma}\label{Fell-Jacobson-Haus-topological-lemma}
If $(X, \tau)$ is a compact Hausdorff space, then the map:
\begin{equation*}
s: x \in X \longmapsto \{x\} \in \mathcal{C}l(X) 
\end{equation*}
is a well-defined homeomorphism onto its image with respect to the relative Fell topology on $\mathcal{C}l(X)$ of Defintion (\ref{topological-Fell-def}), and moreover, the set:
\begin{equation*}
s(X)=\{ \{x\} \in \mathcal{C}l(X): x \in X\}
\end{equation*} 
is a compact and thus a closed subset of $\mathcal{C}l(X)$ with respect to the Fell topology.
\end{lemma}
\begin{proof}
Since $(X, \tau)$  is compact Hausdorff and the space $\mathcal{C}l(X)$ equipped with the Fell topology is compact Hausdorff by Lemma (\ref{topological-Fell-lemma}), we only have to check that $s$ is continuous and note that $s$ is well-defined since $(X, \tau)$  is Hausdorff.

Let $(x_\lambda)_{\lambda \in \Lambda} \subseteq X$ be a net that converges to some $x \in X$ with respect to the topology $\tau$. We claim that $(\{x_\lambda\})_{\lambda \in \Lambda} \subseteq \mathcal{C}l(X)$ converges to $\{x\} \in \mathcal{C}l(X)$ with respect to the Fell topology. 

Let $K \subseteq X$ be a compact set with respect to $\tau$ and let $n \in \N$ and $A_0, \ldots, A_n \in \tau \setminus \{\emptyset \}$ and let $F=\cup_{j=0}^n \{A_j\} \subseteq \tau$. Assume that $\{x\} \in U(K,F)=\{ Y \in \mathcal{C}l(X): Y\cap K= \emptyset \ \land \ Y\cap A_j \neq \emptyset \text{ for all } j \in \{0, \ldots,n\}\}$.  Thus:
\begin{equation*}x \in (X \setminus K) \cap \left(\bigcap_{j=0}^n A_j\right) \in \tau
\end{equation*}
since $K$ is closed as $(X,\tau)$ is Hausdorff.  Therefore there exists $\alpha \in \Lambda$ such that for all $\lambda \geq \alpha$, we have that:
\begin{equation*}x_\lambda \in (X \setminus K) \cap \left(\bigcap_{j=0}^n A_j\right) \in \tau,
\end{equation*}
which implies that $\{x_\lambda\} \in U(K,F)$ for all $\lambda \geq \alpha$, which completes that proof.
\end{proof}
\begin{lemma}\label{hausdorff-fell-jacobson}
If $\A$ is a unital C*-algebra such that $\mathrm{Prim}(\A)$ equipped with its Jacobson topology is Hausdorff, then
on $\mathrm{Prim}(\A)$, the relative Fell topology agrees with the Jacobson topology and $\mathrm{Prim}(\A)$ is a compact and thus closed subset of the Fell topology.
\end{lemma}
\begin{proof}
By the fact that the Jacobson topology on $\mathrm{Prim}(\A)$ is compact in the unital case \cite[Proposition 3.1.8]{Dixmier} and Lemma (\ref{Fell-Jacobson-Haus-topological-lemma}), we have that the map:
\begin{equation*}s: P \in (\mathrm{Prim}(\A), Jacobson) \mapsto \{P\} \in \left(\mathcal{C}l(\mathrm{Prim}(\A)), \tau_{\mathcal{C}l(\mathrm{Prim}(\A))}\right)
\end{equation*}  is a well-defined  homeomorphism onto its image with respect to the relative topology such that  $s(\mathrm{Prim}(\A))\subset \mathcal{C}l(\mathrm{Prim}(\A)$ is compact  and thus closed in the topology $ \tau_{\mathcal{C}l(\mathrm{Prim}(\A))}$, and note that $\left(\mathcal{C}l(\mathrm{Prim}(\A)), \tau_{\mathcal{C}l(\mathrm{Prim}(\A))}\right)$ is also compact.

Next, let $P \in \mathrm{Prim}(\A)$.  Since the Jacobson topology is Hausdorff, we have that $\{P\}$ is closed in the Jacobson topology.  Hence, by Definition (\ref{Fell-topology}), there exists a unique ideal $I \in \mathrm{Ideal}(\A)$ such that $fell(I)=\{P\}$.  However,  \cite[Theorem 5.4.3]{Murphy} implies that $I=\cap_{J \in fell(I)} J= P,$ and thus $fell(P)=\{P\}$ for all $P \in \mathrm{Prim}(\A)$.
 Hence since $fell$ is a bijection, we gather that:
\begin{equation*}
fell^{-1} \left( \{ \{J\} \in \mathcal{C}l(\mathrm{Prim}(\A)): J \in \mathrm{Prim}(\A)\} \right)=\mathrm{Prim}(\A),
\end{equation*}
Hence, the map: 
\begin{equation*}
fell^{-1} \circ s : P \in (\mathrm{Prim}(\A),  Jacobson) \mapsto P \in  (\mathrm{Prim}(\A),  Fell)
\end{equation*} is a homeomorphism onto $\mathrm{Prim}(\A)$ since the map $fell$ is a homeomorphism by the end of the proof of Lemma (\ref{Fell-conv}), where $(\mathrm{Prim}(\A), Fell)$ denotes the relative Fell topology on $\mathrm{Prim}(\A)$, which completes the proof. 
\end{proof}

 Before we move provide the final result of this section, we  present a classical result with proof, in which the Jacobson topology on the primitive ideals of a unital commutative $\A$ is homeomorphic to the maximal ideal space with its weak-* topology (this is true, of course, with non-unital as well and the following proof is exactly the same in this case, but we only consider the unital case). Of course, $\mathrm{Prim}(\A)$ is  compact  on any unital C*-algebra (commutative or not)\cite[Proposition 3.1.8]{Dixmier}, so the main purpose of the following theorem is to provide Hausdorff separation in the case of commutativity.

\begin{theorem}\label{comm-Jacobson}
If $\A$ is a unital commutative C*-algebra and $M_\A$ denotes its space of non-zero multiplicative linear functionals with its weak-* topology, then the map:
\begin{equation*}
\varphi \in  M_\A \longmapsto \ker \varphi \in \mathrm{Prim}(\A).
\end{equation*}
 is a homeomorphism onto $\mathrm{Prim}(\A)$ with its Jacobson topology, and therefore $\mathrm{Prim}(\A)$ with its Jacobson topology is a compact Hausdorff space. 
\end{theorem}
\begin{proof} (Thank you to Tristan Bice for notifying me of \cite[Proposition 4.3.3]{Pedersen} and its usefulness, which considerably shortened the proof of this theorem from a previous version of this article).
 By \cite[Theorem 5.4.4]{Murphy},  the set $\mathrm{Prim}(\A) $ is the set of maximal ideals.  However, for all $\varphi \in M_\A$, the ideal $\ker \varphi $ is maximal.  Hence, the map  $\varphi \in  M_\A \longmapsto \ker \varphi \in \mathrm{Prim}(\A)$ is a bijection by \cite[Theorem I.2.5]{Davidson}.   Furthermore, by \cite[Theorem 5.1.6]{Murphy}, the set of pure states on $\A$ is equal to $M_\A$.  Therefore, by \cite[Proposition 4.3.3]{Pedersen}, the map $\varphi \in  M_\A \longmapsto \ker \varphi \in \mathrm{Prim}(\A)$ is a homeomorphism onto $\mathrm{Prim}(\A)$ since it is a continuous and open bijection.  Since $M_\A$ is locally compact Hausdorff by \cite[Corollary I.2.6]{Davidson}, the set $\mathrm{Prim}(\A)$ with its Jacobson topology is a compact Hausdorff space.
\end{proof}

\begin{corollary}\label{comm-metric}
Let $\A$ be a unital AF algebra  and fix any non-decreasing sequence of finite-dimensional C*-subalgebras $\mathcal{U}=(\A_n)_{n \in \N}$ such that $\A=\overline{\bigcup_{n \in \N} \A_n }^{\Vert \cdot \Vert_\A}$.  Let $\left(\mathrm{Prim}(\A), \mathsf{m}_{i(\mathcal{U})}\right)$ denote $\mathrm{Prim}(\A)$ equipped with the relative topology induced by the metric topology of $ \mathsf{m}_{i(\mathcal{U})}$ of Corollary (\ref{ind-lim-metric-cor}).
\begin{enumerate}
\item If the Jacobson topology on $\mathrm{Prim}(\A)$ is Hausdorff, then  $\left(\mathrm{Prim}(\A), \mathsf{m}_{i(\mathcal{U})}\right)$  has the same topology as the Jacobson topology or the relative Fell topology on $\mathrm{Prim}(\A).$
\item If $\A$ is a unital  commutative AF algebra, then  $\left(\mathrm{Prim}(\A), \mathsf{m}_{i(\mathcal{U})}\right)$  is homeomorphic to the space of non-zero multiplicative linear functionals on $\A$ denoted $M_\A$  with its weak-* topology, in which the homeomorphism is given by:
\begin{equation*}
\varphi \in M_\A \longmapsto \ker \varphi \in \mathrm{Prim}(\A).
\end{equation*}
\end{enumerate}
\end{corollary}
\begin{proof}
For (1), combine Theorem (\ref{AF-ideal-metric-thm}) with  Lemma (\ref{hausdorff-fell-jacobson}). For (2), combine Theorem (\ref{AF-ideal-metric-thm}) with Lemma (\ref{hausdorff-fell-jacobson}) and Theorem (\ref{comm-Jacobson}).
\end{proof}

\begin{remark}
The metric of Corollary (\ref{ind-lim-metric-cor}) can be seen as an explicit presentation of a metric on a metrizable topology on ideals presented in \cite{Beckhoff92}, where this metrizable topology is presented only in the case of AF algebras and metrizes the Fell topology in the AF case, which we also proved for the metric of Corollary (\ref{ind-lim-metric-cor}) via a different approach in Theorem (\ref{AF-ideal-metric-thm}) by our inverse limit topology, which thus provides a suitable topology  for the ideal space of any C*-algebra formed by an inductive limit and many possibilities for future study on its own.  Also, we note that the metric of Corollary (\ref{ind-lim-metric-cor})  allows us to explicitly calculate distances between ideals in Remark (\ref{explicit-ideal-metric}), and therefore, make interesting comparisons with certain classical metrics on irrationals, and this metric also serves the purpose of providing fusing families of quotients in Proposition (\ref{ind-lim-quotients-prop}).
\end{remark}

\section{Convergence of Quotients of AF algebras in Quantum Propinquity}\label{conv-quotients-quantum}

In the case of unital AF algebras, we provide criteria for when convergence of ideals in the Fell topology  provides convergence of quotients in the quantum propinquity topology, when the quotients are equipped with faithful tracial states.  But, first, as we saw in  Proposition (\ref{ind-lim-quotients-prop}), it seems that an inductive limit is suitable for describing fusing families with regard to convergence of ideals.  Thus, in order to avoid the notational trouble of too many inductive limits, we will phrase many results in this section in terms of closure of union.   Now, when a quotient has a faithful tracial state, it turns out that the *-isomorphism provided in Proposition (\ref{ind-lim-quotients-prop}) is a quantum isometry  (Theorem-Definition (\ref{def-thm})) between the induced quantum compact metric spaces of Theorem (\ref{AF-lip-norms-thm}) and Theorem (\ref{AF-lip-norms-thm-union}), which preserves the finite-dimensional structure as well in Theorem (\ref{ind-lim-quantum-iso}).  The purpose of this is to apply Theorem (\ref{af-cont}) directly to the quotient spaces.  This utilizes our criteria for quantum isometries between AF algebras in \cite{Aguilar16a}.

\begin{theorem}\label{ind-lim-quantum-iso}
Let $\A$ be a unital AF algebra  with unit $1_\A$ such that $\mathcal{U} = (\A_n)_{n\in\N}$ is an increasing sequence of unital finite dimensional C*-subalgebras such that $\A=\overline{\cup_{n \in \N} \A_n}^{\Vert \cdot \Vert_\A}$  with $\A_0=\C 1_\A .$  Let $I \in \mathrm{Ideal}(\A) \setminus \{\A\}$. By Proposition (\ref{ind-lim-quotients-prop}), the C*-algebra $\A/I =\overline{\cup_{n \in \N}((\A_n +I)/I)}^{\Vert \cdot \Vert_{\A/I}}$ and denote $\mathcal{U}/I = ((\A_n +I)/I)_{n \in \N},$ and note that \\  $(\A_0 + I)/I = \C 1_{\A/I }$.

 If $\A/I$ is equipped with a  faithful tracial state, $\mu$, then using notation from Proposition (\ref{ind-lim-quotients-prop}), the map $\mu \circ \phi_I $ is a faithful traical state on $\underrightarrow{\lim}\mathcal{I}(\A/I)$.

Furthermore, let $\beta : \N \rightarrow (0,\infty)$ have limit $0$ at infinity.  If $\Lip^\beta_{\mathcal{I}(\A/I),\mu \circ \phi_I}$ is the $(2,0)$-quasi-Leibniz Lip norm on $ \underrightarrow{\lim} \mathcal{I}(\A/I)$ given by Theorem (\ref{AF-lip-norms-thm})  and  $\Lip^\beta_{\mathcal{U}/I ,\mu }$ is the $(2,0)$-quasi-Leibniz Lip norm on $\A/I$ given by Theorem (\ref{AF-lip-norms-thm-union}), then:
\begin{equation*} \phi_I^{-1} : \left(\A/I ,\Lip^\beta_{\mathcal{U}/I ,\mu }\right) \rightarrow \left( \underrightarrow{\lim} \mathcal{I}(\A/I) , \Lip^\beta_{\mathcal{I}(\A/I),\mu \circ \phi_I}\right)
\end{equation*} is a quantum isometry of  Theorem-Definition (\ref{def-thm}) and: 
\begin{equation*}
\qpropinquity{} \left(\left( \underrightarrow{\lim} \mathcal{I}(\A/I) , \Lip^\beta_{\mathcal{I}(\A/I),\mu \circ \phi_I}\right), \left(\A/I ,\Lip^\beta_{\mathcal{U}/I ,\mu }\right) \right)=0
\end{equation*}
Moreover, for all $n \in \N$, we have:
\begin{equation*}
\qpropinquity{} \left( \left(\A_n/(I \cap \A_n), \Lip^\beta_{\mathcal{I}(\A/I),\mu \circ \phi_I} \circ \indmor{\gamma_I}{n}\right), \left((\A_n+I)/I , \Lip^\beta_{\mathcal{U}/I ,\mu }\right)\right)=0.
\end{equation*}
\end{theorem}
\begin{proof}
Since $I \neq \A $, the AF algebra $\A/I$ is unital and $(\A_0 + I)/I = \C 1_{\A/I }$ as $\A_0 = \C 1_\A $.  Since $\mu$ is faithful on $\A/I$, we have $\mu \circ \phi_I$ is faithful on $ \underrightarrow{\lim} \mathcal{I}(\A/I)$ since $\phi_I$ is a *-isomorphism by Proposition (\ref{ind-lim-quotients-prop}). 

Using Notation (\ref{ind-lim-quotient}), define $\mathcal{U}(\A/I)=\left(\indmor{\gamma_I}{m}\left(\A_m /(I \cap \A_m)\right)\right)_{m \in \N}$.  By \cite[Chapter 6.1]{Murphy}, the sequence $\mathcal{U}(\A/I)=\left(\indmor{\gamma_I}{m}\left(\A_m /(I \cap \A_m)\right)\right)_{m \in \N}$ is an increasing sequence of unital  finite dimensional C*-subalgebras of  $ \underrightarrow{\lim} \mathcal{I}(\A/I)$ such that:
\begin{equation*} \underrightarrow{\lim} \mathcal{I}(\A/I)= \overline{\cup_{m \in \N}\indmor{\gamma_I}{m}\left(\A_m /(I \cap \A_m)\right)}^{\Vert \cdot \Vert_{\underrightarrow{\lim} \mathcal{I}(\A/I)}}
\end{equation*} and $\indmor{\gamma_I}{0}\left(\A_0 /(I \cap \A_0)\right)=\C1_{\underrightarrow{\lim} \mathcal{I}(\A/I)}.$

 Thus, we may define  $\Lip^\beta_{\mathcal{U}(\A/I),\mu \circ \phi_I}$ on $ \underrightarrow{\lim} \mathcal{I}(\A/I)$ from   Theorem (\ref{AF-lip-norms-thm-union}),    and  $\Lip^\beta_{\mathcal{U}/I ,\mu }$ on $\A/I$ from Theorem (\ref{AF-lip-norms-thm-union}).  

Now, fix $m \in \N$, since $\phi_I \circ \indmor{\gamma_I}{m} = \phi_I^m $ by Proposition (\ref{ind-lim-quotients-prop}), we thus have: 
\begin{equation*}
\indmor{\gamma_I}{m}\left(\A_m /(I \cap \A_m)\right) = \phi_I^{-1} \circ \phi_I^m \left(\A_m /(I \cap \A_m)\right)=   \phi_I^{-1} \left( (\A_m+I)/I \right). 
\end{equation*} 
Since the chosen faithful tracial state on $\underrightarrow{\lim}\mathcal{I}(\A/I)$ is $\mu \circ \phi_I$, we have by \cite[Theorem 5.3]{Aguilar16a} that 
  $\left(\indmor{\gamma_I}{m}\left(\A_m /(I \cap \A_m)\right), \Lip^\beta_{\mathcal{U}(\A/I),\mu \circ \phi_I}\right)$ is quantum isometric to 
$\left((\A_m+I)/I , \Lip^\beta_{\mathcal{U}/I ,\mu }\right)$ by the map $\phi_I^{-1}$ restricted to $(\A_m+I)/I$ for all $m \in \N$.  However, the quantum metric space $\left(\indmor{\gamma_I}{m}\left(\A_m /(I \cap \A_m)\right), \Lip^\beta_{\mathcal{U}(\A/I),\mu \circ \phi_I}\right)$ is quantum isometric to the quantum metric space $\left(\left(\A_m /(I \cap \A_m)\right), \Lip^\beta_{\mathcal{U}(\A/I),\mu \circ \phi_I} \circ \indmor{\gamma_I}{m} \right)$ by the map $\indmor{\gamma_I}{m}$.    Since quantum isometry is an equivlance relation, we conclude that:
\begin{equation*}
\qpropinquity{} \left( \left(\A_m/(I \cap \A_m), \Lip^\beta_{\mathcal{U}(\A/I),\mu \circ \phi_I} \circ \indmor{\gamma_I}{m}\right), \left((\A_m+I)/I , \Lip^\beta_{\mathcal{U}/I ,\mu }\right)\right)=0
\end{equation*} by  Theorem-Definition (\ref{def-thm}).  

Moreover,  \cite[Theorem 5.3]{Aguilar16a} also implies that:
\begin{equation*} \phi_I^{-1} : \left(\A/I ,\Lip^\beta_{\mathcal{U}/I ,\mu }\right) \rightarrow \left( \underrightarrow{\lim} \mathcal{U}(\A/I) , \Lip^\beta_{\mathcal{U}(\A/I),\mu \circ \phi_I}\right)
\end{equation*} is a quantum isometry.  Next, define  $\Lip^\beta_{\mathcal{I}(\A/I),\mu \circ \phi_I}$ from Theorem (\ref{AF-lip-norms-thm}).  By \cite[Proposition 5.2]{Aguilar16a}, we may replace $\Lip^\beta_{\mathcal{U}(\A/I),\mu \circ \phi_I}$ with $\Lip^\beta_{\mathcal{I}(\A/I),\mu \circ \phi_I}$, which completes the proof. 
\end{proof}

Thus, the quantum isometry, $\phi_I$, of Theorem (\ref{ind-lim-quantum-iso}) is in some sense the best one could hope for since it preserves the finite-dimensional approximations in the quantum propinquity.  Next, we give criteria for when a family of quotients converge in the quantum propinquity with respect to ideal convergence.

\begin{theorem}\label{quotients-converge}
Let $\A$ be a unital AF algebra  with unit $1_\A$ such that $\mathcal{U} = (\A_n)_{n\in\N}$ is an increasing sequence of unital finite dimensional C*-subalgebras such that $\A=\overline{\cup_{n \in \N} \A_n}^{\Vert \cdot \Vert_\A}$, with $\A_0=\C 1_\A .$  Let $(I^n)_{n \in \overline{\N}} \subseteq \mathrm{Ideal}(\A)\setminus \{\A\}$ such that $\{\mu_k : \A/I^k \rightarrow \C : k \in \overline{\N}\}$ is a family of faithful tracial states. Let $Q^k : \A \rightarrow \A/I^k $ denote the quotient map for all $k \in \overline{\N}$. If: 
\begin{enumerate}
\item $(I^n)_{n \in \N} \subseteq \mathrm{Ideal}(\A)$ converges to $I^\infty \in \mathrm{Ideal}(\A)$ with respect to $\mathsf{m}_{i(\mathcal{U})}$ of Corollary (\ref{ind-lim-metric-cor}) or the Fell topology (Definition (\ref{Fell-topology})) with fusing sequence $(c_n)_{n \in \N}$ for the fusing family $\left\{ I^n=\overline{\cup_{k \in \N} I^n \cap \A_k}^{\Vert \cdot \Vert_\A} : n \in \overline{\N} \right\}$,
\item for each $N\in \N$, we have that $\left( \mu_k \circ Q^k \right)_{k \in  \N_{\geq c_N}}$ converges to $\mu_\infty \circ Q^\infty$ in the weak-* topology on $\StateSpace\left(\A_N\right) $, and
\item $\{\beta^k : \overline{\N} \rightarrow (0,\infty) \}_{k \in \overline{\N}}$ is a family of convergent sequences such that for all $N\in \N $ if $k \in \N_{\geq c_N}$, then  $\beta^k(n)=\beta^\infty(n)$ for all $n \in \{0,1,\ldots,N\}$ and there exists $B: \overline{\N} \rightarrow (0,\infty)$ with $B(\infty)=0$ and $\beta^m(l) \leq B(l)$ for all $m,l \in \overline{\N}$,
\end{enumerate}
then using notation from Theorem (\ref{ind-lim-quantum-iso}): 
\begin{equation*}
\lim_{n \to \infty} \qpropinquity{} \left(\left(\A/I^n , \Lip^{\beta^n}_{\mathcal{U}/I^n , \mu_n }\right), \left(\A/I^\infty , \Lip^{\beta^\infty}_{\mathcal{U}/I^\infty , \mu_\infty }\right)\right)=0
\end{equation*}
\end{theorem}
\begin{proof}
By Lemma (\ref{metric-fusing-equiv-lemma}), the assumption that $(I^n)_{n \in \N} \subseteq \mathrm{Ideal}(\A)$ converges to $I^\infty \in \mathrm{Ideal}(\A)$ with respect to $\mathsf{m}_{i(\mathcal{U})}$  or the Fell topology implies that: 
\begin{equation*}\left\{ I^n=\overline{\cup_{k \in \N} I^n \cap \A_k}^{\Vert \cdot \Vert_\A} : n \in \overline{\N} \right\}
\end{equation*} is a fusing family with some fusing sequence $(c_n)_{n \in \N}$ such that \\ $\left\{\underrightarrow{\lim}\ \mathcal{I}(\A/ I^n) : n \in \overline{\N} \right\}$ is an IL-fusing family with fusing sequence $(c_n)_{n \in \N}$.

Fix $N \in \N$ and $k \in \N_{\geq c_N}$.   Let $x \in \A_N $, and let $Q^k_N : \A_N \rightarrow \A_N/(I^k \cap \A_N)$ and $Q_N^\infty : \A_N \rightarrow \A_N/(I^\infty \cap \A_N)$ denote the quotient maps, and let     Let $\phi_{I^k}:\underrightarrow{\lim}\ \mathcal{I}(\A/ I^k) \rightarrow \A/I^k $ denote the *-isomorphism given in Proposition (\ref{ind-lim-quotients-prop}) and recall that $ \mathcal{I}(\A/ I^k)=  \left(\A_n/(I^k \cap \A_n), \gamma_{I^k,n}\right)_{n \in \N}$ from Notation (\ref{ind-lim-quotient}).  Now, by Proposition (\ref{ind-lim-quotients-prop}) and its  commuting diagram, we gather: 
\begin{equation*}
\begin{split}
\mu_k \circ \phi_{I^k} \circ \indmor{\gamma_{I^k}}{N} \circ Q^k_N (x)&= \mu_k \circ \phi^N_{I^k} \circ  Q^k_N (x)\\
& =  \mu_k \circ \phi^N_{I^k}(x+I^k\cap \A_N)  = \mu_k (x+I^k) = \mu_k \circ Q^k (x).
\end{split}
\end{equation*}
Therefore, by hypothesis (2), the sequence $\left(\mu_k \circ \phi_{I^k} \circ \indmor{\gamma_{I^k}}{N} \circ Q^k_N \right)_{k \in \N_{\geq c_N}}$ converges to $\mu_\infty \circ \phi_{I^\infty} \circ \indmor{\gamma_{I^\infty}}{N} \circ Q^\infty_N $ in the weak-* topology on $\A_N$.  

Hence, the sequence $\left(\mu_k \circ \phi_{I^k} \circ \indmor{\gamma_{I^k}}{N}\right)_{k \in \N_{\geq c_N}}$ converges to $\mu_\infty \circ \phi_{I^\infty} \circ \indmor{\gamma_{I^\infty}}{N}  $ in the weak-* topology on $\StateSpace(\A_N / (I^\infty \cap \A_N)) $ by \cite[Theorem V.2.2]{Conway}.  Thus, by hypothesis (3)  and by Theorem (\ref{af-cont}), we have that:
\begin{equation*}
\lim_{n \to \infty} \qpropinquity{} \left( \left(\underrightarrow{\lim}\ \mathcal{I}(\A/ I^n), \Lip^{\beta^n}_{ \mathcal{I}(\A/ I^n), \mu_n \circ \phi_{I^n}}\right),  \left(\underrightarrow{\lim}\ \mathcal{I}(\A/ I^\infty), \Lip^{\beta^\infty}_{ \mathcal{I}(\A/ I^\infty), \mu_\infty \circ \phi_{I^\infty}}\right)\right)=0.
\end{equation*}
But, as $\phi_{I^n}^{-1}$ is an isometric isomorphism for all $n \in \overline{\N}$ by Theorem (\ref{ind-lim-quantum-iso}), we conclude:
\begin{equation*}
\lim_{n \to \infty} \qpropinquity{} \left(\left(\A/I^n , \Lip^{\beta^n}_{\mathcal{U}/I^n , \mu_n }\right), \left(\A/I^\infty , \Lip^{\beta^\infty}_{\mathcal{U}/I^\infty , \mu_\infty }\right)\right)=0,
\end{equation*}
which completes the proof.
\end{proof}

\subsection{The Boca-Mundici AF algebra}\label{b-m-af}
The Boca-Mundici AF algebra arose in \cite{Boca08} and \cite{Mundici88} independently and is constructed from the Farey tessellation.  In both \cite{Boca08}, \cite{Mundici88}, it was shown that the all Effros-Shen AF algebras (Notation (\ref{af-theta-notation})) arise as quotients up to *-isomorphism of certain primitive ideals of the Boca-Mundici AF algebra, which is the main motivation for our convergence result due to our work with convergence of Effros-Shen algebras in \cite{AL}.  In both \cite{Boca08}, \cite{Mundici88}, it was also shown that the center of the Boca-Mundici AF algebra is *-isomorphic to C([0,1]), which  provided the framework for C. Eckhardt to introduce a noncommutative analogue to the Gauss map in \cite{Eckhardt11}.

We present the construction of this algebra as presented in the paper by F. Boca \cite{Boca08}.  We refer mostly to Boca's work as his unique results pertaining to the Jacobson topology (for example \cite[Corollary 12]{Boca08}, which is the result that led us to begin this paper) are more applicable to our work (see Proposition (\ref{f-ideal-homeo})). As in \cite{Boca08}, we define the Boca-Mundici AF algebra recursively by the following Relations (\ref{farey-relations}).   We note that the relations presented here are the same as in \cite[Section 1]{Boca08}, but instead of starting at $n=0$, these relations begin at $n=1$, so that this formulation of the  Boca-Mundici AF algebra, denoted $\F$ (for Farey), as an inductive limit begins with $\C$. 
\begin{equation}\label{farey-relations}
\begin{cases}
q(n,0)=q(n,2^{n-1})=1, \ \  p(n,0)=0, \ \ p(n,2^{n-1})=1& :  n \in \N \setminus \{0\};\\
 q(n+1,2k)=q(n,k), \ \ p(n+1, 2k)=p(n,k),  &: n \in \N \setminus \{0\},\\
 & \quad k \in \{0, \ldots, 2^{n-1} \} ; \\
 q(n+1, 2k+1)=q(n,k)+ q(n,k+1), &: n \in \N \setminus \{0\},\\
 & \quad  k \in \{0, \ldots, 2^{n-1} -1\};\\
  p(n+1, 2k+1) = p(n,k) + p(n,k+1), & : n \in \N \setminus \{0\}, \\
  & \quad k \in \{0, \ldots, 2^{n-1}-1 \};  \\
r(n,k) = \frac{p(n,k)}{q(n,k)}, & : n \in \N \setminus \{0\}, \\
& \quad k \in \{0, \ldots, 2^{n-1} -1\}.
\end{cases}
\end{equation}
We now define the finite dimensional algebras which determine the inductive limit $\F$. 
\begin{definition}\label{f-fd} For $n \in \N \setminus \{ 0 \}$, define the finite dimensional C*-algebras,
\begin{equation*}
\F_n = \bigoplus_{k=0}^{2^{n-1}} \M (q(n,k)) \text{ and } \F_0 = \C .
\end{equation*}
\end{definition}

Next, we  define *-homomorphisms to complete the inductive limit recipe.  We utilize partial multiplicity matrices.  

\begin{definition}\label{f-part-mult-mat} For $n \in \N \setminus \{0\} $, let $F_n$ be the $(2^{n}+1) \times (2^{n-1}+1)$ matrix with entries in $\{ 0,1\}$ determined entry-wise by:
\begin{equation*}
(F_n)_{h,j} = \begin{cases}
1 & \text{if} \ \left(h=2k+1, k \in \{0, \ldots, 2^{n-1}\}, \land j=k+1 \right) \\
 &  \quad \lor \left(h=2k,  k \in \{1, \ldots, 2^{n-1}\} \land (j=k \lor j=k+1)\right); \\
 0 & \text{otherwise}. 
 \end{cases}
\end{equation*}
\end{definition}
For example, \begin{equation*}
F_1 = \left[\begin{array}{cc}
1 & 0 \\
1 & 1 \\
0 & 1
\end{array}\right],
F_2 = \left[ \begin{array}{ccc}
1 & 0 & 0 \\
1 & 1 & 0 \\
0 & 1 & 0 \\
0 & 1 & 1 \\
0 & 0 & 1
\end{array}\right]
\end{equation*}
We would like these matrices to determine unital *-monomorphisms, so that our inductive limit is a unital C*-algebra, which motivates the following Lemma (\ref{f-unital}).
\begin{lemma}\label{f-unital}
Using Definition (\ref{f-part-mult-mat}) and Relations (\ref{farey-relations}), if $n \in \N \setminus \{0\}$, then: 
\begin{equation*}
F_n \begin{pmatrix} q(n,0) \\q(n,1) \\ \vdots \\ q(n , 2^{n-1})   \end{pmatrix}= \begin{pmatrix}  q(n+1,0) \\ q(n+1,1) \\ \vdots \\ q(n+1,2^{n})  \end{pmatrix}.
\end{equation*}
\end{lemma}
\begin{proof}
Let $n \in \N \setminus \{0\}$.  Let $k \in \{1, \ldots , 2^{n-1}\}$ and consider $q(n+1, 2k-1)$.  Now, by Definition (\ref{f-part-mult-mat}), row $2k-1+1=2k$ of $F_n$ has $1$ in entry $k$ and $k+1$, and $0$ elsewhere. Thus: 
\begin{equation*}
\begin{split}
\left((F_n)_{2k,1}, \ldots , (F_n)_{2k,2^{n-1}+1}\right)\cdot \begin{pmatrix} q(n,0) \\q(n,1) \\ \vdots \\ q(n , 2^{n-1})   \end{pmatrix}&  = q(n,k-1)+q(n,k-1+1)\\
& =q(n+1, 2k-1)
\end{split}
\end{equation*}   
by Relations (\ref{farey-relations}).
Next, let $k \in \{0, \ldots, 2^{n-1}\}$ and consider $q(n+1, 2k)$.  By Definition (\ref{f-part-mult-mat}), row $2k+1$ of $F_n $ has $1$ in entry $k+1$ and $0$ elsewhere.  Thus:  
\begin{equation*}
\left((F_n)_{2k+1,1}, \ldots , (F_n)_{2k+1,2^{n-1}+1}\right)\cdot \begin{pmatrix} q(n,0) \\q(n,1) \\ \vdots \\ q(n , 2^{n-1})   \end{pmatrix}  = q(n,2k)=q(n+1, 2k)
\end{equation*}
by Relations (\ref{farey-relations}). Hence,  by matrix multiplication, the proof is complete.    
\end{proof}

\begin{definition}[{\cite{Boca08, Mundici88}}]\label{boca-mundici-af}  Define
$\varphi_0 : \F_0 \rightarrow \F_{1}$ by $\varphi_0 (a)= a \oplus a$. For $n \in \N \setminus \{0\}$, by \cite[Lemma III.2.1]{Davidson} and Lemma (\ref{f-unital}), we let $\varphi_n : \F_n \rightarrow \F_{n+1}$  be a unital *-monomorphism determined by $F_n$ of Definition (\ref{f-part-mult-mat}).  Using Definition (\ref{f-fd}), we let the unital C*-inductive limit (Notation (\ref{ind-lim})): 
\begin{equation*}
\F = \underrightarrow{\lim} (\F_n, \varphi_n )_{n \in \N}
\end{equation*}
denote the {\em Boca-Mundici AF algebra}.

Let $\F^n = \indmor{\varphi}{n}(\F_n)$ for all $n \in \N$ and $\mathcal{U}_\F = (\F^n)_{n \in \N}$, which is a non-decreasing sequence of C*-subalgebras of $\F$ such that $\F=\overline{\cup_{n \in \N} \F^n}^{\Vert \cdot \Vert_\F}$, where $\F^0 = \C1_\F $ (see \cite[Chapter 6.1]{Murphy}).
\end{definition}
We note that in \cite{Boca08}, the AF algebra $\F$ is constructed by a Bratteli diagram displayed as \cite[Figure 2]{Boca08}, so in order to utilize the results of \cite{Boca08}, we verify that we have the same Bratteli diagram up to adding one  vertex of label 1 at level $n=0$ satisfying the conditions at the beginning of \cite[Section 1]{Boca08}. But, first, we fix some notation for Bratteli diagrams and state some well-known results that will prove useful. 

\begin{definition}[{\cite{Bratteli72}}]\label{Bratteli-diagram} A {\em Bratteli diagram} is given by a directed graph $\mathcal{D}=(V^\mathcal{D}, E^{\mathcal{D}})$ with labelled vertices and multiple edges between two vertices is allowed.  The set $V^\mathcal{D} \subset \N^2 $ is the set of labeled vertices and $E^\mathcal{D} \subset \N^2 \times \N^2$ is the set of edges, which consist of ordered pairs from $V^\mathcal{D}$.  For each $n \in \N$, let $ v^\mathcal{D}_n \in \N $.

Define $V^\mathcal{D}=\cup_{n \in \N} V^{\mathcal{D}}_n$, where for $n \in \N$, we let:
\begin{equation*}
V^{\mathcal{D}}_n= \left\{ (n,k) \in \N \times \N  : k \in \{0, \ldots, v^\mathcal{D}_n\} \right\}  ,
\end{equation*}
 and we denote the label of the vertices $(n,k) \in V^\mathcal{D}$ by $[n,k]_\mathcal{D} \in \N \setminus \{0\}$.  
 
Next, let $E^\mathcal{D} \subset V^\mathcal{D} \times V^\mathcal{D}$. Now, we list some axioms for  $V^\mathcal{D}$ and $E^\mathcal{D}$. 
\begin{itemize}
\item[(i)] For all $n \in \N$, if $m \in \N \setminus \{n+1\}$, then $((n,k),(m,q)) \not\in E^\mathcal{D} $ for all $k \in \left\{0, \ldots, v^\mathcal{D}_n \right\}$ and $q \in  \left\{0, \ldots, v^\mathcal{D}_m \right\}.$
\item[(ii)] If $(n,k) \in V^\mathcal{D}$, then there exists $q \in \left\{0, \ldots,  v^\mathcal{D}_{n+1}\right\}$ such that $((n,k),(n+1,q)) \in E^\mathcal{D} .$
	\item[(iii)] If $n \in \N \setminus \{0\}$ and $(n,k) \in V^\mathcal{D}$ , then there exists $q \in \left\{0, \ldots,   v^\mathcal{D}_{n-1}\right\}$ such that $((n-1,q), (n,k)) \in E^\mathcal{D} $.
\end{itemize} 

If $\mathcal{D}$ satisfies the all of the above properties, then we call  $\mathcal{D}$ a {\em Bratteli diagram}, and we denote the set of all Bratteli diagrams  by $\mathscr{BD}$.

We also introduce the following notation.   For each $n \in \N$, let:
\begin{equation*}
E^{\mathcal{D}}_n =(V^\mathcal{D}_n \times  V^\mathcal{D}_{n+1}) \cap  E^\mathcal{D}  , 
\end{equation*}
which  by axiom (i), we have that $E^\mathcal{D} = \cup_{n \in \N} E^{\mathcal{D}}_n$. Also, for $ ((n,k),(n+1,q)) \in E^\mathcal{D}_n$, we denote $ [(n,k),(n+1,q)]_\mathcal{D} \in \N \setminus \{0\}$ as the number of edges from $(n,k)$ to $(n+1,q )$.  
Let $(n,k) \in V^\mathcal{D}$, define:
\begin{equation*}
R^\mathcal{D}_{(n,k)} = \left\{ (n+1,q) \in V^\mathcal{D}_{n+1} : ((n,k), (n+1, q)) \in E^\mathcal{D} \right\},
\end{equation*}
which is non-empty by axiom (ii).  Also, for $n \in \N$, we refer to $V^\mathcal{D}_n , E^\mathcal{D}_n ,$ and $\left( V^\mathcal{D}_n ,  E^\mathcal{D}_n\right)$ as the vertices at level $n$, edges at level $n$, and diagram at level $n$, respectively.
\end{definition}

\begin{remark}
It is easy to see that this definition coincides with Bratteli's of {\cite[Section 1.8]{Bratteli72}} in that we simply trade his arrow notation with that of edges and number of edges. That is, given a Bratteli diagram $\mathcal{D}$, the correspondence is: $(n,k) \searrow^p (n+1,q)$ if and only if $((n,k),(n+1,q)) \in E^\mathcal{D}$  and  $[(n,k),(n+1,q)]_\mathcal{D}=p.$ 
\end{remark}
\begin{definition}[{\cite{Bratteli72}}]\label{diagram-af} Let $\mathcal{I} = (\A_n,\alpha_n)_{n\in\N}$ be an inductive sequence of finite dimensional C*-algebras with C*-inductive limit $\A$, where  $\alpha_n$ is injective for all $n\in\N$.  Thus, $\A$ is an AF algebra by {\cite[Chapter 6.1]{Murphy}}.  Let $\mathcal{D}_b (\A)$ be a diagram associated to $\A$ constructed as follows.  

Fix $n \in \N$.  Since $\A_n $ is finite dimensional, $\A_n \cong \oplus_{k=0}^{a_n} \M(n(k))   $ such that $a_n \in \N $ and $n(k)\in \N \setminus \{0\}$ for $ k \in \{0, \ldots, a_{n} \}$.  Define: 
\begin{equation*}
v^{\mathcal{D}_b (\A)}_n = a_n , \ V^{\mathcal{D}_b (\A)}_n  = \left\{ (n,k) \in \N^2 : k \in \left\{0, \ldots , v^{\mathcal{D}_b (\A)}_n \right\}\right\},
\end{equation*}
and  label $[n,k]_{\mathcal{D}_b (\A)} = \sqrt{\dim(\M(n(k)))}$ for $ k \in \left\{0, \ldots, v^{\mathcal{D}_b (\A)}_n\right\}$.

Let $A_n$ be the $a_{n+1}+1 \times a_n +1$-partial multiplicity matrix assocaited to $\alpha_n : \A_n \rightarrow \A_{n+1} $ with entries $(A_n)_{i,j} \in \N, i \in \{1, \ldots, a_{n+1}+1\}, j \in \{1, \ldots, a_{n}+1\}$ given by {\cite[Lemma III.2.2]{Davidson}}.  Define: 
\begin{equation*}
E^{\mathcal{D}_b (\A)}_n  = \left\{((n,k),(n+1,q)) \in \N^2 \times \N^2 : (A_n)_{q+1, k+1} \neq 0 \right\},
\end{equation*}
and if $((n,k),(n+1,q)) \in E^{\mathcal{D}_b (\A)}_n$, then let the number of edges be $[(n,k),(n+1,q)]_{\mathcal{D}_b (\A)}$ $=(A_n)_{q+1, k+1}$.  

Let $V^{\mathcal{D}_b (\A)} = \cup_{n \in \N} V^{\mathcal{D}_b (\A)}_n , \ E^{\mathcal{D}_b (\A)} =\cup_{n \in \N} E^{\mathcal{D}_b (\A)}_n, $ and ${\mathcal{D}_b (\A)}=(V^{\mathcal{D}_b (\A)},E^{\mathcal{D}_b (\A)}).$  By {\cite[Section 1.8]{Bratteli72}}, we conclude ${\mathcal{D}_b (\A)} \in \mathscr{BD}$ is a Bratteli diagram as in Definition (\ref{Bratteli-diagram}).

If $\A$ is an AF algebra of the form $\A=\overline{\cup_{n \in \N} \A_n }^{\Vert \cdot \Vert_\A}$ where $\mathcal{U}=(\A_n)_{n \in \N}$ is a non-decreasing sequence of finite dimensional C*-subalgebras of $\A$, then the diagram $\mathcal{D}_b (\A)$ has the same vertices as the one above, and the edges are formed by the partial multiplicity matrix built from the partial multiplicities of the inclusion mappings $\iota_n : \A_n \rightarrow \A_{n+1}$ for all $n \in \N$.
\end{definition}
\begin{remark}\label{converse-diagram}
We note that the converse of the Definition (\ref{diagram-af}) is true in the sense that given a Bratteli diagram, one may construct an AF algebra associated to it.  The process is described in {\cite[Section 1.8]{Bratteli72}}, and in particular, one may construct partial multiplicity matrices from the edge set, which then provide injective *-homomorphisms to build an inductive limit.
\end{remark}
As an example, which will be used in Proposition  (\ref{f-quotient-iso}), we display the Bratteli diagram for the Effros-Shen AF algebras of Notation (\ref{af-theta-notation}).  
\begin{example}\label{af-theta-diagram}
Fix $\theta \in (0,1) \setminus \Q$ with continued fraction expansion $\theta=[a_j]_{j \in \N}$ using Expression (\ref{continued-fraction-eq}) with rational approximations $\left(\frac{p^\theta_n}{q^\theta_n}\right)_{n \in \N}$ given by Expression (\ref{pq-rel-eq}).  Let $\af{\theta}$ be the Effros-Shen AF algebra from Notation (\ref{af-theta-notation}).  Thus, $v^{\mathcal{D}_b (\af{\theta})}_0=0$ and $V^{\mathcal{D}_b (\af{\theta})}_0 = \{(0,0)\}$ with $[0,0]_{\mathcal{D}_b (\af{\theta})}=1$.  For $n \in \N \setminus \{0\}$, we have $v^{\mathcal{D}_b (\af{\theta})}_n =1$ and $V^{\mathcal{D}_b (\af{\theta})}_n=\{(n,0), (n,1)\}$ with $[n,0]_{\mathcal{D}_b (\af{\theta})}=q^\theta_{n} , [n,1]_{\mathcal{D}_b (\af{\theta})}=q^\theta_{n-1}$.   The partial multiplicity matrix for $n=0$ is: \begin{equation*}A_0 =  \begin{pmatrix}a_{1}  \\ 1 \end{pmatrix},\end{equation*} and let $n \in \N \setminus \{0\}$, then the partial multiplicity matrix is: 
\begin{equation*}A_n =\begin{pmatrix}a_{n+1} & 1 \\ 1 & 0 \end{pmatrix},
\end{equation*} by Notation (\ref{af-theta-notation}) and {\cite[Lemma III.2.1]{Davidson}}, which determines the edges.  We now provide the diagram as a graph, where the label in the edges denotes number of edges and the top row contains the vertices $(n,1)$ with their labels with $n$ increasing from left to right with the bottom row having vertices $(n,0)$ with their labels with $n$ increasing  from left to right.  Let $n \geq 4:$
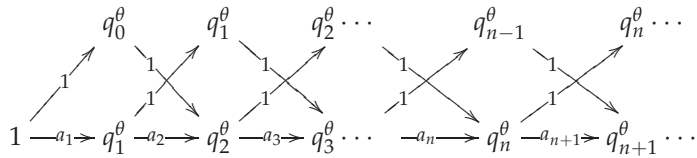
\begin{displaymath}
\xymatrix{
  & q^\theta_0 \ar[dr]|(0.35)1 & q^\theta_1 \ar[dr]|(0.35)1 & q^\theta_2 \cdots \ \ \    \ar[dr]|(0.35)1 & q^\theta_{n-1} \ar[dr]|(0.35)1 & q^\theta_n \cdots\\
 1 \ar[r]|{a_1} \ar[ur]|1 & q^\theta_1 \ar[r]|(.4){a_2} \ar[ur]|(0.35)1  & q^\theta_2 \ar[r]|(.4){a_3} \ar[ur]|(0.35)1 & q^\theta_3 \cdots \ \ \  \ar[r]|{a_n}  \ar[ur]|(0.35)1  & q^\theta_{n}  \ar[r]|(.4){a_{n+1}} \ar[ur]|(0.35)1& q^\theta_{n+1} \cdots }
\end{displaymath}
\end{example}

Returning to the diagram setting, we define what an ideal of a diagram is.

\begin{definition}\label{ideal-diagrams}  Let $\mathcal{D}= (V^\mathcal{D}, E^\mathcal{D})$ be a Bratteli diagram defined in Definition (\ref{Bratteli-diagram}).  We call $\mathcal{D}(I)= (V^I , E^I)$ an {\em ideal diagram} of $\mathcal{D}$ if $V^I \subseteq V^\mathcal{D} , E^I \subseteq E^\mathcal{D} $ and: 
\begin{itemize}
\item[(i)](directed)  if $(n,k) \in V^I$ and $((n,k), (n+1,q)) \in E^\mathcal{D}$, then $(n+1,q) \in V^I$. 
\item[(ii)](hereditary) if $(n,k) \in V^\mathcal{D}$ and $R^\mathcal{D}_{(n,k)} \subseteq V^I$, then $ (n,k) \in V^I $.
\item[(iii)](edges)  If $(n,k), (n+1,q) \in V^I$ such that $((n,k),(n+1,q)) \in E^\mathcal{D}$, then $((n,k),(n+1,q)) \in E^I $.
\end{itemize}
Furthermore, if $(n,k) \in V^\mathcal{D} \cap V^I$, then $[n,k]_\mathcal{D}=[n,k]_{\mathcal{D}(I)}$.  And, if $((n,k),(n+1,q)) \in E^\mathcal{D} \cap E^I,$ then $[(n,k),(n+1,q)]_\mathcal{D}=[(n,k),(n+1,q)]_{\mathcal{D}(I)}$.

Also, for $n \in \N$, denote $V^I_n = V^\mathcal{D}_n \cap V^I $ and $E^I_n = E^\mathcal{D}_n \cap E^I $  with $I_n = (V^I_{n} , E^I_n ) $ to also include all associated labels and number of edges, and we will refer to $V^I_n$ as the vertices at level $n$ of the diagram. Let $\mathrm{Ideal}(\mathcal{D})$ denote the set of ideals of $\mathcal{D} $.
\end{definition}

\begin{notation}\label{ideal-af}Let $\A=\overline{\cup_{n \in \N} \A_n }^{\Vert \cdot \Vert_\A}$ be an AF algebra where $\mathcal{U}=(\A_n)_{n \in \N}$ is a non-decreasing sequence of finite dimensional C*-subalgebras of $\A$.  Let $\mathcal{D}_b (\A)$ be the diagram given by Definition (\ref{diagram-af}).  

Let $I \in \mathrm{Ideal}(\A)$ be a norm closed two-sided ideal of $\A$, then by \cite[Lemma 3.2]{Bratteli72}, the subset $\Lambda$ of $\mathcal{D}_b (\A)$ formed by $I$ is an ideal in the sense of Definition (\ref{ideal-diagrams}), and denote this by $\mathcal{D}_b(\A)(I) \in \mathrm{Ideal}(\mathcal{D}_b(\A))$, where  $\mathrm{Ideal}(\mathcal{D}_b(\A))$ is the set of ideals of $\mathcal{D}_b (\A)$ from Definition (\ref{ideal-diagrams}).
\end{notation}
\begin{proposition}\label{ideal-corr}{\cite[Lemma 3.2]{ Bratteli72}} Let $\A=\overline{\cup_{n \in \N} \A_n }^{\Vert \cdot \Vert_\A}$ be an AF algebra where $\mathcal{U}=(\A_n)_{n \in \N}$ is a non-decreasing sequence of finite dimensional C*-subalgebras of $\A$ and Bratteli diagram $\mathcal{D}_b(\A)$ from Definition (\ref{diagram-af}). Using Notation (\ref{ideal-af}) and Definition (\ref{ideal-diagrams}), the map: 
\begin{equation*}
i(\cdot, \mathcal{D}_b(\A)) : I \in  \mathrm{Ideal}(\A) \longmapsto \mathcal{D}_b(\A)(I) \in \mathrm{Ideal}(\mathcal{D}_b (\A))
\end{equation*}
given by {\cite[Lemma 3.2]{Bratteli72}} is a well-defined bijection, where the vertices of $V^{\mathcal{D}_b(\A)(I)}_n$ are uniquely determined by $I \cap \A_n$ for each $n \in \N$.
\end{proposition}

\begin{proposition}\label{f-diagram}The Bratteli diagram of $\F$, denoted 
$\mathcal{D}_b(\F)=(V^{\mathcal{D}_b(\F)}, E^{\mathcal{D}_b(\F)})$ of Definition (\ref{diagram-af}), satisfies for all $n \in \N \setminus \{0\}$: 
\begin{itemize}
\item[(i)] $V^{\mathcal{D}_b(\F)}_n = \left\{(n,k) : k \in \{0, \ldots, 2^{n-1}\}\right\} $
\item[(ii)] $((n,k), (n+1, l)) \in E^{\mathcal{D}_b(\F)}_n$ if and only if $\vert 2k-l \vert \leq 1 $. And, there exists only one edge between any two vertices for which there is an edge.
\end{itemize}
\end{proposition}
\begin{proof}
Property (i) is clear by Definition (\ref{f-fd}).  By \cite[Section III.2 Definition \textbf{Bratteli diagram}]{Davidson}, an edge exists from $(n,s)$ to $(n+1, t)$ if and only if its associated entry in the partial multiplicity matrix $(F_n)_{t+1,s+1}$ is non-zero.    

Now, assume that $\vert 2s-t \vert \leq 1$.  Assume $t= 2k+1$ for some $ k \in \{0, \ldots , 2^{n-1}-1 \} $.  We thus have $\vert 2s-t \vert \leq 1  \iff k \leq s \leq k+1 \iff s \in \{k,k+1 \} $, since $s \in \N$.

Next, assume that $ t= 2k$ for some $ k \in \{0, \ldots , 2^{n-1} \}$. We thus have  \\
$\vert 2s-t \vert \leq 1 \iff  -1/2 +k \leq s \leq 1/2 +k \iff \vert s-k \vert \leq 1/2 \iff s=k $ since $s \in \N$.  But, considering both $t$ odd and even, these equivalences are equivalent to the conditions for $(F_n)_{t+1,s+1}$ to be non-zero by Definition (\ref{f-part-mult-mat}), which determine the edges of $\mathcal{D}_b(\F)$.  Furthermore, since the non-zero entries of $F_n$ are all $1$, only one edge exists between vertices for which there is an edge.
\end{proof}

Next, we describe the ideals of $\F$, whose quotients are *-isomorphic to the Effros-Shen AF algebras.  

\begin{definition}[{\cite{Boca08}}]\label{f-ideals}
Let $\theta \in (0,1) \setminus \Q$. We define the ideal $I_\theta \in \mathrm{Ideal}(\F) $ diagrammatically.

By the proof of \cite[ Proposition 4.i]{Boca08}, for each $n \in \N \setminus \{0\}$, there exists a unique $j_n (\theta) \in \{0, \ldots , 2^{n-1}-1 \}$ such that $r(n,j_n (\theta) ) < \theta < r(n,j_n (\theta)+1 )$ of Relations (\ref{farey-relations}).   The set of vertices is defined by:
\begin{equation*}  V^{\mathcal{D}_b (\F)} \setminus \left(\{(n,j_n (\theta)),  (n,j_n (\theta)+1) : n \in \N \setminus \{0\} \cup \{(0,0)\}\right)
\end{equation*}
 and we denote this set by $V^{\mathcal{D}(I_\theta )} $.  Let $E^{\mathcal{D}(I_\theta )} $ be the set of edges of $\mathcal{D}_b(\F)$, which are between the vertices in $V^{\mathcal{D}(I_\theta )} $ and let $\mathcal{D}(I_\theta )=\left(V^{\mathcal{D}(I_\theta )} ,E^{\mathcal{D}(I_\theta )} \right)$.  By \cite[ Proposition 4.i]{Boca08}, the diagram $\mathcal{D}(I_\theta ) \in \mathrm{Ideal}(\mathcal{D}_b (\F))$ is an ideal diagram of Definition (\ref{ideal-diagrams}). 
 
Using Proposition (\ref{ideal-corr}), define: 
\begin{equation*}I_\theta=i(\cdot, \mathcal{D}_b (\A))^{-1}\left(\mathcal{D}(I_\theta)\right) \in \mathrm{Ideal}(\A).
\end{equation*}    By \cite[ Proposition 4.i]{Boca08}, if $n \in \N \setminus \{0,1\}$ and $1 \leq j_n (\theta) \leq 2^{n-1} -2 ,$ then:
 \begin{equation*}
 I_\theta \cap \F^n=\indmor{\varphi}{n}\left( \left(\oplus_{k=0}^{j_n(\theta)-1}\M(q(n,k)) \right) \oplus \{0\} \oplus \{0\} \oplus \left( \oplus_{k=j_n(\theta) +2}^{2^{n-1}}  \M(q(n,k)) \right)\right).
\end{equation*}
If $j_n(\theta) = 0$, then: 
 \begin{equation*}
 I_\theta \cap \F^n =\indmor{\varphi}{n}\left(  \{0\} \oplus \{0\} \oplus \left( \oplus_{k=j_n(\theta) +2}^{2^{n-1}}  \M(q(n,k)) \right)\right).
\end{equation*}
If $j_n (\theta) = 2^{n-1} -1$, then:
 \begin{equation*}
 I_\theta \cap \F^n =\indmor{\varphi}{n}\left( \left(\oplus_{k=0}^{j_n(\theta)-1}\M(q(n,k)) \right) \oplus \{0\} \oplus \{0\} \right),
\end{equation*}
and if $n \in \{0,1\}$, then  $I_\theta \cap \F^n = \{0\}$.
We note that $I_\theta \in \mathrm{Prim}(\F)$  by \cite[ Proposition 4.i]{Boca08}.
\end{definition}

Before we move on to describing the quantum metric structure of quotients of the ideals of Definition (\ref{f-ideals}), let's first capture more properties of the structure of the ideals introduced in Definition (\ref{f-ideals}), which are sufficient for later results.

\begin{lemma}\label{f-ideal-vert}
Using notation from Definition (\ref{f-ideals}), if $n \in \N \setminus \{0\}, \theta \in (0,1)\setminus \Q $, then $j_{n+1}(\theta) \in \{2j_n (\theta), 2j_n (\theta)+1 \}$. 
\end{lemma} 
\begin{proof}
We first note that the vertices $V^{\mathcal{D}_b(\A)}\setminus V^{\mathcal{D}(I_\theta)}$ determine a Bratteli diagram associated to the AF algebra $\F/I_\theta$, which we will denote $\mathcal{D}_b (\A/I_\theta)$, as in Definition (\ref{diagram-af}) by \cite[Proposition 3.7]{Bratteli72} up to shifting the vertices in $\N^2$ uniformly, in which the edges for $\mathcal{D}_b (\A/I_\theta)$ are given by all the edges from $E^{\mathcal{D}_b(\A)}$ between all vertices in $V^{\mathcal{D}_b(\A)}\setminus V^{\mathcal{D}(I_\theta)}$.  Thus, by Defintion (\ref{f-ideals}), the vertex set for  $\mathcal{D}_b (\A/I_\theta)$ is: 
\begin{equation}\label{f-quotient-vert}
V^{\mathcal{D}_b(\A)}\setminus V^{\mathcal{D}(I_\theta)}=\left\{ (n,j_n (\theta)), (n,j_n (\theta)+1) \in \N^2 : n \in \N \setminus \{0\}\right\} \cup \{ (0,0)\},
\end{equation}
and in particular, this vertex set along with the edges between the vertices satisfy axioms (i),(ii), (iii) of Definition (\ref{Bratteli-diagram}).

Consider $n=1 $.  Since there are only 3 vertices at level $n=2$, the conclusion is satisfied since $j_{2}(\theta), j_2(\theta)+1 \in \{0,1,2\}$ and $j_1 (\theta) =0$ since there are only 2 vertices at level $n=1$. Furthemore, note by definition, we have $j_n (\theta) \leq 2^{n-1}-1$ since $j_{n}(\theta) + 1 \in \{0, \ldots, 2^{n-1} \}$. \setcounter{case}{0}
\begin{case}For $n \geq 2$, we show that $j_{n+1}(\theta) \geq 2j_n (\theta)$. \end{case} We note that if $j_n(\theta)=0$, then clearly $j_{n+1}(\theta) \geq 0=2j_n (\theta)$.  Thus, we may assume that $j_n (\theta) \geq 1$. Hence, we may assume by way of contradiciton that $j_{n+1}(\theta) \leq 2j_n( \theta)-1$.  Consider $j_n (\theta) +1 $.  By Expression (\ref{f-quotient-vert}), the only vertices at level $n+1$ of the diagram of $\F/I_\theta$ are $(n+1,j_{n+1}(\theta))$ and $(n+1, j_{n+1}(\theta)+1)$. Consider $j_{n+1}(\theta)+1$.  Now: 
\begin{equation*}\vert 2(j_n (\theta)+1) - (j_{n+1}(\theta)+1) \vert = \vert 2j_n (\theta)-j_{n+1}(\theta) +1 \vert .
\end{equation*}  But,  by our contradiction assumption, we have $2j_n (\theta)-j_{n+1}(\theta) +1 \geq 2j_n (\theta)+1-2j_n (\theta) +1 = 2 $.  Thus, by Proposition  (\ref{f-diagram}), there is no edge from $(n,j_n (\theta)+1)$ to $(n+1,j_{n+1}(\theta)+1)$. Next, consider  $j_{n+1}(\theta)$.  Similarly, we have $\vert 2(j_n (\theta)+1) - j_{n+1}(\theta) \vert =\vert 2j_n (\theta) - j_{n+1}(\theta) +2 \vert$.  However, the indices $2j_n (\theta) - j_{n+1}(\theta) +2 \geq 2j_n (\theta)+1-2j_n (\theta) +2 = 3 $.  And, again by Proposition  (\ref{f-diagram}), there is no edge from $(n,j_n (\theta)+1)$ to $(n+1,j_{n+1}(\theta))$.  But, by Expression (\ref{f-quotient-vert}), this implies that $(n, j_{n+1}(\theta)+1)$ is a vertex in the quotient diagram $\F/I_\theta$ in which there does not exist a vertex $(n+1, l)$ in the diagram of $\F/I_\theta$ such that $((n, j_{n+1}(\theta)+1),(n+1, l))$ is an edge in the diagram of $\F/I_\theta$, which is a contradiction  since the quotient diagram is a Bratteli diagram that would not satisfy axiom (ii) of Definition (\ref{Bratteli-diagram}). Therefore, we conclude $j_{n+1}(\theta) \geq 2j_n (\theta)$. 
\begin{case}
For $n \geq 2$, we show that $j_{n+1}(\theta) \leq 2j_n (\theta) +1 $.\end{case} Now, if $j_n(\theta) = 2^{n-1}-1 $, then $j_{n+1}(\theta)+1 \leq 2^n=2(2^{n-1}-1) +2$ and thus $j_{n+1}(\theta) \leq 2(2^{n-1}-1) +1=2j_n (\theta) +1$ and we would be done.  Thus, we may assume that $j_n(\theta) \leq 2^{n-1}-2 $ and we note that this can only occur in the case that $n \geq 3$, which implies that the case of $n=2$ is complete.  Thus, we may assume by way of contradiction that $j_{n+1}(\theta) \geq 2j_n (\theta) +2 $. Consider $j_n (\theta)$. As in Case 1, we provide a contradiction via a diagram approach.  Consider $j_{n+1}(\theta)+1$.  Now, we have $\vert 2j_n(\theta) - (j_{n+1}(\theta)+1)\vert = \vert 2j_n(\theta) - j_{n+1}(\theta)-1 \vert$.  But, by our contradiction assumption, we gather that $2j_n(\theta) - j_{n+1}(\theta)-1 \leq 2j_n(\theta)  -2j_n (\theta)  -2 -1= -3$ and $\vert 2j_n(\theta) - (j_{n+1}(\theta)+1)\vert \geq 3$.  Thus, by Proposition  (\ref{f-diagram}), there is no edge from $(n,j_n (\theta))$ to $(n+1,j_{n+1}(\theta)+1)$.  Next, consider $j_{n+1}(\theta)$.  Similarly, we have $2j_n(\theta) - j_{n+1}(\theta) \leq 2j_n(\theta)-2j_n(\theta)-2 = -2$ and $\vert 2j_n(\theta) - j_{n+1}(\theta)\vert \geq 2$. Thus, by Proposition  (\ref{f-diagram}), there is no edge from $(n,j_n (\theta))$ to $(n+1,j_{n+1}(\theta))$.  Thus, by Expression (\ref{f-quotient-vert}) and the same diagram argument of Case 1, we have reached a contradiction.  Hence, $j_{n+1}(\theta) \leq 2j_n (\theta) +1.$

Hence, combining Case 1 with the coments immediately preceding Case 1 and Case 2, the proof is complete. 
\end{proof}

The following proposition and remark make use of the Baire space and some of its properties, so we define the Baire space now.
\begin{definition}[\cite{Miller95}]\label{Baire-Space-def}
The \emph{Baire space} $\BaireSpace$ is the set $(\N\setminus\{0\})^\N$ endowed with the metric $\mathsf{d}$ defined, for any two $(x(n))_{n\in\N}$, $(y(n))_{n\in\N}$ in $\BaireSpace$, by:
\begin{equation*}
\mathsf{d}\left((x(n))_{n\in\N}, (y(n))_{n\in\N}\right) = \begin{cases}
0 &:\text{if $x(n) = y(n)$ for all $n\in\N$},\\
2^{-\min\left\{ n \in \N : x(n) \not= y(n) \right\}} &:\text{otherwise}\text{.}
\end{cases}
\end{equation*}
\end{definition}

In the next two results, on the subset of ideals of Definition (\ref{f-ideals}), we provide a useful topological result about the metric on ideals of Corollary (\ref{ind-lim-metric-cor}), in which the equivalence of (1) and (3) is a consequence of \cite[Corollary 12]{Boca08}, which is unique to Boca's work on the AF algebra, $\F$. Furthermore, Boca showed that $\mathrm{Prim}(\F)$ with the Jacobson topology is not $T_1$ and therefore not Hausdorff in \cite[Remark 8(ii)]{Boca08}, and thus the following proposition does not immediately follow from Corollary (\ref{comm-metric}).  However, in the next proposition, it is the metric of Corollary (\ref{ind-lim-metric-cor}) that allows us to recover the Jacobson topology from the Fell topology on the subset of ideals of Definition (\ref{f-ideals}), which displays a direct advantage of this metric.

\begin{proposition}\label{f-ideal-homeo}
If $(\theta_n)_{n \in \overline{\N}} \subseteq (0,1)\setminus \Q$, then using notation from Definition (\ref{boca-mundici-af}) and Definition (\ref{f-ideals}), the following are equivalent:
\begin{enumerate}
\item $(\theta_n)_{n \in \N}$ converges to $\theta_\infty$ with respect to the usual topology on $\R$;
\item  $(\mathsf{cf}(\theta_n))_{n \in \N}$ converges to $\mathsf{cf}(\theta_\infty)$ with respect to the Baire space, $\BaireSpace$ and its metric from Definition (\ref{Baire-Space-def}), where $\mathsf{cf}$ denotes the unique continued fraction expansion of an irrational;
\item $(I_{\theta_n})_{n \in \N}$ converges to $I_{\theta_\infty}$ with respect to the Jacobson topology (Definition (\ref{Jacobson-topology})) on $\mathrm{Prim}(\F)$;
\item $(I_{\theta_n})_{n \in \N}$ converges to $I_{\theta_\infty}$ with respect to the metric topology of $\mathsf{m}_{i(\mathcal{U}_\F)}$ of Corollary (\ref{AF-ideal-metric-thm}) or the Fell topology of Definition (\ref{Fell-topology}).
\end{enumerate}
\end{proposition}
\begin{proof}
The equivalence between (1) and (2) is a classic result, in which a proof can be found in \cite[Proposition 5.10]{AL}.  The equivalence between (1) and (3) is immediate from \cite[Corollary 12]{Boca08}.  And, therefore, (2) is equivalent to (3). Thus, it remains to prove that (3) is equivalent to (4).  

(4) implies (3) is an immediate consequence of Proposition (\ref{Fell-Jacobson}) and Theorem (\ref{AF-ideal-metric-thm}) as the Fell topology is stronger.  Hence, assume (3), then since we have already established (3) implies (2), we may assume (2).  For each $n \in \overline{\N}$, let $\mathsf{cf}(\theta_n)=[a^n_j]_{j \in \N}$.  By assumption, the coordinates $a^n_0=0$ for all $n \in \overline{\N}$.  Now, assume that there exists $N \in \N \setminus \{0\}$ such that $a^n_j=a^\infty_j$ for all $n \in \N$ and $j \in \{0, \ldots, N\}$. Assume without loss of generality that $N$ is odd. Thus, using \cite[Figure 5]{Boca08}, we have that: 
\begin{equation}\label{vert-eq}
L_{a^n_1-1} \circ R_{a^n_2} \circ \cdots \circ L_{a^n_N} = L_{a^\infty_1-1} \circ R_{a^\infty_2} \circ \cdots \circ L_{a^\infty_N}
\end{equation}
for all $n \in \N$.  But, Equation (\ref{vert-eq}) determines the vertices for the diagram of the quotient $\F/I_{\theta_n}$ for all $n \in \overline{\N}$ by the proof of \cite[Proposition 4.i]{Boca08}.  But, the vertices of the diagram of the quotient $\F/I_{\theta_n}$ are simply the complement of the vertices of the diagram of $I_{\theta_n}$ by \cite[Theorem III.4.4]{Davidson}.  Now, primitive ideals must have the same vertices at level $0$ of the diagram since they cannot equal $\A$ by Definition (\ref{Jacobson-topology}) and are thus non-unital.  But, for any $\eta \in (0,1) \setminus \Q$,  the ideals $I_\eta$ must always have the same vertices at level $1$ of the diagram as well since the only two vertices are $(1,0), (1,1)$ and $r(1,0)=0 < \theta <1=r(1,1)$ by Relations (\ref{farey-relations}) for all $\theta \in (0,1) \setminus \Q$.   Thus, Equation (\ref{vert-eq}), we gather that $I_{\theta_n} \cap \F^j = I_{\theta_\infty} \cap \F^j$ for all $n \in \N$ and: 
\begin{equation*}j \in \left\{0, \ldots, \max \left\{ 1, a_1^\infty-1 +\left(\sum_{k=2}^N a_k^N\right)\right\} \right\},
\end{equation*} where $\max \left\{ 1, a_1^\infty-1 +\left(\sum_{k=2}^N a_k^N\right)\right\}   \geq N$ as the terms of the continued fraction expansion are all positive integers for coordinates greater than $0$. Thus, by the definition of the metric on the Baire Space and the metric $\mathsf{m}_{i(\mathcal{U}_\F)}$, we conclude that convergence in the the Baire space metric of $(\mathsf{cf}(\theta_n))_{n \in \N}$ to $\mathsf{cf}(\theta_\infty)$ implies convergence of $(I_{\theta_n})_{n \in \N}$  to $I_{\theta_\infty}$ with respect to the metric  $\mathsf{m}_{i(\mathcal{U}_\F)}$  or the Fell topology by Theorem (\ref{AF-ideal-metric-thm}).
\end{proof}
The next result follows  from Proposition (\ref{f-ideal-homeo}) and the proofs of \cite[Proposition 4.i and Lemma 11]{Boca08}, but we provide a proof.
\begin{proposition}\label{f-j-irr-ideal-homeo}
The map:
\begin{equation*}
\theta \in (0,1)\setminus \Q \longmapsto I_\theta \in \mathrm{Prim}(\A)
\end{equation*}
is a homeomorphism onto its image when $(0,1)\setminus \Q$ is equipped with the topology induced by the usual topology on $\R$ and $\mathrm{Prim}(\A)$ is equipped with either the Jacobson topology, Fell topology, or the metric topology of $\mathsf{m}_{i(\mathcal{U}_\F)}$ of Corollary (\ref{AF-ideal-metric-thm}).
\end{proposition}
\begin{proof}
By Proposition (\ref{f-ideal-homeo}), the fact that the Jacobson topology of a separable C*-algebra is second countable (see \cite[Corollary 4.3.4]{Pedersen}), and the Fell topology of an AF algebra is metrizable (see Theorem (\ref{AF-ideal-metric-thm}) or more generally Lemma (\ref{fell-metrize-lemma})), sequential continuity suffices and thus we only need to to verify that the map defined in this proposition is a well-defined bijection onto its image.  However, it is well-defined by Definition (\ref{f-ideals}).  Thus, injectivity remains, which will follow from the next claim. 
\begin{claim}
If $\theta \in (0,1)\setminus \Q$, then:
\begin{equation*}
\lim_{n \to \infty} r(n,j_n(\theta))=\theta,
\end{equation*}
where for all $n \in \N \setminus \{0\}$, the quantity $r(n,j_n(\theta))$ is defined in Relations (\ref{farey-relations}) and Definition (\ref{f-ideals}).
\end{claim}
\begin{proof}[Proof of claim]
Fix $\theta \in (0,1) \setminus \Q$.  Let $\left(\frac{p_n^\theta}{q_n^\theta}\right)_{n \in \N}$ denote the standard rational approximations of $\theta$ that converge to $\theta$ from Expression (\ref{pq-rel-eq}).  Now, by the proofs of \cite[Proposition 4.i and Lemma 11]{Boca08}, there exists an increasing sequence $(k_n)_{n \in \N} \subseteq \N \setminus \{0\}$ such that the ordered  pair: 
\begin{equation}\label{farey-cfrac-eq}\left(r(k_n, j_{k_n}(\theta)),r(k_n, j_{k_n}(\theta)+1)\right)\in \left\{ \left(\frac{p_n^\theta}{q_n^\theta},\frac{p_{n-1}^\theta}{q_{n-1}^\theta}\right), \left( \frac{p_{n-1}^\theta}{q_{n-1}^\theta},\frac{p_{n}^\theta}{q_{n}^\theta}\right)\right\} \text{ for all } n \in \N \setminus \{0\}.
\end{equation}
Next, fix $n \in \N \setminus \{0\}$.  Consider $r(n,j_n (\theta))$.  By Lemma (\ref{f-ideal-vert}), first assume that $j_{n+1}(\theta)=2j_n(\theta)$.  Then, we have:
\begin{equation*}
r(n+1,j_{n+1}(\theta))=\frac{p(n+1,2j_n (\theta))}{q(n+1,2j_n(\theta))} =r(n,j_n(\theta))
\end{equation*}
by Relations (\ref{farey-relations}). Also, we have:
\begin{equation*}
\begin{split}
r(n+1, j_{n+1}(\theta)+1)&=\frac{p(n+1,2j_n (\theta)+1)}{q(n+1,2j_n (\theta)+1)}\\
&=\frac{p(n,j_n(\theta))+p(n,j_n(\theta)+1)}{p(n,j_n(\theta))+p(n,j_n(\theta)+1)}\leq r(n,j_{n}(\theta)+1)
\end{split}
\end{equation*}
by Relations (\ref{farey-relations}) and the fact that $p(n,j_n(\theta)+1)q(n,j_n(\theta))-p(n,j_n(\theta))q(n,j_n(\theta)+1)=1> 0$ from \cite[Section 1]{Boca08}. For the case $j_{n+1}(\theta)=2j_n (\theta)+1$, a similar argument shows that $r(n+1, j_{n+1}(\theta))\geq r(n,j_n(\theta))$ and $r(n+1, j_{n+1}(\theta)+1)=r(n,j_n(\theta)+1)$.  Hence, for all $n \in \N\setminus \{0\}$, we gather that:
\begin{equation}\label{shrink-part-eq}
r(n+1, j_{n+1}(\theta)+1)-r(n+1,j_{n+1}(\theta)) \leq r(n,j_n(\theta)+1)-r(n,j_n(\theta)).
\end{equation}
For all $n \in \N \setminus \{0\}$ such that $n \geq k_1$, define $N_n =\max \{k_m : k_m \leq n\}$. Note that since $(k_n)_{n \in \N}$ is increasing, we have that $(N_n)_{n \geq k_1}$ is non-decreasing and  $\lim_{n \to \infty} N_n = \infty$.     Now, fix $n \in \N \setminus \{0\}$ such that $n \geq k_1$, combining Expression (\ref{farey-cfrac-eq}) and (\ref{shrink-part-eq}), we have by  Definition (\ref{f-ideals}):
\begin{align*}
0< \theta -r(n,j_n(\theta))&<r(n,j_n(\theta)+1)-r(n,j_n(\theta)) \\
& \leq r(N_n, j_{N_n}(\theta)+1)-r(N_n, j_{N_n}(\theta))=\left\vert \frac{p^\theta_{N_n}}{q^\theta_{N_n}} -  \frac{p^\theta_{N_n-1}}{q^\theta_{N_n-1}} \right\vert, 
\end{align*}
and therefore $\lim_{n \to \infty} r(n,j_n (\theta))=\theta$ since $\lim_{n \to  \infty} \frac{p^\theta_{n}}{q^\theta_{n}}=\theta$ and $(N_n)_{n \geq k_1}$ is non-decreasing with  $\lim_{n \to \infty} N_n = \infty.$
\end{proof}
Next, let  $\theta, \eta \in (0,1) \setminus \Q$.  Assume that $I_\theta = I_\eta$ and thus their diagrams agree  \cite[Theorem 3.3]{Bratteli72}.  Hence, we have that $j_n(\theta)=j_n(\eta)$ for all $n \in \N $, and thus  $r(n,j_n(\theta))=r(n,j_n(\eta))$ for all $n \in \N \setminus \{0\}$.  Therefore, by the claim:
\begin{equation*}
\theta=\lim_{n \to \infty} r(n, j_n(\theta))=\lim_{n \to \infty} r(n,j_n (\eta))=\eta,
\end{equation*}
which completes the proof.
\end{proof}
\begin{remark}\label{explicit-ideal-metric}An immediate consequence of Proposition (\ref{f-j-irr-ideal-homeo}) is that  if: $(0,1)\setminus \Q$ is equipped with its relative topology from the usual topology on $\R$, the set $\{I_\theta \in \mathrm{Prim}(\A) : \theta \in (0,1)\setminus \Q\}$ is equipped with its relative topology induced by the Jacobson topology,  and the set $\{I_\theta \in \mathrm{Prim}(\A) : \theta \in (0,1)\setminus \Q\}$ is equipped with its relative topology induced by the metric topology of $\mathsf{m}_{i(\mathcal{U}_\F)}$ of Corollary (\ref{ind-lim-metric-cor}) or the Fell topology of Definition (\ref{Fell-topology}), then all these spaces are homeomorphic to the Baire space $\BaireSpace$ with its metric topology from Definition (\ref{Baire-Space-def}).  In particular, the totally bounded metric $\mathsf{m}_{i \left(\mathcal{U}_\F \right)}$ topology on the set of ideals  $\left\{I_\theta \in \mathrm{Prim}(\A) :  \theta \in (0,1)\setminus \Q\right\}$ is homeomorphic to $ (0,1)\setminus \Q$ with its totally  bounded  metric topology inherited from the usual topology on $\R$.  Hence, in some sense, the metric $\mathsf{m}_{i \left(\mathcal{U}_\F \right)}$ topology shares more metric information with $ (0,1)\setminus \Q$ and its metric than the Baire space metric topology as the Baire space complete and not totally bounded \cite[Theorem 6.5]{AL} (since the Baire space is complete, if it were totally bounded, then it would be compact, which would therefore contradict the fact that is is homeomorphic to the irrationals). This can also be displayed in metric calculations.

 Indeed, consider $\theta, \mu \in (0,1)\setminus \Q$ with continued fraction expansions $\theta=[a_j]_{j \in \N}$ and  $\mu=[b_j]_{j \in \N }$, in which $a_0=0, a_1=1000, a_j=1 \forall j \geq 2$ and $b_0, b_1 = 1, b_j=1 \forall j \geq 2$, and thus $\theta \approx 0.001, \mu \approx 0.618, \vert \theta -\mu\vert \approx 0.617.$  In the Baire metric $\mathsf{d}(\mathsf{cf}(\theta), \mathsf{cf}(\mu)) = 0.5 $, and, in the ideal metric $\mathsf{m}_{i\left(\mathcal{U}_\F\right)}(I_\theta , I_\mu)=0.25 $ since at level $n=1$ the diagram for $\F/I_\theta$ begins with $L_{999}$ and for $\F/I_\mu$ begins with $R_{b_2}$ by \cite[Proposition 4.i]{Boca08}, so the ideal diagrams differ first at $n=2$. Now, assume that for $\mu$ we have instead $b_1= 999,b_j=1 \forall j \geq 2$, and thus $\vert \theta - \mu \vert \approx 0.000000998$, but in the Baire metric, we still have that $\mathsf{d}(\mathsf{cf}(\theta), \mathsf{cf}(\mu)) = 0.5 $, while $\mathsf{m}_{i\left(\mathcal{U}_\F\right)}(I_\theta , I_\mu)=2^{-1000} $  since at level $n=1$ the diagram for $\F/I_\theta$ begins with $L_{999}$ and for $\F/I_\mu$ begins with $L_{998}$ and then transitions to $R_{b_2}$  by \cite[Proposition 4.i]{Boca08}, so the ideal diagrams differ first at $n=1000$.  In conclusion, in this example, the absolute value metric $\vert \cdot \vert$ behaves much more like the metric $\mathsf{m}_{i\left(\mathcal{U}_\F\right)}$  than the Baire metric.
\end{remark}

Fix $\theta \in (0,1) \setminus \Q $, we present a *-isomorphism from $\F/I_\theta$  to $\af{\theta}$ (Notation (\ref{af-theta-notation})) as a proposition to highlight a useful property for our purposes.  Of course, \cite[Proposition 4.i]{Boca08} already established that $\F/I_\theta$  and $\af{\theta}$ are *-isomorphic, but here we simply provide an explicit detail of such a *-isomorphism, which will serve us in the results pertaining to tracial states in Lemma (\ref{trace-level-1}).

\begin{proposition}\label{f-quotient-iso}
If $\theta \in (0,1) \setminus \Q $ with continued fraction expansion $\theta=[a_j]_{j \in \N}$ as in Expression (\ref{continued-fraction-eq}), then using Notation (\ref{af-theta-notation}) and Definition (\ref{f-ideals}), there exists a *-isomorphism $\mathfrak{af}_\theta : \F/I_\theta \rightarrow \af{\theta}$ such that 
if $x = x_0 \oplus \cdots \oplus x_{2^{a_1-1}} \in \F_{a_1}$, then:
\begin{equation*}
\mathfrak{af}_\theta \left( \indmor{\varphi}{a_1}(x) + I_\theta\right) = \indmor{\alpha_\theta}{1} \left(x_{j_{a_1}(\theta)+1} \oplus x_{j_{a_1}(\theta)} \right) \in  \indmor{\alpha_\theta}{1}\left(\af{\theta,1}\right).
\end{equation*}
\end{proposition}
\begin{proof}
By the proof of \cite[Proposition 4.i]{Boca08}, the Bratteli diagram of $\F/I_\theta$ begins with the diagram $L_{a_1 -1}$ of \cite[Figure 5]{Boca08} at level $n=1$.  Now, the diagram $C_a \circ C_b$ of \cite[Figure 6]{Boca08} is a section of the diagram of Example (\ref{af-theta-diagram}), in which the left column of $C_{a_1-1} \circ C_{a_2}$ is the bottom row of the first two levels from left to right after level $n=0$ of Example (\ref{af-theta-diagram}).  Therefore, by the placement of $\circledast$ at level $a_1$ in \cite[Figure 6]{Boca08}, define a map $f : (\F^{a_1}+ I_\theta)/I_\theta  \rightarrow \indmor{\alpha_\theta}{1}\left(\af{\theta,1}\right)$ by:
\begin{equation*}f: 
\left( \indmor{\varphi}{a_1}(x) + I_\theta\right) \mapsto \indmor{\alpha_\theta}{1} \left(x_{j_{a_1}(\theta)+1} \oplus x_{j_{a_1}(\theta)} \right),
\end{equation*}
where $x = x_0 \oplus \cdots \oplus x_{2^{a_1-1}} \in \F_{a_1}$. We show that $f$ is a *-isomorphism from  $(\F^{a_1}+ I_\theta)/I_\theta$ onto $ \indmor{\alpha_\theta}{1}(\af{\theta,1}).$

 We first show that $f$ is well-defined.  Let $c,e \in (\F^{a_1}+ I_\theta)/I_\theta$ such that $c=e$.  Now, we have $c=\indmor{\varphi}{a_1}(c') + I_\theta, e=\indmor{\varphi}{a_1}(e') + I_\theta$ where  $c'=c_0' \oplus \cdots \oplus c_{2^{a_1-1}}'\in \F_{a_1}$ and $  e'=e_0' \oplus \cdots \oplus e_{2^{a_1-1}}' \in \F_{a_1}$.  But, the assumption $c=e$ implies that  $\indmor{\varphi}{a_1}(c'-e') \in I_\theta \cap \F^{a_1} $.  Thus, by Definition (\ref{f-ideals}) of $I_\theta$, we have that  $c_{j_{a_1}(\theta)+1}' \oplus c_{j_{a_1}(\theta)}' =e_{j_{a_1}(\theta)+1}' \oplus e_{j_{a_1}(\theta)}'$, and since $j_{a_1}(\theta)=q^\theta_0$ and $j_{a_1}(\theta)+1=q^\theta_1 $ by \cite[Proposition 4.i]{Boca08} and the discussion at the start of the proof, we gather that $f$ is a well-defined  *-homomorphism since the canonical maps  $\indmor{\alpha_\theta}{1}$ and $\indmor{\varphi}{a_1}$ are *-homomorphisms.  

For surjectivity of $f$, let  $x= \indmor{\alpha_\theta}{1} \left(x_{q^\theta_1} \oplus x_{q^\theta_0} \right)$, where $x_{q^\theta_1} \oplus x_{q^\theta_0}  \in \af{\theta,1}$.  Define $y= y_0 \oplus \cdots y_{2^{a_1-1}} \in \F_{a_1}$ such that $y_{j_{a_1}(\theta)}=x_{q^\theta_0}$ and $ y_{j_{a_1}(\theta)+1}=x_{q^\theta_1}$ with $y_k = 0$ for all $k \in \{0, \ldots, 2^{a_1-1}\}\setminus \{j_{a_1}(\theta), j_{a_1}(\theta)+1\}.$  Hence, the image $f\left( \indmor{\varphi}{a_1}(y) + I_\theta\right)=x$.  

For injectivity of $f$, let $x=x_0 \oplus \cdots \oplus x_{2^{a_1-1}} \in \F_{a_1}$ and  $y=y_0 \oplus \cdots \oplus y_{2^{a_1-1}} \in \F_{a_1}$ such that $f\left( \indmor{\varphi}{a_1}(x) + I_\theta\right)=f\left( \indmor{\varphi}{a_1}(y) + I_\theta\right).$  Thus, since $ \indmor{\alpha_\theta}{1}$ is injective, we have that $x_{j_{a_1}(\theta)+1} \oplus x_{j_{a_1}(\theta)} =y_{j_{a_1}(\theta)+1} \oplus y_{j_{a_1}(\theta)}.$  But, this then implies that  $\indmor{\varphi}{a_1}(x-y) \in I_\theta \cap \F^{a_1} \subseteq I_\theta$ by Definition (\ref{f-ideals}), and therefore, the terms  $\indmor{\varphi}{a_1}(x)+I_\theta = \indmor{\varphi}{a_1}(y)+I_\theta$, which completes the argument that $f$ is a *-isomorphism from  $(\F^{a_1}+ I_\theta)/I_\theta$ onto $ \indmor{\alpha_\theta}{1}(\af{\theta,1}).$  

Lastly, using Definition (\ref{diagram-af}), consider the Bratteli diagram of $\F/I_\theta$ given by the sequence of unital C*-subalgebras $\left( (\F^{x_{j+1}}+I_\theta)/I_\theta\right)_{j \in \N},$ where $x_{j+1}=\sum_{k=1}^{j+1} a_k$ for all $j \in \N$. Hence, the proof of  \cite[Proposition 4.i]{Boca08} and  \cite[Figure 6]{Boca08} provide that this diagram of  $\F/I_\theta$ is equivalent to the Bratteli diagram of $\af{\theta}$ beginning at $\af{\theta,1}$ given by Example (\ref{af-theta-diagram}), where this equivalence of Bratteli diagrams is given by \cite[Section 23.3 and Theorem 23.3.7]{Bhat}. Therefore,  combining the equivalence relation of \cite[Section 23.3 and Theorem 23.3.7]{Bhat} and the construction of the *-isomorphism in \cite[Proposition III.2.7]{Davidson}, we conclude that there exists a *-isomorphism $\mathfrak{af}_\theta : \F/I_\theta \rightarrow \af{\theta}$ such that $\mathfrak{af}_\theta (z)=f(z)$ for all $z \in  (\F^{a_1}+ I_\theta)/I_\theta ,$ which completes the proof.   
\end{proof}

From the *-isomorphism of Proposition (\ref{f-quotient-iso}), we may provide a faithful tracial state for the quotient $\F/I_\theta$ from the unique faithful tracial state of $\af{\theta}$. Indeed:

\begin{notation}\label{f-quotient-trace}
Fix $\theta \in (0,1) \setminus \Q $.  There is a unique faithful tracial state on $\af{\theta}$ denoted $\sigma_\theta$ (see \cite[Lemma 5.3 and Lemma 5.5]{AL}).  Thus, 
\begin{equation*}
\tau_\theta=\sigma_\theta \circ \mathfrak{af}_\theta
\end{equation*} is a unique faithful tracial state on $\F/I_\theta$ with $\mathfrak{af}_\theta$ from Proposition (\ref{f-quotient-iso}).

Let $Q_\theta : \F \rightarrow \F/I_\theta $ denote the quotient map. Thus, by \cite[Theorem V.2.2]{Conway}, there exists a unique linear functional on $\F$ denoted, $\rho_\theta$, such that $\ker \rho_\theta \supseteq I_\theta$ and $\tau_\theta \circ Q_\theta (x)=\rho_\theta (x)$ for all $x \in \F$.  Since $\tau_\theta $ is a tracial state and:
\begin{equation*}
\tau_\theta \circ Q_\theta (x)=\rho_\theta (x)
\end{equation*}
for all $x \in \F$, we conclude that  $\rho_\theta$ is also a tracial state that vanishes on $I_\theta$.  Furthermore, $\rho_\theta$ is faithful on $\F \setminus I_\theta $ since $\tau_\theta$ is faithful on $\F/I_\theta .$
\end{notation}

One more ingredient remains before we define the quantum metric structure for the quotient spaces $\F/I_\theta$.

\begin{lemma}\label{f-beta}
Let $\theta \in (0,1) \setminus \Q $. Using notation from Definition (\ref{boca-mundici-af}) and Definition (\ref{f-ideals}), if we define:
\begin{equation*}
\beta^{\theta}: n \in \N \longmapsto \frac{1}{\dim((\F^n +I_\theta)/I_\theta)} \in (0, \infty),
\end{equation*}
then $\beta^\theta (n) =  \frac{1}{q(n,j_n(\theta))^2 + q(n, j_n (\theta)+1)^2} \leq \frac{1}{n^2}$ for all $n \in \N\setminus \{0\}$ and $\beta^\theta (0)=1 $. 
\end{lemma}
\begin{proof}
First, the quotient $(\F^0 + I_\theta)/I_\theta = \C1_{\F/I_\theta}$.  Hence, the term $\beta^\theta (0)=1$.  

Fix $n \in \N \setminus \{0\}$.  Since $(\F^n +I_\theta)/I_\theta$ is *-isomorphic to $\F^n/(I_\theta \cap \F^n)$ (see Proposition (\ref{ind-lim-quotients-prop})), we have that
\begin{equation*}\dim((\F^n +I_\theta)/I_\theta)=\dim(\F^n/(I_\theta \cap \F^n)) =q(n,j_n(\theta))^2 + q(n, j_n (\theta)+1)^2 
\end{equation*} by Definition (\ref{f-ideals}) and the dimension of the quotient is the difference in dimensions of $\F^n$ and $I_\theta \cap \F^n$ . Therefore, the term  $\beta^\theta(n) = \frac{1}{ q(n,j_n(\theta))^2 + q(n, j_n (\theta)+1)^2 }.$

To prove the inequality of the Lemma, we claim that for all $n \in \N \setminus \{0\}$, we have $q(n,j_n(\theta)) \geq n$ or $q(n, j_n (\theta)+1) \geq n$. We proceed by induction.  If $n=1$, then $q(1,j_1(\theta))=1$ and $q(1, j_1 (\theta)+1) =1$ by Relations (\ref{farey-relations}). Next assume the statement of the claim is true for $n=m$. Thus, we have that $q(m,j_m(\theta)) \geq m$ or $q(m, j_m (\theta)+1) \geq m$.  First, assume that $q(m,j_m(\theta)) \geq m$.  By Lemma (\ref{f-ideal-vert}), assume that $j_{m+1}(\theta)= 2j_m (\theta) $.  Thus, we gather $q(m+1, j_{m+1}(\theta)+1)=q(m+1, 2j_m (\theta)+1)= q(m,j_m(\theta))+q(m,j_m(\theta)+1) \geq m+1$ by Relations (\ref{farey-relations}) and since  $q(m,j_m(\theta)+1) \in \N \setminus \{0\}$.   The case when $j_{m+1}(\theta)= 2j_m (\theta) +1$ follows similarly as well as the case when $q(m, j_m (\theta)+1) \geq m$, which completes the induction argument .  

In particular, for all $n \in \N \setminus \{0\}$, we have $q(n,j_n(\theta)) \geq n$ or $q(n, j_n (\theta)+1) \geq n,$ which  implies that $q(n,j_n(\theta))^2 \geq n^2$ or $q(n, j_n (\theta)+1)^2 \geq n^2$. And thus, the term:
\begin{equation*} 
\frac{1}{q(n,j_n(\theta))^2 + q(n, j_n (\theta)+1)^2} \leq \frac{1}{n^2}
\end{equation*} for all $n \in \N\setminus \{0\}.$
\end{proof}

Hence, we have all the ingredients to define the quotient quantum metric spaces of the ideals of Definition (\ref{f-ideals}).

\begin{notation}\label{f-quotient-qcms}
Fix $\theta \in (0,1) \setminus \Q $. Using  Definition (\ref{boca-mundici-af}), Definition (\ref{f-ideals}), Notation (\ref{f-quotient-trace}), and Lemma (\ref{f-beta}),  let: \begin{equation*}
\left(\F/I_\theta, \Lip_{\mathcal{U}_\F/I_\theta , \tau_\theta}^{\beta^\theta}\right)
\end{equation*}
denote the $(2,0)$-quasi-Leibniz quantum compact metric space given by Theorem (\ref{ind-lim-quantum-iso}) associated to the ideal $I_\theta$, faithful tracial state $\tau_\theta$, and $\beta^\theta : \N \rightarrow (0, \infty)$ having limit $0$ at infinity by Lemma (\ref{f-beta}).
\end{notation}

\begin{remark}\label{theta-q-iso}
Fix $\theta \in (0,1) \setminus \Q$.  Although $\F/I_\theta $ and $\af{\theta}$ are *-isomorphic, it is unlikely that $\left(\F/I_\theta, \Lip_{\mathcal{U}_\F/I_\theta , \tau_\theta}^{\beta^\theta}\right)$ is quantum isometric to $\left(\af{\theta}, \Lip_{\mathcal{I}_\theta ,\sigma_\theta}^{\beta_\theta} \right)$ of Theorem (\ref{af-theta-thm})  based on the Lip-norm constructions. Thus, one could not simply apply Proposition (\ref{f-ideal-homeo}) to Theorem (\ref{af-theta-thm}) to achieve Theorem (\ref{main-result}).
\end{remark}

In order to provide our continuity results, we describe the faithful tracial states on the quotients in sufficient detail through Lemma (\ref{f-ideal-trace-coeff}) and Lemma (\ref{trace-level-1}).

\begin{lemma}\label{f-ideal-trace-coeff}Fix $\theta \in (0,1) \setminus \Q $. Let $\mathsf{tr}_d$ be the unique tracial state of $\M(d)$. Using notation from Definitions (\ref{boca-mundici-af}, \ref{f-ideals}), if  $n \in \N \setminus \{0\}$ and $a= a_0 \oplus \cdots \oplus a_{2^{n-1}} \in \F_n$, then using  Notation (\ref{f-quotient-trace}):
\begin{equation*} \rho_\theta \circ \indmor{\varphi}{n} ( a ) =  c(n,\theta) \mathsf{tr}_{q(n,j_n (\theta))}\left(a_{j_n (\theta)}\right) + (1-c(n,\theta)) \mathsf{tr}_{q(n,j_n (\theta)+1)}\left(a_{j_n (\theta)+1}\right), 
\end{equation*} where $c(n, \theta)  \in (0,1) $ and $\rho_\theta \circ \indmor{\varphi}{0}(a)=a$ for all $a \in \F_0$.

Furthermore, let $n \in \N \setminus \{0\}$, then:
\begin{equation*}
c(n+1, \theta) =
\begin{cases}
\frac{(q(n,j_n(\theta))+q(n,j_n(\theta)+1))c(n, \theta)-q(n,j_n(\theta))}{q(n,j_n(\theta)+1)} &: \text{if } j_{n+1} (\theta) = 2j_n (\theta)\\
& \\
\left(1+ \frac{q(n,j_n(\theta)+1)}{q(n,j_n(\theta))}\right)c(n,\theta) &: \text{if }j_{n+1}(\theta)=2j_n (\theta)+1\end{cases}.
\end{equation*}
\end{lemma}
\begin{proof}
Fix $\theta \in (0,1) \setminus \Q $. If $n=0$, then $\rho_\theta \circ \indmor{\varphi}{0}(a)=a$ for all $a \in \F_0$ since $\F_0 =\C.$ Let $n \in \N \setminus \{0\}$ and $a= a_0 \oplus \cdots \oplus a_{2^{n-1}} \in \F_n$.   Now, $\rho_\theta $ is a tracial state on $\F$, and thus, the composition $\rho_\theta \circ \indmor{\varphi}{n}$ is a tracial state on $\F_n $.  Hence, by \cite[ Example IV.5.4]{Davidson}:  
\begin{equation*}
\rho_\theta \circ \indmor{\varphi}{n}(a) = \sum_{k=0}^{2^{n-1}} c_k \mathsf{tr}_{q(n,k)}(a_k), 
\end{equation*}
where $\sum_{k=0}^{2^{n-1}} c_k =1$ and $c_k \in [0,1]$ for all $k \in \{0, \ldots , 2^{n-1}\}$.  But, since $\rho_\theta$ vanishes on $I_\theta$ by definition of $\rho_\theta$ in Notation (\ref{f-quotient-trace}),  we conclude that $c_k = 0$ for all $k \in \{0, \ldots , 2^{n-1}\} \setminus \{j_n (\theta), j_n (\theta)+1 \}$.  Also,  the fact that $\rho_\theta $ is faithful on $\F \setminus I_\theta$ implies that $c_{j_n (\theta)}, c_{j_n (\theta)+1} \in (0,1)$ and $c_{j_n (\theta)}+ c_{j_n (\theta)+1}=1$.  Define $c(n, \theta) = c_{j_n (\theta)}$ and clearly $c_{j_n (\theta)+1}=1-c(n, \theta) $.  

Next, let $n \in \N \setminus \{0\}$ and let $j_{n+1} (\theta) = 2j_n (\theta)$.   Combining Lemma (\ref{f-ideal-vert}) and Proposition (\ref{f-diagram}), there is one edge from $(n,j_n(\theta))$ to $(n+1,j_{n+1}(\theta)) $ and one edge from $(n,j_n(\theta))$ to $(n+1, j_{n+1}(\theta) +1) $ with no other edges from $(n, j_n(\theta))$ to either $(n,j_n(\theta))$ or $(n+1, j_{n+1}(\theta) +1) $.  Also, there is one edge from $(n, j_n(\theta)+1)$ to $(n+1,j_{n+1}(\theta) +1 )$ with no other edges from $(n, j_n(\theta)+1)$ to either $(n,j_n(\theta))$ or $(n+1, j_{n+1}(\theta) +1) $. 

 Hence, consider an element $ a= a_0 \oplus \cdots \oplus a_{2^{n-1}} \in \F_n$ such that $a_k =0$ for all $k \in \{0, \ldots , 2^{n-1}\} \setminus \{j_n (\theta), j_n (\theta)+1 \}.$ Since the edges determine the partial multiplicities of $\varphi_n $, we have that $\varphi_n (a) = b_0 \oplus \cdots \oplus b_{2^n}$ such that
 \begin{equation}\label{embed-up} b_{j_{n+1}(\theta) }= Ua_{j_n(\theta)}U^* \text{ and } b_{j_{n+1}(\theta) +1} =V \left[\begin{array}{cc} a_{j_n(\theta)} & 0 \\
0 &  a_{j_n(\theta)+1}\end{array}\right]V^*,
\end{equation} for some unitaries $U \in \M(q(n+1,j_{n+1}(\theta) )), V \in \M(q(n+1,j_{n+1}(\theta) +1))$ by \cite[Lemma III.2.1]{Davidson}. Also, the terms $b_k = 0$ for all $k \in  \{0, \ldots , 2^{n-1}\} \setminus \{j_{n+1} (\theta), j_{n+1} (\theta)+1 \}$. But, by definition of the canonical *-homomorphisms $\indmor{\varphi}{n},\indmor{\varphi}{n+1}$, we have that $\indmor{\varphi}{n}(a) = \indmor{\varphi}{n+1}(\varphi_n (a) )$ \cite[Chapter 6.1]{Murphy}.  Now, assume that $a_{j_n (\theta)}=1_{\M(q(n,j_n(\theta)))}$ and $ a_{j_n (\theta)+1} = 0$.  Therefore, by Expression (\ref{embed-up}): 
\begin{equation}\label{coeff-match}
\end{equation}
\begin{align*}
c(n,\theta) &= \rho_\theta \circ \indmor{\varphi}{n}(a) = \rho_\theta \circ  \indmor{\varphi}{n+1}(\varphi_n (a) )\\&= c(n+1,\theta) \mathsf{tr}_{q(n+1,j_{n+1} (\theta))}\left(Ua_{j_n (\theta)}U^*\right)\\
& \quad  + (1-c(n+1,\theta)) \mathsf{tr}_{q(n+1,j_{n+1} (\theta)+1)}\left(V \left[\begin{array}{cc} a_{j_n(\theta)} & 0 \\
0 &  0\end{array}\right]V^*\right)\\
& =c(n+1, \theta)\cdot 1 + (1-c(n+1, \theta))\mathsf{tr}_{q(n+1,j_{n+1} (\theta)+1)}\left( \left[\begin{array}{cc} 1_{\M(q(n,j_n(\theta)))} & 0 \\
0 &  0\end{array}\right]\right)\\
& = c(n+1, \theta)+ (1-c(n+1, \theta))\frac{1}{q(n+1,j_{n+1} (\theta)+1)}q(n,j_n(\theta)).
\end{align*}
Thus, since $q(n+1, 2j_n (\theta)+1 )=q(n,j_n(\theta))+q(n,j_n(\theta)+1)$ from Relations (\ref{farey-relations}) and $j_{n+1}(\theta)+1 =2j_n (\theta)+1 $,  we conclude that: 
\begin{equation*}
c(n+1, \theta) = \frac{(q(n,j_n(\theta))+q(n,j_n(\theta)+1))c(n, \theta)-q(n,j_n(\theta))}{q(n,j_n(\theta)+1)}.
\end{equation*}

Lastly, assume that $j_{n+1}(\theta)= 2j_n (\theta) + 1$.  Let $ a= a_0 \oplus \cdots \oplus a_{2^{n-1}} \in \F_n$ such that $a_k =0$ for all $k \in \{0, \ldots , 2^{n-1}\} \setminus \{j_n (\theta), j_n (\theta)+1 \}$.  A similiar argument shows that $\varphi_n (a) = b_0 \oplus \cdots \oplus b_{2^n}$ such that:
\begin{equation*} b_{j_{n+1}(\theta) }= Y \left[\begin{array}{cc} a_{j_n(\theta)} & 0 \\
0 &  a_{j_n(\theta)+1}\end{array}\right]Y^* \text{ and } b_{j_{n+1}(\theta) +1} =Z  a_{j_n(\theta)+1}Z^*, \end{equation*} where $Y \in \M(q(n+1,j_{n+1}(\theta) )), Z \in \M(q(n+1,j_{n+1}(\theta) +1))$ are unitary.  Now, assume that $a_{j_n (\theta)}=1_{\M(q(n,j_n(\theta)))}$ and $ a_{j_n (\theta)+1} = 0$. Therefore, similarly to Expression (\ref{coeff-match}): 
\begin{equation*}
c(n, \theta) = c(n+1, \theta) \frac{1}{q(n+1,j_{n+1}(\theta))}q(n,j_n(\theta)), 
\end{equation*}
and therefore: 
\begin{equation*}
c(n+1, \theta) =\left(1+ \frac{q(n,j_n(\theta)+1)}{q(n,j_n(\theta))}\right)c(n,\theta) 
\end{equation*}
by Relations (\ref{farey-relations}). And, by Lemma (\ref{f-ideal-vert}), this exhausts all possibilities for $c(n+1, \theta)$, and the proof is complete.
\end{proof}

\begin{lemma}\label{trace-level-1}  Using notation from Lemma (\ref{f-ideal-trace-coeff}), if $\theta \in (0,1) \setminus \Q $, then: 
\begin{equation*}c(1, \theta) = 1-\theta.
\end{equation*}
Moreover, using notation from Definition (\ref{f-ideals}), if $\theta, \mu \in (0,1) \setminus \Q $ such that there exists $N \in \N \setminus \{0\}$ with $I_\theta \cap \F^N = I_\mu \cap \F^N$, then there exists $a, b \in \R, a \neq 0$ such that:
\begin{equation*} c(N, \theta)=a\theta +b, \ c(N, \mu)=a\mu +b. 
\end{equation*}  
\end{lemma}
\begin{proof}
Let $\theta \in (0,1) \setminus \Q $, and denote its continued fraction expansion by $\theta=[a_j]_{j \in \N}$.     Recall, by Proposition (\ref{f-quotient-iso}), we have for all $x = x_0 \oplus \cdots \oplus x_{2^{a_1-1}} \in \F_{a_1}$:
\begin{equation}\label{first-iso}
\mathfrak{af}_\theta \left( \indmor{\varphi}{a_1}(x) + I_\theta\right) = \indmor{\alpha_\theta}{1} \left(x_{j_{a_1}(\theta)+1} \oplus x_{j_{a_1}(\theta)} \right).
\end{equation}
Next, by Notation (\ref{f-quotient-trace}), we note that: \begin{equation}\label{trace-iso}
\rho_\sigma \circ \indmor{\varphi}{a_1}= \tau_\theta \circ Q_\theta  \circ \indmor{\varphi}{a_1}= \sigma_\theta \circ \mathfrak{af}_\theta  \circ Q_\theta  \circ \indmor{\varphi}{a_1}
\end{equation}  Now, consider $x= x_0 \oplus \cdots \oplus x_{2^{a_1-1}} \in \F_{a_1}$ such that $x_{j_{a_1}(\theta)+1} = 1_{q^\theta_1}$ and $x_k = 0$ for all $k \in \{0, \ldots , 2^{a_1-1}\} \setminus \{j_{a_1}(\theta)\}$. Then, by Lemma (\ref{f-ideal-trace-coeff}) and Expressions (\ref{first-iso},\ref{trace-iso}), we have that $(1-c(a_1, \theta))=\rho_\theta \circ \indmor{\varphi}{a_1}(x) = \sigma_\theta \circ \indmor{\alpha_\theta}{1} \left(1_{q^\theta_1} \oplus 0 \right)= a_1 \theta$ by \cite[Lemma 5.5]{AL}.  And, thus:
\begin{equation}\label{coeff-mid}
 c(a_1, \theta) = 1- a_1 \theta.
 \end{equation} Thus, if $a_1 =1$, then we would be done.

Assume that $a_1 \geq 2$. By the proof of \cite[Proposition 4.i]{Boca08}, the Bratteli diagram of $\F/I_\theta$ begins with the diagram $L_{a_1 -1}$ of \cite[Figure 5]{Boca08} at level $n=1$.  Thus, the term $j_m (\theta)= 0$ for all $m \in \{1, \ldots, a_1 \} $.  Hence, if $m \in \{1, \ldots, a_1 -1 \}$, then $j_{m+1}(\theta) = 2j_m(\theta) $.  

We claim that for all $m \in  \{1, \ldots, a_1 \}$ we have that: 
\begin{equation}\label{middle-coeff}c(m, \theta)= m c(1, \theta) - (m -1) .
\end{equation}  We proceed by induction.  The case $m=1$ is clear.  Assume true for $m \in \{1, \ldots, a_1 -1 \}$.  Consider $m+1 $.  Since $j_{m+1} (\theta)=2j_m(\theta)$, by Lemma (\ref{f-ideal-trace-coeff}), we have that: 
\begin{equation}\label{coeff-start}
\begin{split}c(m+1, \theta) &= \frac{(q(m,0)+q(m,1))c(m, \theta)-q(m,0)}{q(m,1)}\\
&=\frac{c(m,\theta)+q(m,1)c(m,\theta) - 1}{q(m,1)}.
\end{split}
\end{equation}   By Relations (\ref{farey-relations}), we gather that $q(m,1)=m$.  Hence, by induction hypothesis and Expression (\ref{coeff-start}), we have:
\begin{equation*}
\begin{split} c(m+1, \theta) &= \frac{m c(1, \theta)-(m-1)+m (mc(1, \theta) - (m-1)) -1}{m} \\
&= c(1, \theta) - 1+1/m +m c(1, \theta) -(m-1)-1/m \\
&= (m+1)c(1, \theta) - ((m+1)-1),
\end{split}
\end{equation*} which completes the induction argument. Hence, by Expression (\ref{middle-coeff}), we conclude $c(a_1, \theta)= a_1c(1, \theta) - (a_1 -1),$ which implies that:
\begin{equation}\label{first-coeff} 
c(1, \theta)=1-\theta 
\end{equation} by Equation (\ref{coeff-mid}).

Lastly,  let $\theta, \mu \in (0,1) \setminus \Q $.  We prove the remaining claim in the Lemma by induction.  Assume $N=1$.  Then, by Equation (\ref{first-coeff}), the coefficients  $c(1, \mu)=1-\mu$ and $c(1,\theta)=1-\theta$, which completes the base case. 

Assume true for $N \in \N \setminus \{0,1\}$. Assume that $I_\mu \cap \F^{N+1}=I_\theta \cap \F^{N+1}$. Now, since $\F^N \subseteq \F^{N+1}$, we thus have $I_\mu \cap \F^{N}=I_\theta \cap \F^{N}$.  Hence, by induction hyposthesis, there exists $a,b \in \R , a \neq 0$ such that $c(N, \mu)=a\mu+b$ and $c(N, \theta)=a\theta +b$. But, as   $I_\mu \cap \F^{N+1}=I_\theta \cap \F^{N+1}$, the vertices a level $N+1$ agree in the ideal diagrams by Proposition (\ref{ideal-corr}).  In particular, by Definition (\ref{f-ideals}), we have $j_{N+1}(\theta)=j_{N+1}(\mu)$, and similarly, the term $j_N (\theta)=j_N (\mu)$ by $I_\mu \cap \F^{N}=I_\theta \cap \F^{N}$.  The  conclusion follows by Lemma (\ref{f-ideal-trace-coeff}).
\end{proof}

We can now prove the main result of this section.

\begin{theorem}\label{main-result}
Using Definition (\ref{f-ideals}) and Notation (\ref{f-quotient-qcms}), the map:
\begin{equation*}
I_\theta \in \left(\mathrm{Prim}(\F), \tau \right) \longmapsto \left(\F/I_\theta, \Lip_{\mathcal{U}_\F/I_\theta , \tau_\theta}^{\beta^\theta}\right) \in (\mathcal{QQCMS}_{2,0}, \qpropinquity{})
\end{equation*}
is continuous to the class of $(2,0)$-quasi-Leibniz quantum compact metric spaces metrized by the quantum propinquity $\qpropinquity{}$, where $\tau$ is either the Jacobson topology, the relative metric topology of $\mathsf{m}_{i\left(\mathcal{U}_\F \right)}$  of Corollary (\ref{ind-lim-metric-cor}), or the relative Fell topology  of Definition (\ref{Fell-topology}).  
\end{theorem}
\begin{proof}
By Proposition (\ref{f-ideal-homeo}) and Proposition (\ref{f-j-irr-ideal-homeo}), we only need to show continuity with respect to the metric $\mathsf{m}_{i\left(\mathcal{U}_\F \right)}$ with sequential continuity.  Thus, let $(I_{\theta_n} )_{n \in \overline{\N}} \subset \mathrm{Prim}(\F)$ be a sequence such that $(I_{\theta_n} )_{n \in \N}$ converges to $I_{\theta_\infty}$ with respect to $\mathsf{m}_{i\left(\mathcal{U}_\F \right)}$.  Therefore, by Lemma (\ref{metric-fusing-equiv-lemma}),  this implies that: 
\begin{equation*}\left\{ I_{\theta_n} =\overline{\cup_{k \in \N} I_{\theta_n}  \cap \F^k}^{\Vert \cdot \Vert_\F} : n \in \overline{\N} \right\}
\end{equation*} is a fusing family with some fusing sequence $(c_n)_{n \in \N}$. Thus, condition (1) of Theorem (\ref{quotients-converge}) is satisfied.

For condition (2) of Theorem (\ref{quotients-converge}), let $N \in \N$, then by definition of fusing sequence, if $k \in \N_{\geq c_N}$, then $I_{\theta_k}  \cap \F^N=I_{\theta_\infty} \cap \F^N$.  Now, let $k \in \overline{\N}_{\geq c_N}$. Consider $\rho_{\theta_k}$  on $\F^N$. By Lemma (\ref{trace-level-1}),  there exists $a ,b \in \R, a \neq 0$,  such that   $c(N,\theta_k)=a \theta_k +b $  for all $k \in \overline{\N}_{\geq c_N}$.  But, by Proposition (\ref{f-ideal-homeo}), we obtain $(\theta_n)_{n \in \N}$ converges to $\theta_\infty $ with respect to the usual topology on $\R$.  Hence, the sequence $\left(c(N,\theta_k)\right)_{k \in \N_{\geq c_N}}$ converges to $c(N,\theta_\infty)$ with respect to the usual topology on $\R$ and the same applies to $(1-c(N,\theta_k))_{k \in  \N_{\geq c_N}}$.  However, by Lemma (\ref{f-ideal-trace-coeff}), the coefficient $c(N,\theta_k)$ determines $\rho_k$ for all $k \in \overline{\N}_{\geq c_N}$.  Hence, \cite[Lemma 3.3]{Aguilar16a} provides that $\left(\rho_{\theta_k}\right)_{k \in  \N_{\geq c_N}}$ converges to $ \rho_{\theta_\infty}$ in the weak-* topology on $\StateSpace\left(\F^N\right)$.

Condition (3) of Theorem (\ref{quotients-converge}) follows a similar argument as in the proof of condition (2) since the sequences $\beta^\theta$ of Lemma (\ref{f-beta}) are determined by the terms $j_n (\theta)$.  Also, all $\beta^\theta$ are uniformly bounded by the sequence $(1/n^2)_{n \in \N}$ which converges to $0$.  Therefore, the proof is complete.
\end{proof}

As an aside to Remark (\ref{theta-q-iso}), we obtain the following analogue to Theorem (\ref{af-theta-thm}) in terms of quotients.
\begin{corollary}
Using Notation (\ref{f-quotient-qcms}), the map: 
\begin{equation*}
\theta \in \left((0,1) \setminus \Q, \vert \cdot \vert \right) \longmapsto  \left(\F/I_\theta, \Lip_{\mathcal{U}_\F/I_\theta , \tau_\theta}^{\beta^\theta}\right) \in (\mathcal{QQCMS}_{2,0}, \qpropinquity{})
\end{equation*}
is continuous from $(0,1) \setminus \Q$, with its topology as a subset of $\R$ to the class of $(2,0)$-quasi-Leibniz quantum compact metric spaces metrized by the quantum propinquity $\qpropinquity{}$.
\end{corollary}
\begin{proof}
Apply  Proposition (\ref{f-j-irr-ideal-homeo}) to Theorem (\ref{main-result}).
\end{proof}

\section*{Acknowledgements}

I am grateful to my Ph.D. thesis advisor, Dr. Fr\'{e}d\'{e}ric Latr\'{e}moli\`{e}re, for all his beneficial comments and support.

\providecommand{\bysame}{\leavevmode\hbox to3em{\hrulefill}\thinspace}
\providecommand{\MR}{\relax\ifhmode\unskip\space\fi MR }
\providecommand{\MRhref}[2]{%
  \href{http://www.ams.org/mathscinet-getitem?mr=#1}{#2}
}
\providecommand{\href}[2]{#2}

\vfill

\end{document}